\documentclass[11pt]{amsart}

\usepackage[utf8]{inputenc}
\usepackage{fullpage}
\usepackage{amscd,amssymb,amsmath,amsthm}
\usepackage{enumerate}
\usepackage{mathtools}
\usepackage{shuffle}
\usepackage[colorlinks=true,citecolor=blue]{hyperref}
\usepackage{color}
\usepackage{appendix}

\usepackage{tikz-cd}
\tikzcdset{
arrow style=tikz,
diagrams={>={Straight Barb[scale=0.8]}}
}

\newtheorem{theorem}{Theorem}[section]
\newtheorem{proposition}[theorem]{Proposition}
\newtheorem{lemma}[theorem]{Lemma}
\newtheorem{corollary}[theorem]{Corollary}
\theoremstyle{definition}
\newtheorem{definition}[theorem]{Definition}
\newtheorem{example}[theorem]{Example}
\theoremstyle{remark}
\newtheorem{remark}[theorem]{Remark}

\DeclareMathOperator{\ad}{ad}
\DeclareMathOperator{\Aut}{Aut}

\DeclareMathOperator{\End}{End}
\DeclareMathOperator{\Ext}{Ext}

\DeclareMathOperator{\Hom}{Hom}
\DeclareMathOperator{\id}{id}
\DeclareMathOperator{\im}{im}
\let\Im\relax
\DeclareMathOperator{\Im}{Im}
\DeclareMathOperator{\Lie}{Lie}
\DeclareMathOperator{\pr}{pr}

\DeclareMathOperator{\Res}{Res}
\DeclareMathOperator{\SL}{SL}
\DeclareMathOperator{\Spec}{Spec}
\DeclareMathOperator{\UVIC}{UVIC}
\DeclareMathOperator{\Vect}{Vect}
\DeclareMathOperator{\VIC}{VIC}

\newcommand{\alg}{{\rm alg}}
\newcommand{\an}{{\rm an}}
\newcommand{\BL}{{\rm BL}}
\newcommand{\dR}{{\rm dR}}
\newcommand{\can}{{\rm can}}
\newcommand{\defeq}{\coloneqq}
\newcommand{\eqdef}{=\vcentcolon}

\newcommand{\KZ}{{\rm KZ}}
\newcommand{\KZB}{{\rm KZB}}
\newcommand{\To}{\longrightarrow}

\title{Elliptic KZB connections via universal vector extensions}
\author{Tiago J. Fonseca and Nils Matthes}

\date{\today}

\subjclass[2010]{}

\address{IMECC-UNICAMP, Departamento de matemática, Rua Sérgio Buarque de Holanda, 651, CEP 13083-859, Campinas, SP, Brazil}
\email{tfonseca@unicamp.br}

\address{Department of Mathematical Sciences, University of Copenhagen, Universitetsparken 5, 2100 Copenhagen Ø, Denmark}
\email{nils.oliver.matthes@gmail.com}

\begin{document}

\maketitle

\begin{abstract}
  Using the formalism of bar complexes and their relative versions, we give a new, purely algebraic, construction of the so-called \emph{universal elliptic KZB connection} in arbitrary level. We compute explicit analytic formulae, and we compare our results with previous approaches to elliptic KZB equations and multiple elliptic polylogarithms in the literature.
  
   Our approach is based on a number of results concerning logarithmic differential forms on universal vector extensions of elliptic curves. Let $S$ be a scheme of characteristic zero, $E \to S$ be an elliptic curve, $f:E^{\natural} \to S$ be its universal vector extension, and $\pi:E^{\natural} \to E$ be the natural projection. Given a finite subset of torsion sections $Z\subset E(S)$, we study the dg-algebra over $\mathcal{O}_S$ of relative logarithmic differentials $\mathcal{A} = f_*\Omega^{\bullet}_{E^{\natural}/S}(\log \pi^{-1}Z)$. In particular, we prove that the residue exact sequence in degree one splits canonically, and we derive the formality of $\mathcal{A}$. When $S$ is smooth over a field $k$ of characteristic zero, we also prove that sections of $\mathcal{A}^1$ admit canonical lifts to absolute logarithmic differentials in $f_*\Omega^1_{E^{\natural}/k}(\log \pi^{-1}Z)$, which extends a well known property for regular differentials given by the `crystalline nature' of universal vector extensions.
  \end{abstract}

\setcounter{tocdepth}{1}
\tableofcontents

\section{Introduction}

The main goal of this paper is to give a purely algebraic construction of the so-called universal elliptic Knizhnik--Zamolodchikov--Bernard (KZB) connection in arbitrary level \cite{bernard98, LR07, CEE09, CG20} (cf. \cite{hain20,luo19,hopper21}) in terms of universal vector extensions of elliptic curves. In doing so, we establish a number of new results concerning logarithmic differential forms on universal vector extensions. As an application, we shall also obtain new algebraic formulas for elliptic KZB connections, which in particular yield an explicit solution to some rationality questions concerning these equations.

Our principal motivation is to provide an algebraic approach to multiple elliptic polylogarithms \cite{BL94,levin97,LR07,BL11} and their closely related notions, such as elliptic multiple zeta values \cite{enriquez16}. In the literature, these objects are usually defined and studied using analytic versions of the elliptic KZB connection, and algebraicity is only shown a posteriori. This obscures the essentially algebraic nature of the elliptic KZB connection, and makes the relation to arithmetic algebraic geometry more indirect --- e.g., it is not clear how to write special values of multiple elliptic polylogarithms in terms of periods, in the sense of Kontsevich--Zagier (cf. \cite{FM20}). In this work, we use the universal vector extension of an elliptic curve to give a purely algebraic definition of the elliptic KZB connection and we also show how to retrieve the various versions found in the literature via `analytification'. Our theory is in complete analogy to the genus zero case, the Knizhnik--Zamolodchikov (KZ) connection \cite{KZ84}, which is most naturally defined using algebraic formulas.

\subsection{The elliptic KZB connection over $\mathbb{C}$}\label{par:intro-KZB}

The elliptic KZB connection on a complex elliptic curve $(E,O)$ can be defined as a pro-algebraic connection with logarithmic singularities at $O$
\[
  \nabla_E : \mathcal{V}_E \longrightarrow \Omega^1_{E}(\log O) \hat{\otimes} \mathcal{V}_E 
\]
satisfying the following universal property: given a base point $b \in E\setminus O$, there is a vector $v_E$ in the fibre $\mathcal{V}_{E}(b)$ such that, for every \emph{unipotent} connection $\nabla : \mathcal{V} \to \Omega^1_{E}(\log O)\otimes \mathcal{V}$ equipped with $v \in \mathcal{V}(b)$, there is a unique morphism $(\mathcal{V}_E,\nabla_E) \to (\mathcal{V},\nabla)$ sending $v_E$ to $v$ (cf. \cite[\S 1]{Kim09}). Recall that `unipotent' means that $(\mathcal{V},\nabla)$ can be written as a finite iterated extension of the trivial connection $(\mathcal{O}_E,d)$.

Alternatively, by Serre's GAGA and the Riemann--Hilbert correspondence, $(\mathcal{V}_E,\nabla_E)$ can be characterised by its pro-local system of horizontal sections $\mathbb{V}_E$, whose stalk at $x \in E\setminus O$ is the pro-unipotent completion over $\mathbb{C}$ of the fundamental torsor of paths $\pi_1(E\setminus O;b,x)$:
\[
  \mathbb{V}_{E,x} = \pi^{\rm un}_1(E\setminus O;b,x).
\]
Concretely, local sections of $\mathbb{V}_E$ are described by holomorphic functions, possibly multi-valued, given by homotopy-invariant linear combinations of iterated integrals à la Chen
\[
  x \longmapsto \int_b^x \omega_1\cdots \omega_n
\]
of 1-forms $\omega_i$ on the once-punctured elliptic curve $E\setminus O$. Thus, the elliptic KZB connection can be thought more concretely as the differential equations these iterated integrals satisfy. Note that the equations themselves do not depend on the choice of base point $b$.

The above description in terms of local systems immediately generalises to a family of elliptic curves
\[
  \begin{tikzcd}
    E \arrow{r} & S \arrow[bend right=40]{l}[swap]{O},
  \end{tikzcd}
\]
where $S$ is a complex manifold. Here, working locally over $S$, one can take $b$ to be a base section of $E\to S$ and consider the local system $\mathbb{V}_E$ whose stalk at a point $x \in E\setminus O$ above $s\in S$ is
\[
  \mathbb{V}_{E,x} = \pi^{\rm un}_1(E_s\setminus O(s); b(s),x).
\]
In particular, the elliptic KZB connection of the family is an integrable connection  $(\mathcal{V}_E,\nabla_E)$ on the total space of the family minus the identity section $E\setminus O$, which restricts to the previously defined $(\mathcal{V}_{E_s},\nabla_{E_s})$ on every fibre $E_s \setminus O(s)$. In other words, it can be regarded as an \emph{isomonodromic deformation}, parametrised by $S$, of the elliptic KZB connection at one fibre. It is also possible to characterise $(\mathcal{V}_E,\nabla_E)$ directly by a relative version of the universal property recalled in the first paragraph above (cf. \cite[Section 3]{CdPS19}).

The (level 1) \emph{universal} elliptic KZB connection corresponds to the universal family $\mathcal{E} \to \mathcal{M}_{1,1}$. Higher level elliptic KZB connections are defined analogously, with logarithmic singularities on torsion points $E[N]$, see \cite{CG20,hopper21}. For the purposes of this introduction, we focus on the level 1 case, although all of our results work more generally in arbitrary level. 

\subsection{Construction of KZB over the universal vector extension}\label{par:intro-construction-kzb}

We state some our results in simplified form. Let $S$ be a scheme of characteristic zero, $(E/S,O)$ be an elliptic curve over $S$, and
\[
  \pi: E^{\natural} \longrightarrow E
\]
be its universal vector extension. Formally, $E^{\natural}$ is given as an extension of $E$ by a certain vector group of rank 1, in the category of commutative $S$-group schemes, satisfying a suitable universal property (see \S\ref{par:def-uve} below for a precise definition), and $\pi$ is the natural projection. In this paper, the key property of $\pi:E^{\natural}\to E$ is that it is a principal $\mathbb{G}_a$-bundle over which every $S$-unipotent vector bundle trivialises.

We shall directly construct a connection on $E^{\natural}$ with logarithmic singularities along the vertical divisor $\pi^{-1}O$ which can be shown a posteriori to be the pullback of the elliptic KZB connection on $E$ by $\pi$. Our first result describes global relative differential forms on $E^{\natural}/S$ with logarithmic singularities along $\pi^{-1}O$. For simplicity, assume that
\[
  S= \Spec R
\]
is affine and small enough so that $\pi^{-1}O \cong \mathbb{G}_{a,S} = \Spec R[t]$. 

\begin{theorem}\label{thm:intro-log-forms}
  There exists a $\nu \in \Gamma(E^{\natural}, \Omega^1_{E^{\natural}/S})$ such that $\nu|_{\pi^{-1}O} = dt$. Given such a $\nu$, there is a unique family $(\omega^{(n)})_{n\ge 0}$ in $\Gamma(E^{\natural}, \Omega^1_{E^{\natural}/S}(\log \pi^{-1}O))$ such that
  \[
    \Gamma(E^{\natural}, \Omega^1_{E^{\natural}/S}) = R \nu \oplus R \omega^{(0)}
  \]
  and, for $n\ge 1$,
  \begin{itemize}
    \item[(a)] $\Res(\omega^{(n)}) = t^{n-1}/(n-1)!$,
    \item[(b)] $\omega^{(n)}\wedge \omega^{(0)} = 0$,
    \item[(c)] $d\omega^{(n)} = \nu \wedge \omega^{(n-1)}$.
\end{itemize}
Moreover,
\[
  \Gamma(E^{\natural}, \Omega^1_{E^{\natural}/S}(\log \pi^{-1}O)) = R \nu \oplus \bigoplus_{n\ge 0} R \omega^{(n)}
\]
and
\[
  \Gamma(E^{\natural}, \Omega^2_{E^{\natural}/S}(\log \pi^{-1}O)) = \bigoplus_{n\ge 0}R\nu \wedge \omega^{(n)}.
\]
\end{theorem}

We call $\omega^{(0)},\omega^{(1)},\ldots$ \emph{Kronecker differentials}, as they are purely algebraic variants of classical elliptic functions considered by Kronecker (see \S\ref{intro:par-analytic} below).

Under the above hypotheses, we can explicitly construct a \emph{relative KZB connection} on $E^{\natural}/S$ by setting
\[
  \nabla_{E^{\natural}/S} : \mathcal{O}_{E^{\natural}}\hat{\otimes} R\langle \! \langle a,b \rangle\! \rangle \longrightarrow \Omega^1_{E^{\natural}/S}(\log \pi^{-1}O) \hat{\otimes} R\langle \! \langle a,b \rangle\!\rangle \text{, } \qquad \nabla_{E^{\natural}/S} =  d +\omega_{E^{\natural}/S}
\]
with
\[
  \omega_{E^{\natural}/S} = -\nu \otimes a  - \sum_{n\ge 0}\omega^{(n)}\otimes \ad_a^nb.
\]
Here, $R\langle \! \langle a, b\rangle \! \rangle$ denotes a ring of non-commutative power series with the $( a,b)$-adic topology, and $\ad_a$ is the operator $x\mapsto ax-xa$. In the above formula for $\omega_{E^{\natural}/S}$, an element of $R\langle \! \langle a,b\rangle \!\rangle$ acts  on $R\langle \! \langle a,b\rangle \!\rangle$ by left multiplication. The integrability of $\nabla_{E^{\natural}/S}$, which amounts to the equation
\[
  d\omega_{E^{\natural}/S} + \omega_{E^{\natural}/S}\wedge \omega_{E^{\natural}/S} = 0,
\]
follows from (b) and (c) in Theorem \ref{thm:intro-log-forms}, together with the fact that $\nu$ and $\omega^{(0)}$ are closed 1-forms (Proposition \ref{prop:closed}).

\begin{remark}
  The above explicit formula for the relative elliptic KZB connection is actually derived from a natural construction involving the bar complex of the dg-algebra $\Gamma(E^{\natural}, \Omega^{\bullet}_{E^{\natural}/S}(\log \pi^{-1}O))$, which holds for arbitrary $S$ of characteristic zero (see \S\ref{par:intro-what-we-do}). This construction also commutes with arbitrary base change in $S$. In Proposition \ref{prop:univ-property-kzb}, we characterise it by a universal property as in \S\ref{par:intro-KZB}.
\end{remark}

From now on, assume moreover that $S$ is smooth over a field $k$ of characteristic zero. The next step is to lift the relative KZB connection to an absolute integrable $k$-connection, the `isomonodromic deformation':
\[
  \nabla_{E^{\natural}/S/k}: \mathcal{O}_{E^{\natural}}\hat{\otimes}R\langle \!\langle  a,b \rangle \! \rangle \longrightarrow \Omega^1_{E^{\natural}/S/k}(\log \pi^{-1}O)\hat{\otimes}R\langle \!\langle  a,b \rangle \! \rangle\text{, }\qquad \nabla_{E^{\natural}/S/k} =  d + \omega_{E^{\natural}/S/k}.
\]
In this simplified situation, this amounts to the construction of the absolute connection form $\omega_{E^{\natural}/S/k}$, which is a suitable lift of the relative connection form $\omega_{E^{\natural}/S}$ satisfying the integrability equation
\[
  d\omega_{E^{\natural}/S/k} + \omega_{E^{\natural}/S/k}\wedge \omega_{E^{\natural}/S/k} = 0.
\]

Our next result shows that relative logarithmic differentials on $E^{\natural}/S$ admit canonical lifts to absolute differentials. This extends a well-known property for regular differentials on universal vector extensions reflecting their `crystalline nature' (cf. \cite[Section 6]{bost13}, \cite{FM22}). 

\begin{theorem}\label{intro:thm-canonical-lift}
  The relative differentials $\nu,\omega^{(0)},\omega^{(1)},\ldots \in \Gamma(E^{\natural}, \Omega^1_{E^{\natural}/S}(\log \pi^{-1}O))$ lift uniquely to absolute differentials $\widetilde{\nu},\widetilde{\omega}^{(0)},\widetilde{\omega}^{(1)},\ldots \in \Gamma(E^{\natural},\Omega^1_{E^{\natural}/k}(\log \pi^{-1}O))$ such that:
  \[
    e^*\widetilde{\nu} = e^*\widetilde{\omega}^{(0)} = 0,
  \]
  where $e \in E^{\natural}(S)$ denotes the identity section, and, for $n\ge 1$,
  \[
    \widetilde{\omega}^{(n)}\wedge \widetilde{\nu}\wedge \widetilde{\omega}^{(0)} \equiv n \alpha_{21}\wedge \widetilde{\nu} \wedge \widetilde{\omega}^{(n+1)} \mod \Omega^2_{R/k}\wedge \Gamma(E^{\natural},\Omega^1_{E^{\natural}/k}),
  \]
  where $\alpha_{21} \in \Omega^1_{R/k}$ is a coefficient of the Gauss--Manin connection matrix (cf. \cite[Remark 3.7]{FM22})
  \begin{align*}
    d\widetilde{\omega}^{(0)} &=\alpha_{11}\wedge \widetilde{\omega}^{(0)} + \alpha_{21}\wedge \widetilde{\nu}\\
    d\widetilde{\nu} &=\alpha_{12}\wedge \widetilde{\omega}^{(0)} + \alpha_{22}\wedge \widetilde{\nu}.
  \end{align*}
\end{theorem}

Now it is natural to consider the canonical lift of the relative KZB connection:
\[
  \widetilde{\nabla}_{E^{\natural}/S} : \mathcal{O}_{E^{\natural}}\hat{\otimes}R\langle \!\langle  a,b \rangle \! \rangle \longrightarrow \Omega^1_{E^{\natural}/k}(\log \pi^{-1}O)\hat{\otimes}R\langle \!\langle  a,b \rangle \! \rangle\text{, }\qquad \widetilde{\nabla}_{E^{\natural}/S} =  d + \widetilde{\omega}_{E^{\natural}/S},
\]
with
\[
  \widetilde{\omega}_{E^{\natural}/S} = -\widetilde{\nu} \otimes a  - \sum_{n\ge 0}\widetilde{\omega}^{(n)}\otimes \ad_a^nb.
\]
Crucially, this $k$-connection is \emph{not} integrable. The next result computes its curvature.

\begin{theorem}
  There is a unique 1-form over $S$ with coefficients in $k$-derivations of $R\langle \! \langle a, b\rangle\! \rangle$
  \[
    \Phi \in \Omega^1_{R/k}\hat{\otimes} \operatorname{Der}_{k}R\langle \! \langle a, b\rangle\! \rangle
  \]
  such that
  \[
    d\widetilde{\omega}_{E^{\natural}/S} + \widetilde{\omega}_{E^{\natural}/S}\wedge \widetilde{\omega}_{E^{\natural}/S} + \Phi(\widetilde{\omega}_{E^{\natural}/S}) = 0 
  \]
  in $\Gamma(E^{\natural},\Omega^2_{E^{\natural}/k}(\log \pi^{-1}O))\hat{\otimes}R\langle \! \langle a, b\rangle\! \rangle$. In particular, the connection on $\mathcal{O}_{E^{\natural}}\hat{\otimes}R\langle \!\langle  a,b \rangle \! \rangle$ defined by
  \[
   \nabla_{E^{\natural}/S/k} = d + \omega_{E^{\natural}/S/k}\text{, }\qquad  \omega_{E^{\natural}/S/k} = \widetilde{\omega}_{E^{\natural}/S} + \Phi,
  \]
  is integrable.
\end{theorem}

In short, the elliptic KZB connection of the family is obtained by `correcting' the canonical lift of the relative elliptic KZB connection by $\Phi$. We actually retrieve $\Phi$ as the connection form of the dual of the Gauss--Manin connection on the relative fundamental Hopf algebra of $E\setminus O$; see \S\ref{par:intro-what-we-do}.

\subsection{Analytic formulae}\label{intro:par-analytic}

All of the above can be explicitly computed on a given family. In order to compare our results with the traditional analytic approach in the literature, we work out in detail the case of the universal framed elliptic curve over the upper half-plane $\mathcal{E} \to \mathfrak{H}$, whose fibre at $\tau \in \mathfrak{H}$ is
\[
  \mathcal{E}_{\tau} = \mathbb{C}/(\mathbb{Z} + \tau\mathbb{Z}).
\]
In this analytic situation, the universal vector extension can be uniformised as follows:
\[
  \mathcal{E}^{\natural}_{\tau} = \mathbb{C}^2/L_{\tau}\text{, }\qquad L_{\tau} = \{(m+n\tau,2\pi i n) \in \mathbb{C}^2:m,n \in \mathbb{Z}\}.
\]

Let $\theta_{\tau}(z)$ be Jacobi's odd theta function\footnote{Our normalisation is that of Proposition \ref{prop:analytic-kronecker-diff} below.}, and consider the so-called \emph{Kronecker theta function} (cf. \cite[\S2]{LR07}, \cite[\S3.4]{BL11})
\[
    F_{\tau}(z,x) = \frac{\theta'_{\tau}(0)\theta_{\tau}(z+x)}{\theta_{\tau}(z)\theta_{\tau}(x)}.
\]
If $(z,w)$ denote the coordinates on $\mathbb{C}^2$, then the Kronecker differentials associated to
\[
  \nu = dw \in \Gamma(\mathcal{E}^{\natural}_{\tau}, \Omega^1_{\mathcal{E}^{\natural}_{\tau}})
\]
by Theorem \ref{thm:intro-log-forms} are given by
\[
  \omega^{(n)} = \varphi^{(n)}_{\tau}(z,w)dz \in \Gamma(\mathcal{E}^{\natural}_{\tau}, \Omega^1_{\mathcal{E}^{\natural}_{\tau}}(\log \pi^{-1}O)),
\]
where $\varphi^{(n)}_{\tau}(z,w)$ are complex-analytic functions on $\mathcal{E}^{\natural}_{\tau}$ defined by the generating series
\[
  e^{wx}F_{\tau}(z,x) = \sum_{n\ge 0}\varphi_{\tau}^{(n)}(z,w)x^{n-1}.
\]
Note that $\omega^{(0)} = dz$. Thus, the relative KZB connection form is
\begin{align*}
  \omega_{\mathcal{E}^{\natural}/\mathfrak{H}} = -dw \otimes a - \sum_{n\ge 0}\varphi^{(n)}_{\tau}(z,w)dz\otimes \ad_a^nb = -dw \otimes a - dz \otimes \ad_ae^{w\ad_a}F_{\tau}(z,\ad_a)b.
\end{align*}

The canonical lifts of the above relative differentials, characterised by the properties of Theorem \ref{intro:thm-canonical-lift}, are explicitly given by
\[
  \widetilde{\nu} = dw, \qquad   \widetilde{\omega}^{(n)} =  \varphi^{(n)}_{\tau}(z,w)\left(dz - w \frac{d\tau}{2\pi i} \right) + n\varphi^{(n+1)}_{\tau}(z,w) \frac{d\tau}{2\pi i}.
\]
By direct computation of the curvature $d\widetilde{\omega}_{\mathcal{E}^{\natural}/\mathfrak{H}} + \widetilde{\omega}_{\mathcal{E}^{\natural}/\mathfrak{H}}\wedge \widetilde{\omega}_{\mathcal{E}^{\natural}/\mathfrak{H}}$, we obtain
\begin{equation}\label{eq:intro-Dtau}
  \Phi = -\frac{d\tau}{2\pi i}\otimes D_{\tau}\text{, }\qquad D_{\tau} = b\frac{\partial}{\partial a} + \frac{1}{2}\sum_{n\ge 2}(2n-1)G_{2n}(\tau) \sum_{\substack{j+k=2n-1\\ j,k>0}}[(-\ad_a)^jb,\ad_a^kb]\frac{\partial}{\partial b},
\end{equation}
where $G_{2n}(\tau) = \sum_{(r,s) \neq (0,0)}(r+s\tau)^{-2n}$ are the classical Eisenstein series. The final expression for the KZB connection form then becomes
\[
  \omega_{\mathcal{E}^{\natural}/\mathfrak{H}/\mathbb{C}} = -dw\otimes a - dz\otimes \ad_ae^{w\ad_a}F_{\tau}(z,\ad_a)b - \frac{d\tau}{2\pi i}\otimes \left( \ad_aF'_{\tau}(z,w,\ad_a)b + D_{\tau}\right),
\]
where $F'_{\tau}(z,w,x) = e^{wx}\frac{\partial}{\partial x}F_{\tau}(z,x) + \frac{1}{x^2}$.

\subsection{Relation to the literature}

The elliptic KZB connection \cite{bernard98} arose in Conformal Field Theory as a genus one version of the KZ connection \cite{KZ84}, which in its simplest guise is the pro-algebraic connection
\begin{equation}\label{eq:intro-kz}
    \nabla_{\KZ}: \mathcal{V}_{\KZ} \longrightarrow \Omega^1_{\mathbb{P}^1}(\log \{0,1,\infty\})\hat{\otimes} \mathcal{V}_{\KZ}\text{, }\qquad  \nabla_{\KZ} = d  - \frac{dz}{z}\otimes x_0 - \frac{dz}{1-z} \otimes x_1,
  \end{equation}
  where $\mathcal{V}_{\KZ}$ is the trivial pro-vector bundle over $\mathbb{P}^1$ with fibre the algebra of non-commutative power series $\mathbb{C}\langle \!\langle x_0,x_1 \rangle\!\rangle$. The KZ connection encodes quantities of deep arithmetic interest, obtained as iterated integrals of the differential 1-forms $\frac{dz}{z}$, $\frac{dz}{1-z}$. Namely, flat sections of $\nabla_{\KZ}$ are described by multiple polylogarithms, and their monodromy by multiple zeta values; see for instance \cite[\S4]{brown13}.

  Motivated by an elliptic analogue of the theory of multiple polylogarithms, Levin and Racinet \cite{LR07} were led to consider elliptic KZB connections as defined in \S\ref{par:intro-KZB}. In contrast to the genus zero case, however, the pro-vector bundle $\mathcal{V}_E$ is \emph{not} trivial, since the condition
  \begin{equation}\label{eq:intro-obstruction}
    H^1(E,\mathcal{O}_E) \neq 0
  \end{equation}
  amounts to the existence of non-trivial unipotent vector bundles on $E$. To obtain a formula as explicit as \eqref{eq:intro-kz}, they  compute the pullback of the elliptic KZB connection on $\mathcal{E}_{\tau} = \mathbb{C}/(\mathbb{Z} + \mathbb{Z}\tau)$ by the (analytic) uniformisation map $\mathbb{C} \to \mathcal{E}_{\tau}$:
  \[
    \nabla_{\tau}: \mathcal{O}_{\mathbb{C}}\hat{\otimes} \mathbb{C} \langle \! \langle a, b \rangle\! \rangle \longrightarrow \Omega^1_{\mathbb{C}}(\log (\mathbb{Z} + \tau \mathbb{Z})) \hat{\otimes} \mathbb{C} \langle \! \langle a, b \rangle\! \rangle\text{, }\qquad \nabla_{\tau} = d  -dz \otimes \ad_aF_{\tau}(z,\ad_a)b,
  \]
  with corresponding action of $\mathbb{Z} + \tau \mathbb{Z}$ on $\mathbb{C}\langle \!\langle a,b\rangle\!\rangle$ given by $(m+n\tau)\cdot f(a,b) = e^{-2\pi i na}f(a,b)$.

  By considering the commutative diagram
  \[
    \begin{tikzcd}
       \mathbb{C}^2 \arrow{r}\arrow{d} & \mathcal{E}_{\tau}^{\natural} \arrow{d}{\pi} \\
       \mathbb{C} \arrow{r} & \mathcal{E}_{\tau}
    \end{tikzcd}
  \]
  where horizontal arrows are the natural uniformisation maps and the left vertical arrow is the projection $(z,w)\mapsto z$, one can readily check that $f(a,b) \mapsto e^{-wa}f(a,b)$ induces an isomorphism between the pullbacks to $\mathbb{C}^2$ of our $(\mathcal{O}_{\mathcal{E}^{\natural}_{\tau}}\hat{\otimes}\mathbb{C}\langle \!\langle a, b\rangle \! \rangle,\nabla_{\mathcal{E}^{\natural}_\tau})$, as given in \S\ref{intro:par-analytic}, and Levin and Racinet's $(\mathcal{O}_{\mathbb{C}}\hat{\otimes} \mathbb{C} \langle \! \langle a, b \rangle\! \rangle, \nabla_{\tau})$.

  At this point, it is also instructive to compare our construction with Brown and Levin's theory of multiple elliptic polylogarithms \cite{BL11}, which rely on real-analytic logarithmic 1-forms $\nu_{\BL},\omega^{(0)}_{\BL},\omega^{(1)}_{\BL},\ldots$ defined by
\[
    \nu_{\BL} = 2\pi i dr\text{, }\qquad e^{2\pi i rx}F_{\tau}(z,x)dz = \sum_{n\ge 0}\omega^{(n)}_{\BL}x^{n-1},
\]
where $r(z) = \Im(z)/\Im(\tau)$. The presence of $r$ in their construction is justified by the transformation property
\[
    r(z + m + n\tau) = r(z) + n,
\]
which, together with the modularity properties of Kronecker's function, implies that the differentials $\nu_{\BL},\omega^{(n)}_{\BL}$ descend to $\mathcal{E}_{\tau}$. In this sense, the non-algebraicity in Brown and Levin's work is also related to the same cohomological obstruction \eqref{eq:intro-obstruction}, which turns out to be equivalent to the non-existence of a holomorphic function on $\mathcal{E}_{\tau}$ which transforms in the same way as $r$.

Our Kronecker differentials $\nu,\omega^{(n)}$ should be regarded as algebraic avatars of Brown and Levin's differentials $\nu_{\BL},\omega^{(n)}_{\BL}$. Indeed, the projection $\pi: \mathcal{E}_{\tau}^{\natural} \to \mathcal{E}_{\tau}$ admits a real-analytic section $\sigma: \mathcal{E}_{\tau} \to \mathcal{E}_{\tau}^{\natural}$ induced by $z\mapsto (z,2\pi i r(z))$, and we have
\[
  \nu_{\BL} = \sigma^*\nu\text{, }\qquad \omega_{\BL}^{(n)} = \sigma^*\omega^{(n)}.
\]
With this point of view, the universal vector extension $\mathcal{E}^{\natural}$ can be naively thought of as a space over the elliptic curve $\mathcal{E}_{\tau}$ obtained by adjoining a formal variable $w$ which transforms as $2\pi ir$ under the action of $\mathbb{Z} + \tau \mathbb{Z}$.

The \emph{universal} elliptic KZB connection was first considered in explicit form by Calaque, Enriquez, and Etingof \cite{CEE09}, in relation to the representation theory of braid monodromy groups. It is defined as an integrable connection on $\mathcal{V}_{\KZB}$, the trivial infinite-rank vector bundle over $\mathfrak{H}\times \mathbb{C}$ with fibre $\mathbb{C} \langle \! \langle a, b\rangle \! \rangle$, given by
\[
    \nabla_{\KZB} = d - dz\otimes \ad_aF_{\tau}(z,\ad_a)b - \frac{d\tau}{2\pi i}\otimes \left(\ad_aG_{\tau}(z,\ad_a)b + D_{\tau}\right),
\]
where $G_{\tau}(z,x) = \frac{\partial}{\partial x}F_{\tau}(z,x) + \frac{1}{x^2}$, and $D_{\tau}$ is as in \eqref{eq:intro-Dtau}. By considering a suitable action of $\SL_2(\mathbb{Z}) \ltimes \mathbb{Z}^2$, the connection $(\mathcal{V}_{\KZB},\nabla_{\KZB})$ is then proved to descend to the universal elliptic curve, seen as the orbifold quotient $\mathcal{E} = (\SL_2(\mathbb{Z}) \ltimes \mathbb{Z}^2) \backslash \! \backslash (\mathfrak{H}\times \mathbb{C})$.

The comparison with our connection on the universal vector extension is done via a universal analogue of the previous commutative diagram:
\[
    \begin{tikzcd}
       \mathfrak{H} \times \mathbb{C}^2 \arrow{r}\arrow{d} & \mathcal{E}^{\natural} \arrow{d}{\pi} \\
       \mathfrak{H}\times \mathbb{C} \arrow{r} & \mathcal{E}.
    \end{tikzcd}
\]
It follows from the explicit expressions in \S\ref{intro:par-analytic} that the pullback to $\mathfrak{H} \times\mathbb{C}^2$ of $(\mathcal{O}_{\mathcal{E}^{\natural}}\hat{\otimes}\mathbb{C}\langle \!\langle a, b\rangle \! \rangle,\nabla_{\mathcal{E}^{\natural}/\mathfrak{H}/\mathbb{C}})$ is isomorphic to the pullback of $(\mathcal{V}_{\KZB},\nabla_{\KZB})$.  There are also similar formulae for higher level universal elliptic KZB connections due to Calaque and Gonzalez \cite{CG20} (cf. \cite{hopper21}), and a comparison in full generality is worked out in Section \ref{sec:analytic-formuli} below. 

\subsection{Towards motivic multiple elliptic polylogarithms}

The present work has been originally motivated by the development of a motivic theory of multiple elliptic polylogarithms, in the framework of Brown's \emph{motivic periods} \cite{brown14,brown17} (which have been applied with great success to arithmetic questions concerning multiple zeta values). Recall that motivic periods involve Betti and algebraic de Rham realisations; this paper is purely devoted to the algebraic de Rham aspects of the theory (cf. \cite{FM20}).

More precisely, we are concerned here with \emph{unipotent algebraic de Rham fundamental groups}. In the Tannakian formalism, this amounts to the study of unipotent vector bundles with connection over punctured elliptic curves (see Appendix \ref{appendix:tannakian}). Our use of the universal vector extension is motivated by the cohomological properties (over a field $k$ of characteristic zero)
\[
    H^0(E^{\natural},\mathcal{O}_{E^{\natural}}) = k \text{, }\qquad H^1(E^{\natural},\mathcal{O}_{E^{\natural}}) = 0,
\]
which imply that every unipotent vector bundle over $E^{\natural}$ is canonically trivial. The importance of the universal vector extension in the algebraic de Rham fundamental group theory of a punctured elliptic curve has been previously advocated by Deligne (personal communications with P. Etingof and R. Hain, 2015), and some of its Tannakian implications have already been explored by Enriquez and Etingof \cite{EE18} (cf. \S\ref{par:de-Rham-group} below).

The above discussion is also connected to a number of natural algebraicity questions that have been raised in the literature concerning elliptic KZB equations, as the usual approach relies on analytic uniformisation maps. The algebraicity over $\mathbb{Q}$ of the universal elliptic KZB connection in level 1 was proved by Luo \cite{luo19} (cf. \cite[\S5]{LR07}) by making essential use of the moduli space $\mathcal{M}_{1,\vec{1}}$ classifying elliptic curves with a non-zero tangent vector at the identity. Here, the map $\mathcal{M}_{1,\vec{1}} \to \mathcal{M}_{1,1}$, or the corresponding map on universal elliptic curves, can be thought of as a particular $\mathbb{G}_m$-bundle over which the algebraicity question becomes computationally tractable. In this sense, our approach, which is based on the construction of the elliptic KZB connection on the universal vector extension of a family of elliptic curves, is not far in spirit from that of Luo, with the difference that we use a $\mathbb{G}_a$-bundle instead of a $\mathbb{G}_m$-bundle.

Algebraicity problems were also considered in the literature concerning elliptic polylogarithm sheaves (in the sense of Beilinson and Levin \cite{BL94}), usually motivated by arithmetic questions concerning $p$-adic realisations elliptic polylogarithm functions \cite{BKT10,sprang20}. Note that Sprang's approach \cite{sprang20} also relies on universal vector extensions, and it would be interesting to compare it with the methods of this paper.

\subsection{What we do}\label{par:intro-what-we-do}

Let $k$ be a field of characteristic zero, $S$ be a smooth $k$-scheme, and $(p:E \to S,O)$ be an elliptic curve over $S$. Consider its universal vector extension $f: E^{\natural} \to S$, which fits into a short exact sequence of commutative $S$-group schemes
\[
    \begin{tikzcd}
        0 \arrow{r} & \mathbb{V}(R^1p_*\mathcal{O}_E) \arrow{r} & E^{\natural} \arrow{r}{\pi}& E \arrow{r} & 0,
    \end{tikzcd}
\]
and is universal for extensions of $E$ by an $S$-vector group (see \S\ref{par:def-uve}). Let $Z\subset E$ be a subscheme given by a finite union of torsion sections of $p$. For simplicity, we also assume that $Z$ contains the identity section $O$.

We shall construct the elliptic KZB connection over $E^{\natural}$ with logarithmic singularities along $\pi^{-1}Z$. Our approach is based on the \emph{bar construction} formalism. We refer to the recent work of Chiarellotto, Di Proietto, and Shiho \cite{CdPS19} for general comparison statements of different approaches to unipotent fundamental groups. Regarding our particular framework, we have included in Appendix \ref{appendix:tannakian} precise comparison results with the Tannakian formalism over a field.

\subsubsection{Relative differentials}

Consider the the dg-algebra over $\mathcal{O}_S$ of relative logarithmic differential forms
\[
    \mathcal{A} \defeq f_*\Omega^{\bullet}_{E^{\natural}/S}(\log \pi^{-1}Z).
\]
It follows from a result due independently to Coleman \cite{coleman98} and Laumon \cite{laumon96} (Theorem \ref{thm:coleman-laumon}) that $\mathcal{A}$ is a model for the de Rham cohomology of $E\setminus Z$ over $S$ (Proposition \ref{prop:model-derham}):   
\[
  H^{\bullet}(\mathcal{A}) \cong H^{\bullet}_{\dR}((E\setminus Z)/S).
\]

Our main results in Section \ref{sec:relative-differentials} concern the structure of $\mathcal{A}$ as a dg-algebra over $\mathcal{O}_S$. In particular, we obtain in Theorem \ref{thm:kronecker-subbundle} a canonical decomposition
\[
  \mathcal{A}^1 = f_*\Omega^1_{E^{\natural}/S} \oplus \bigoplus_{n\ge 1}\mathcal{K}^{(n)},
\]
where the $\mathcal{O}_S$-submodules $\mathcal{K}^{(n)}$ are characterised by conditions involving the residue along $\pi^{-1}Z$ and the dg-algebra structure of $\mathcal{A}$. Concretely, locally over $S$, we have
\[
  f_*\Omega^1_{E^{\natural}/S} = \mathcal{O}_S\nu \oplus \mathcal{O}_S \omega^{(0)}\text{, }\qquad \mathcal{K}^{(n)} = \bigoplus_{P \in Z(S)}\mathcal{O}_S \omega^{(n)}_P,
\]
where $\omega^{(n)}_P$ are Kronecker differentials with logarithmic singularities along $\pi^{-1}P$.

This also allows us to prove the formality of the dg-algebra $\mathcal{A}$. More precisely, we obtain a canonical dg-quasi-isomorphism (Theorem \ref{thm:projector})
\[
  \mathcal{A} \longrightarrow H^{\bullet}(\mathcal{A})
\]
which plays a key role in the rest of the paper (cf. \cite[Theorem 19]{BL11}). 

\subsubsection{Relative KZB}

With $\mathcal{A}$ as above (note that $\mathcal{A}$ is a connected dg-algebra over $\mathcal{O}_S$), we consider the bar complex
\[
    \begin{tikzcd}
        0 \arrow{r}& B^0(\mathcal{A}) \arrow{r}{d_B} & B^1(\mathcal{A}) \arrow{r} & \cdots.
    \end{tikzcd}
\]
We shall only need the first two terms, which are explicitly given by $B^0(\mathcal{A}) = \bigoplus_{n\ge 0}(\mathcal{A}^1)^{\otimes n}$ and $B^1(\mathcal{A}) = \bigoplus_{n\ge 1}\bigoplus_{1\le i \le n}(\mathcal{A}^1)^{\otimes i-1}\otimes \mathcal{A}^2 \otimes (\mathcal{A}^1)^{\otimes n-i}$, where tensor products are over $\mathcal{O}_S$. Decomposable tensors are denoted by $a_1\otimes \cdots \otimes a_n = [a_1|\cdots |a_n]$, and the differential in degree zero is explicitly given by
\begin{align*}
    d_B: B^0(\mathcal{A}) \longrightarrow B^1(\mathcal{A})\text{, }\qquad 
    [a_1|\cdots|a_n]&\longmapsto  -\sum_{i=1}^n[a_1|\cdots|a_{i-1}|da_i|a_{i+1}|\cdots |a_n]\\
                          &\hspace{50px} -\sum_{i=1}^{n-1}[a_1|\cdots|a_{i-1}|a_i\wedge a_{i+1}|a_{i+2}|\cdots|a_n].
\end{align*}
Finally, we consider the $\mathcal{O}_S$-module
\[
    \mathcal{H}_{E/S,Z} \defeq H^0(B(\mathcal{A})) = \ker(d_B: B^0(\mathcal{A}) \longrightarrow B^1(\mathcal{A})).
\]
It comes with a natural commutative Hopf algebra structure, given by the deconcatenation coproduct and the shuffle product, and a natural filtration $L_n\mathcal{H}_{E/S,Z}$ by length (see \S \ref{par:bar-construction}). Note that $\mathcal{H}_{E/S,Z}$ can be thought of as the Hopf algebra corresponding to the relative de Rham unipotent fundamental group of $E\setminus Z$ over $S$ at certain `canonical base point' (see \S \ref{par:de-Rham-group} for a discussion in the case where $S$ is the spectrum of a field). 

In Section \ref{sec:relative-kzb}, the \emph{relative elliptic KZB connection} is naturally defined on the pullback of the continuous dual
\[
    \mathcal{H}_{E/S,Z}^{\vee} \defeq \lim_{n \ge 0}\mathcal{H}om_{\mathcal{O}_S}(L_n\mathcal{H}_{E/S,Z},\mathcal{O}_S)
\]
by
\[
    \nabla_{E^{\natural}/S,Z} : f^*\mathcal{H}_{E/S,Z}^{\vee} \longrightarrow \Omega^1_{E^{\natural}/S}(\log \pi^{-1}Z) \hat{\otimes}f^*\mathcal{H}_{E/S,Z}^{\vee}\text{, }\qquad \nabla_{E^{\natural}/S,Z} =  d + \omega_{E^{\natural}/S,Z},
\]
where the \emph{KZB form} $\omega_{E^{\natural}/S,Z} \in \Gamma(S, \mathcal{A}^1\hat{\otimes}\mathcal{H}^{\vee}_{E/S,Z})$ is the length 1 component of the element in $\Gamma(S,\mathcal{H}_{E/S,Z}\hat{\otimes}\mathcal{H}^{\vee}_{E/S,Z})$ induced by the Hopf algebra antipode, and acts on $\mathcal{H}^{\vee}_{E/S,Z}$ by left multiplication (see \S \ref{par:kzb-connection}). In Proposition \ref{prop:univ-property-kzb}, we prove the integrability of $(f^*\mathcal{H}_{E/S,Z}^{\vee}, \nabla_{E^{\natural}/S,Z})$ and we characterise it via a universal property.

Building on the results of Section \ref{sec:relative-differentials}, we show that $\mathcal{H}_{E/S,Z}$ is canonically isomorphic to the tensor coalgebra $T^c H^1_{\dR}((E\setminus Z)/S)$ (Theorem \ref{thm:tensor-coalgebra}). In particular, locally over $S$, we show in Theorem \ref{thm:explicit-kzb-form} that the continuous dual $\mathcal{H}_{E/S,Z}^{\vee}$ is isomorphic to the algebra of non-commutative power series
\[
  \mathcal{H}_{E/S,Z}^{\vee}\cong \frac{\mathcal{O}_S \langle \!\langle a,b, c_P : P \in Z(S) \rangle \! \rangle}{\langle \sum_{P \in Z(S)}c_P - [a,b]\rangle},
\]
and, under the above isomorphism, the KZB form is given by
\[
  \omega_{E^{\natural}/S,Z} = - \nu \otimes a - \omega^{(0)}\otimes b - \sum_{n\ge 1}\sum_{P \in Z(S)}\omega^{(n)}_P\otimes \ad_a^{n-1}c_P.
\]
When $Z=O$ (level 1), we recover the expressions in \S\ref{par:intro-construction-kzb}.

\subsubsection{Canonical lifts}

Our next results concern the sheaves of \emph{absolute} logarithmic differentials $f_*\Omega^{\bullet}_{E^{\natural}/k}(\log \pi^{-1}Z)$. It is known that $E^{\natural}/S$ admits a natural horizontal foliation --- formally, a $D$-group scheme structure  --- which amounts to a splitting of
\[
    \begin{tikzcd}
        0 \arrow{r} & f^*\Omega^1_{S/k} \arrow{r} & \Omega^1_{E^{\natural}/k} \arrow{r} & \Omega^1_{E^{\natural}/S} \arrow{r} & 0
    \end{tikzcd}
\]
satisfying a certain integrability condition and a compatibility with the group scheme structure \cite[\S 6.4]{bost13}. We show in \cite{FM22} that this splitting actually comes from the splitting of
\[
    \begin{tikzcd}
        0 \arrow{r} & \Omega^1_{S/k} \arrow{r} & f_*\Omega^1_{E^{\natural}/k} \arrow{r} & f_*\Omega^1_{E^{\natural}/S} \arrow{r} & 0
    \end{tikzcd}
\]
induced by the retraction $f_*\Omega^1_{E^{\natural}/k} \to \Omega^1_{S/k}$ given by restriction to the identity section $e \in E^{\natural}(S)$:
\[
    f_*\Omega^1_{E^{\natural}/k} = \Omega^1_{S/k} \oplus \mathcal{N}\text{, }\qquad \mathcal{N} = \ker(e^*). 
\]
In Theorem \ref{thm:canonical-lift-kronecker-subbundle}, we extend this result to logarithmic differentials by showing that there is a unique splitting of
\[
  \begin{tikzcd}
      0 \arrow{r} & \Omega^1_{S/k} \arrow{r} & f_*\Omega^1_{E^{\natural}/k}(\log \pi^{-1}Z) \arrow{r} & f_*\Omega^1_{E^{\natural}/S}(\log \pi^{-1}Z) \arrow{r} & 0
  \end{tikzcd}
\]
of the form
\[
  f_*\Omega^1_{E^{\natural}/k}(\log \pi^{-1}Z) = \Omega^1_{S/k} \oplus \mathcal{N} \oplus \bigoplus_{n\ge 1}\mathcal{L}^{(n)}
 \]
  where each $\mathcal{L}^{(n)}$ maps isomorphically onto $\mathcal{K}^{(n)}$, and for every $n\ge 1$
  \[
        \mathcal{L}^{(n)}\wedge \mathcal{N} \wedge \mathcal{N} \equiv d\mathcal{N} \wedge \mathcal{L}^{(n+1)} \mod f_*\mathcal{F}^2.
  \]
Here, $f_*\mathcal{F}^2$ is the second step of the (direct image) Koszul filtration on $f_*\Omega^{\bullet}_{E^{\natural}/k}(\log \pi^{-1}Z)$, given by the ideal generated by $\Omega^2_{S/k}$ (see \S \ref{par:canonical-lifts}).

Locally over $S$, the above result implies that the Kronecker differentials admit canonical lifts to sections of $f_*\Omega^1_{E^{\natural}/k}(\log \pi^{-1}Z)$, which we denote by $\widetilde{\nu}$, $\widetilde{\omega}^{(0)}$, and $\widetilde{\omega}^{(n)}_P$ for $n\ge 1$ and $P \in Z(S)$.

\subsubsection{Absolute KZB}

Finally, in Section \ref{sec:absolute-kzb}, the \emph{elliptic KZB connection} associated to $E/S/k$ punctured at $Z$ is constructed as a suitable lift of the $S$-connection $\nabla_{E^{\natural}/S,Z}$ to an integrable $k$-connection
\[
    \nabla_{E^{\natural}/S/k,Z} : f^*\mathcal{H}_{E/S,Z}^{\vee} \longrightarrow \Omega^1_{E^{\natural}/k}(\log \pi^{-1}Z) \hat{\otimes}f^*\mathcal{H}_{E/S,Z}^{\vee}.
\]

As explained in \S\ref{par:intro-construction-kzb}, the key idea is to obtain $\nabla_{E^{\natural}/S/k,Z}$ as a correction of the canonical lift of the relative connection $\widetilde{\nabla}_{E^{\natural}/S/k,Z}$ by a certain $k$-connection on $S$, which is dual to a \emph{Gauss--Manin connection}
\[
    \delta: \mathcal{H}_{E/S,Z} \longrightarrow \Omega^1_{S/k}\otimes \mathcal{H}_{E/S,Z}.
\]
Our construction of $\delta$ is a bar complex variant of the usual Katz--Oda procedure \cite{KO68}, and relies on the use of \emph{relative bar complexes}. It depends crucially on Coleman and Laumon's result on the cohomology of universal vector extensions. We have also drawn inspiration from an analogous approach developed by Brown and Levin \cite{BL12}. In the $C^{\infty}$ context, there is a similar construction due to Hain and Zucker \cite[Proposition 4.11]{HZ87}.
  
Consider the dg-algebra $\mathcal{C} \defeq f_*\Omega^{\bullet}_{E^{\natural}/k}(\log \pi^{-1}Z)$, which contains $\Omega \defeq \Omega^{\bullet}_{S/k}$ as a dg-subalgebra, and form the relative bar complex
\[
    \begin{tikzcd}
        0 \arrow{r}& B_{\Omega}^0(\mathcal{C}) \arrow{r}{d_B} & B_{\Omega}^1(\mathcal{C}) \arrow{r} & \cdots,
    \end{tikzcd}
\]
the definition of which is similar to the usual bar complex, but tensor products are now taken over the non-commutative ring $\Omega$ (see \S \ref{par:relative-bar}). Then, the Koszul filtration on $\mathcal{C}$ induces a decreasing filtration
\[
    B_{\Omega}(\mathcal{C}) = F^{0}B_{\Omega}(\mathcal{C})\supset F^1B_{\Omega}(\mathcal{C}) \supset F^2B_{\Omega}(\mathcal{C})\supset \cdots
\]
satisfying $B_{\Omega}(\mathcal{C})/ F^1B_{\Omega}(\mathcal{C}) \cong B(\mathcal{A})$. We construct a projection
\[
    \pi: F^1B_{\Omega}(\mathcal{C}) \longrightarrow \Omega^1_{S/k}\otimes B(\mathcal{A})[-1]
\]
which factors through $F^1B_{\Omega}(\mathcal{C})/F^2B_{\Omega}(\mathcal{C})$, and we define
\[
    \delta(\xi) = -\pi ( d_B\widetilde{\xi}),
\]
where $\widetilde{\xi}$ is the canonical lift of the section $\xi$ of $\mathcal{H}_{E/S,Z} = H^0(B(\mathcal{A}))$.

We prove in Theorems \ref{thm:non-abelian-gm} and \ref{thm:integrability} that the above definition yields an integrable $k$-connection on $\mathcal{H}_{E/S,Z}$ which preserves the length filtration and restricts to the tensor power of the Gauss--Manin connection on the graded quotients $L_n\mathcal{H}_{E/S,Z}/L_{n-1}\mathcal{H}_{E/S,Z} \cong H^1_{\dR}((E\setminus Z)/S)^{\otimes n}$. Moreover, its continuous dual
\[
        \delta^{\vee} : \mathcal{H}^{\vee}_{E/S,Z} \longrightarrow \Omega^1_{S/k}\hat{\otimes} \mathcal{H}^{\vee}_{E/S,Z}
 \]
 is a derivation of $\mathcal{H}^{\vee}_{E/S,Z}$ with coefficients in $\Omega^1_{S/k}$, and the $k$-connection
  \[
        \nabla_{E^{\natural}/S/k,Z}: f^*\mathcal{H}^{\vee}_{E/S,Z} \longrightarrow \Omega^1_{E^{\natural}/k}(\log \pi^{-1}Z)\hat{\otimes} f^*\mathcal{H}^{\vee}_{E/S,Z}\text{, }\qquad \nabla_{E^{\natural}/S/k,Z} = f^*\delta^{\vee} + \widetilde{\omega}_{E^{\natural}/S,Z}
  \]
is an integrable lift of the relative KZB connection $\nabla_{E^{\natural}/S,Z}$. 

\subsection{Acknowledgements}

This project has received funding from the European Research Council (ERC) under the European Union’s Horizon 2020 research and innovation programme (Grant Agreement No. 724638). The first author is currently supported by the grant $\#$2020/15804-1, São Paulo Research Foundation (FAPESP), and the second author was supported by a Walter-Benjamin Fellowship of the Deutsche Forschungsgemeinschaft (DFG). We would like to thank Francis Brown, Netan Dogra, Benjamin Enriquez, and Richard Hain for their enlightening (and often encouraging) comments during the writing of this manuscript.

\section{The universal vector extension of an elliptic curve} \label{sec:uve}

\subsection{Definition} \label{par:def-uve}

Let $p:E \to S$ be an elliptic curve. The \emph{universal vector extension} of $E/S$ is a commutative $S$-group scheme $f:E^{\natural} \to S$ with a morphism of $S$-group schemes $\pi: E^{\natural} \to E$ fitting into an exact sequence (of abelian fppf sheaves over $S$)
\begin{equation}\label{eq:exact-uve}
    \begin{tikzcd}
        0 \arrow{r} & \mathbb{V}(R^1p_*\mathcal{O}_E) \arrow{r} & E^{\natural} \arrow{r}{\pi}& E \arrow{r} & 0
    \end{tikzcd}
\end{equation}
which is universal for extensions of $E$ by an $S$-vector group. Namely,
\begin{align*}
    \Hom_{\mathcal{O}_S}(\mathcal{F},R^1p_*\mathcal{O}_E) &\To \Ext_{S_{\rm fppf}}(E,\mathbb{V}(\mathcal{F}))\\
    \varphi &\longmapsto \text{the class of the pushout of \eqref{eq:exact-uve} along }\mathbb{V}(\varphi)
\end{align*}
is an isomorphism for any vector bundle $\mathcal{F}$ over $S$ \cite[Prop. 1.10]{MM74} (cf. \cite[\S 2.2]{FM22}). Since $R^1p_*\mathcal{O}_E$ is a line bundle over $S$, it follows from \eqref{eq:exact-uve} that $\pi :E^{\natural} \to E$ is a $\mathbb{G}_a$-bundle. In particular, $f:E^{\natural} \to S$ is smooth, of finite presentation, and of relative dimension 2. Moreover, the formation of the universal vector extension is compatible with every base change $S' \to S$, meaning that there is a canonical $S'$-isomorphism $E^{\natural}\times_SS' \cong (E\times_S S')^{\natural}$.

\begin{example}\label{ex:algebraic-gluing}
    Assume that $S=\Spec R$ is affine, with $6$ invertible in $R$, and that $E/S$ admits a Weierstrass equation of the form
    \[
        E : y^2z = 4x^3 - g_2xz^2 - g_3z^3\text{, }\qquad g_2,g_3 \in R\text{, }\qquad g_ 2^3-27g_ 3^2 \in R^{\times}.
    \]
    Let $q(x) = 4x^3 - g_2x - g_3 \in R[x]$, and set
    \[
        U_1 = \Spec R[x,y,t]/(y^2-q(x))\text{, }\qquad U_2 = \Spec R[x,z,t]/(z-z^3q(x/z)).
    \]
    Let $U_{12}$ be the open subscheme of $U_1$ given by $y\neq 0$, and $U_{21}$ be the open subscheme of $U_2$ given by $z\neq 0$. Then, $E^{\natural}$ is isomorphic to the gluing of $U_1$ and $U_2$ along the isomorphism $U_{12} \stackrel{\sim}{\to}U_{21}$ given on the corresponding $R$-algebras by
    \[
        R[x,z^{\pm 1},t]/(z - z^3q(x/z)) \stackrel{\sim}{\To} R[x,y^{\pm 1},t]/(y^2 - q(x))\text{, }\qquad 
        (x,z,t)\longmapsto \left(\frac{x}{y},\frac{1}{y},t - \frac{q'(x)}{6y}\right),
    \]
    where $q'(x) = 12x^2-g_2$. This follows from the interpretation of $E^{\natural}$ as a moduli space of line bundles on $E$ equipped with an $S$-connection \cite[C.1-C.2]{katz77}. For instance, if $R'$ is an $R$-algebra, a point $(a,b,c) \in U_1(R')$ corresponds to the isomorphism class of
    \[
        (\mathcal{O}(P)\otimes \mathcal{O}(O)^{-1}, d + \omega_{P,c})\text{,} \qquad \omega_{P,c} = \left(\frac{1}{2}\frac{y+b}{x-a} +c\right)\frac{dx}{y},
    \]
    where $O = (0:1:0)\in E(R')$ is the identity, and $P=(a:b:1) \in E(R')$.
\end{example}

In the category of complex analytic spaces, the universal vector extension of an elliptic curve is more conveniently described in terms of its uniformisation.

\begin{example}\label{ex:uniformisation}
    Assume that $S$ is locally of finite type over $\mathbb{C}$. Then the analytification $E^{\natural,\an}$ is uniformised by the rank 2 holomorphic vector bundle $V^{\an} \defeq \mathbb{V}(H^1_{\dR}(E/S))^{\an}$ over $S^{\an}$: there is an exact sequence of relative complex Lie groups over $S^{\an}$
    \[
        \begin{tikzcd}
            0 \arrow{r} & L \arrow{r} & V^{\an} \arrow{r}{\exp} &  E^{\natural,\an} \arrow{r} & 0,
        \end{tikzcd}
    \]
    where $L$ is the space over $S^{\an}$ corresponding to the locally constant sheaf $(R^1p^{\an}_*\mathbb{Z})^{\vee}$. The map $L \to V^{\an}$ is induced by the morphism $(R^1p^{\an}_*\mathbb{Z})^{\vee} \to (H^1_{\dR}(E/S)^{\an})^{\vee}$ which sends a locally constant family of topological 1-cycles $\gamma$ to the functional $\alpha \mapsto \int_{\gamma}\alpha$. We refer to \cite[I.4.4]{MM74} for a proof. In the particular case where $S= \Spec \mathbb{C}$, let $(\omega,\eta)$ be a basis of $H^1_{\dR}(E/\mathbb{C})$. Then,
    \[
        E^{\natural,\an} \cong \mathbb{C}^2/L\text{, }\qquad L = \{({\smallint}_{\gamma}\omega,{\smallint}_{\gamma}\eta) \in \mathbb{C}^2: \gamma \in H_1(E^{\an},\mathbb{Z})\}.
    \]
\end{example}

\subsection{Coherent and de Rham cohomology} \label{par:cohomology}

We keep the above notation.

\begin{theorem}[Coleman, Laumon]\label{thm:coleman-laumon}
    If $S$ is of characteristic zero, then
    \[
        R^if_*\mathcal{O}_{E^{\natural}} =
        \begin{cases}
            \mathcal{O}_S & i=0\\
            0 & i \ge 1.
        \end{cases}
    \]
\end{theorem}

For a proof of the more general statement concerning universal vector extensions of abelian schemes, we refer to \cite[Théorème 2.4.1]{laumon96} or to \cite[Corollary 2.7]{coleman98}. For elliptic curves, one may also give an elementary proof by computing the \v{C}ech cohomology of the affine cover given in Example \ref{ex:algebraic-gluing}.

\begin{remark}[Algebraic versus analytic functions]\label{rmk:gaga}
    Let $S=\Spec \mathbb{C}$. The above theorem implies that $\Gamma(E^{\natural},\mathcal{O}_{E^{\natural}}) = \mathbb{C}$, i.e., that every global regular function on $E^{\natural}$ is constant. The analogous statement in the analytic category is \emph{false}. In other words, $E^{\natural,\an}$ admits non-constant holomorphic functions. In fact, it follows from Example \ref{ex:uniformisation} that $E^{\natural,\an}$ is isomorphic to $\mathbb{C}^2/\mathbb{Z}^2 \cong \mathbb{C}^{\times}\times \mathbb{C}^{\times}$.
\end{remark}

From now on, we assume that $S$ is of characteristic zero. Let $e \in E^{\natural}(S)$ be the identity section. Since $E^{\natural}/S$ is a smooth group scheme, we have for every $n\ge 0$ an isomorphism
\[
    f^*e^*\Omega^n_{E^{\natural}/S} \stackrel{\sim}{\To} \Omega^n_{E^{\natural}/S},
\]
obtained by extending cotangent vectors at the identity to invariant differential forms via the group law. Then, it follows from the projection formula and Theorem \ref{thm:coleman-laumon} that
\begin{align}\label{eq:pushforward-regular-differentials}
    R^if_*\Omega^n_{E^{\natural}/S} \cong
    \begin{cases}
        e^*\Omega^n_{E^{\natural}/S} & i=0 \\
        0 & i\ge 1.
    \end{cases}
\end{align}
The above equation for $i=0$ means that every section of $f_*\Omega^n_{E^{\natural}/S}$ is an invariant differential $n$-form on $E^{\natural}/S$.

\begin{proposition}\label{prop:closed}
    Every section of $f_*\Omega^n_{E^{\natural}/S}$ is a closed differential form. The induced map
    \[
        f_*\Omega^n_{E^{\natural}/S} \To H^n_{\dR}(E^{\natural}/S)
    \]
    is an isomorphism of $\mathcal{O}_S$-modules.
\end{proposition}

This statement is contained in \cite[Propositions 2.4 and 2.7]{FM22} (cf. \cite[Theorem 2.2]{coleman98}). We reproduce a short proof for completeness.

\begin{proof}
    We have seen that sections of $f_*\Omega^n_{E^{\natural}/S}$ are invariant. That every invariant differential form is closed is a general property of smooth commutative group schemes; it is a consequence of the Maurer--Cartan equation. Now, that the natural maps $f_*\Omega^n_{E^{\natural}/S} \to H^n_{\dR}(E^{\natural}/S)$ are isomorphisms follows from the $f_*$-acyclicity of $\Omega^n_{E^{\natural}/S}$ (Equation \eqref{eq:pushforward-regular-differentials}).
\end{proof}

In fact, the above isomorphism is also compatible with the natural product structures, yielding an isomorphism of dg-algebras over $\mathcal{O}_S$:
\[
    f_*\Omega^{\bullet}_{E^{\natural}/S} \cong H^{\bullet}_{\dR}(E^{\natural}/S).
\]
Since $\pi:E^{\natural} \to E$ is a $\mathbb{G}_a$-bundle, it follows from the Künneth formula that $\pi^*: H^{\bullet}_{\dR}(E/S) \to H^{\bullet}_{\dR}(E^{\natural}/S)$ is also an isomorphism of dg-algebras over $\mathcal{O}_S$. Thus, we obtain a canonical isomorphism
\begin{equation}\label{eq:isom-cohomology-E}
    f_*\Omega^{\bullet}_{E^{\natural}/S} \cong H^{\bullet}_{\dR}(E/S).
\end{equation}

\begin{remark}
    By \cite[I.4]{MM74}, applying the functor $\Lie_S$ to the exact sequence \eqref{eq:exact-uve} gives rise to a short exact sequence of $\mathcal{O}_S$-modules
    \[
        \begin{tikzcd}
            0\arrow{r}{}&(R^1p_*\mathcal{O}_E)^{\vee}\arrow{r}{}&\Lie_S E^\natural\arrow{r}{}&\Lie_S E\arrow{r}{}&0,
        \end{tikzcd}
    \]
    canonically isomorphic to the dual of the Hodge--de Rham short exact sequence
    \begin{equation}\label{eq:hodge-derham}
        \begin{tikzcd}
            0\arrow{r}{}&p_*\Omega^1_{E/S} \arrow{r}{}&H^1_{\dR}(E/S) \arrow{r}{}&R^1p_*\mathcal{O}_E \arrow{r}{}&0.
        \end{tikzcd}
    \end{equation}
    The composition of isomorphisms $f_*\Omega^{1}_{E^{\natural}/S} \cong (\Lie_SE^{\natural})^{\vee} \cong H^1_{\dR}(E/S)$ coincides with the isomorphism \eqref{eq:isom-cohomology-E}.
\end{remark}

The fact that the dg-algebra $f_*\Omega^{\bullet}_{E^{\natural}/S}$ is a model for the de Rham cohomology of $E/S$ can be generalised to punctured elliptic curves as follows.

\begin{proposition}\label{prop:model-derham}
    Let $Z\subsetneq E$ be a smooth closed $S$-subscheme. For every $n\ge 0$, the natural maps
    \[
        H^n(f_*\Omega^{\bullet}_{E^{\natural}/S}(\log \pi^{-1}Z)) \To H^n_{\dR}(\pi^{-1}(E\setminus Z)/S) \stackrel{\pi^*}{\longleftarrow} H^n_{\dR}((E\setminus Z)/S)
    \]
    are isomorphisms.
\end{proposition}

\begin{proof}
    That $\pi^*: H^n_{\dR}((E\setminus Z)/S) \to H^n_{\dR}(\pi^{-1}(E\setminus Z)/S)$ is an isomorphism follows from the Künneth formula and from the fact that $\pi: \pi^{-1}(E\setminus Z) \to E\setminus Z$ is a $\mathbb{G}_a$-bundle.

    By Deligne's theorem \cite[Ch. II, Corollaire 3.15, Remarque 3.16]{deligne70}, there is a canonical isomorphism
    \[
        \mathbb{R}^nf_*\Omega^{\bullet}_{E^{\natural}/S}(\log \pi^{-1}Z) \stackrel{\sim}{\To} H^n_{\dR}(\pi^{-1}(E\setminus Z)/S).
    \]
    We are left to prove that $\Omega^{\bullet}_{E^{\natural}/S}(\log \pi^{-1}Z)$ is a complex of $f_*$-acyclic sheaves. For this, let $i:\pi^{-1}Z \to E^{\natural}$ be the inclusion, and consider the Poincaré residue exact sequence
    \[
        \begin{tikzcd}
            0 \arrow{r} & \Omega^{\bullet}_{E^{\natural}/S} \arrow{r} & \Omega^{\bullet}_{E^{\natural}/S}(\log \pi^{-1}Z) \arrow{r}{\Res} & i_*\Omega^{\bullet}_{\pi^{-1}Z/S}[-1] \arrow{r} & 0.
        \end{tikzcd}
    \]
    As each $\Omega^n_{E^{\natural}/S}$ is $f_*$-acyclic by \eqref{eq:pushforward-regular-differentials}, and each $i_*\Omega^{n-1}_{\pi^{-1}Z/S}$ is $f_*$-acyclic by the fact that both $i$ and $f\circ i$ are affine, we conclude from long exact sequence in cohomology associated to the above short exact sequence that each $\Omega^n_{E^{\natural}/S}(\log \pi^{-1}Z)$ is $f_*$-acyclic.
\end{proof}

\section{Relative logarithmic differentials on the universal vector extension} \label{sec:relative-differentials}

\subsection{Kronecker differentials}\label{subsec:kronecker-differentials}

Recall that the sheaf of invariant differential 1-forms on the vector group $\mathbb{V}(R^1p_*\mathcal{O}_E)/S$ is canonically isomorphic to $R^1p_*\mathcal{O}_E$: a local section $t$ of $R^1p_*\mathcal{O}_E$ corresponds to the relative 1-form $dt$. Thus, under the identification \eqref{eq:isom-cohomology-E}, the Hodge--de Rham exact sequence \eqref{eq:hodge-derham} corresponds to the short exact sequence of sheaves of invariant differential forms
\begin{equation}\label{eq:hodge-derham-uve}
    \begin{tikzcd}
        0 \arrow{r} & p_*\Omega^1_{E/S} \arrow{r} & f_*\Omega^1_{E^{\natural}/S} \arrow{r} & R^1p_*\mathcal{O}_E\arrow{r} & 0,
    \end{tikzcd}
\end{equation}
where the left arrow is given by pullback by $\pi$, and the right arrow is given by restriction to $\mathbb{V}(R^1p_*\mathcal{O}_E) \cong \pi^{-1}O \hookrightarrow E^{\natural}$. 

Assume that $S$ is affine and that $R^1p_*\mathcal{O}_E$ is free with trivialisation $t$. Under the identification
\[
    f_*\mathcal{O}_{\pi^{-1}O} \cong \bigoplus_{n\ge 0}(R^1p_*\mathcal{O}_E)^{\otimes n} \cong \mathcal{O}_S[t],
\]
we obtain a commutative diagram
\begin{equation}\label{eq:residue-es}
    \begin{tikzcd}
        0 \arrow{r} & f_*\Omega^1_{E^{\natural}/S} \arrow{r}\arrow{d}{0} & f_*\Omega^1_{E^{\natural}/S}(\log \pi^{-1}O) \arrow{r}{\Res}\arrow{d}{d} & \mathcal{O}_S[t]\arrow{r}\arrow{d}{d} & 0\\
        0 \arrow{r} & f_*\Omega^2_{E^{\natural}/S} \arrow{r} & f_*\Omega^2_{E^{\natural}/S}(\log \pi^{-1}O) \arrow{r}{\Res} & \mathcal{O}_S[t]dt \arrow{r} & 0
    \end{tikzcd}
\end{equation}
where the rows are Poincaré residue exact sequences (cf. proof of Proposition \ref{prop:model-derham}).

\begin{lemma}\label{lemma:kronecker-differentials}
    Let $\alpha$ be a global section of $f_*\Omega^1_{E^{\natural}/S}(\log \pi^{-1}O)$ satisfying $\Res(\alpha)=1$. Then $d$ restricts to an isomorphism
    \[
        d: \mathcal{O}_S \alpha \stackrel{\sim}{\longrightarrow}f_*\Omega^2_{E^{\natural}/S}.
    \]
\end{lemma}

\begin{proof}
    Since $\Res(d\alpha) = d\Res(\alpha) = 0$, it follows from the residue exact sequence in degree 2 that $d\alpha$ is a section of the line bundle $f_*\Omega^2_{E^{\natural}/S}$, so that the map in the statement is well-defined. To verify that it is an isomorphism, it suffices to prove that it is surjective. Since $H^2(f_*\Omega^{\bullet}_{E^{\natural}/S}(\log \pi^{-1}O)) = 0$ by Proposition \ref{prop:model-derham}, for any section $\beta$ of $f_*\Omega^2_{E^{\natural}/S}$ there is a section $\alpha'$ of $f_*\Omega^1_{E^{\natural}/S}(\log \pi^{-1}O)$ satisfying $d\alpha' = \beta$. By the commutativity of \eqref{eq:residue-es}, $\Res(\alpha')$ is in $\mathcal{O}_S$, so that $\Res(\alpha')\alpha - \alpha'$ is in $f_*\Omega^1_{E^{\natural}/S}$. This shows that $d(\Res(\alpha')\alpha) = \beta$. 
\end{proof}

\begin{remark}\label{rmk:pairing}
    Note that $\alpha$ in the above statement is unique up to a section of $f_*\Omega^1_{E^{\natural}/S}$, so that $d\alpha$ is independent of any choice; it gives a canonical trivialisation of $f_*\Omega^2_{E^{\natural}/S}$. In particular, we obtain a symplectic pairing $\langle \ , \ \rangle : f_*\Omega^1_{E^{\natural}/S}\otimes f_*\Omega^1_{E^{\natural}/S} \to \mathcal{O}_S$ defined by $\langle \omega_1,\omega_2\rangle d\alpha = \omega_1\wedge \omega_2$, which induces by \eqref{eq:hodge-derham-uve} the classical isomorphism $p_*\Omega^1_{E/S} \cong (R^1p_*\mathcal{O}_E)^{\vee}$. Under the identification \eqref{eq:isom-cohomology-E}, the pairing $\langle \ , \ \rangle$ is the usual de Rham pairing on $H^1_{\dR}(E/S)$ (cf. \cite{coleman98}).
\end{remark}

\begin{theorem}\label{thm:kronecker-differentials}
    Assume that $S$ is affine, $R^1p_*\mathcal{O}_E = \mathcal{O}_St$ is free, and let $\nu$ be a global section of $f_*\Omega^1_{E^{\natural}/S}$ satisfying
    \[
        \nu |_{\pi^{-1}O} = dt.
    \]
    Then, there exists a unique family $\{\omega^{(n)}\}_{n \ge 0}$ of global sections of $f_*\Omega^1_{E^{\natural}/S}(\log \pi^{-1}O)$ such that $\omega^{(0)}$ trivialises $p_*\Omega^1_{E/S}$ and, for every $n\ge 1$:
    \begin{enumerate}[(i)]
        \item $\Res(\omega^{(n)}) = t^{n-1}/(n-1)!$,
        \item $\omega^{(n)}\wedge \omega^{(0)} = 0$, and
        \item $d\omega^{(n)} = \nu \wedge \omega^{(n-1)}$.
    \end{enumerate}
    Moreover, we have:
    \begin{equation}\label{eq:kronecker-1}
        f_*\Omega^1_{E^{\natural}/S}(\log \pi^{-1}O) = f_*\Omega^1_{E^{\natural}/S} \oplus  \bigoplus_{n\ge 1}\mathcal{O}_S\omega^{(n)}\text{, }\qquad f_*\Omega^1_{E^{\natural}/S} = \mathcal{O}_S\nu \oplus  \mathcal{O}_S\omega^{(0)}
    \end{equation}
    and
    \begin{equation}\label{eq:kronecker-2}
        f_*\Omega^2_{E^{\natural}/S}(\log \pi^{-1}O) = f_*\Omega^2_{E^{\natural}/S}\oplus \bigoplus_{n\ge 1}\mathcal{O}_S \nu \wedge \omega^{(n)}\text{, }\qquad f_*\Omega^2_{E^{\natural}/S} = \mathcal{O}_S \nu \wedge \omega^{(0)}. 
    \end{equation}
\end{theorem}

\begin{proof}
    Let $\alpha_0$ be any trivialisation of $p_*\Omega^1_{E/S}$ (note that $p_*\Omega^1_{E/S}$ is free by Remark \ref{rmk:pairing}), so that
    \begin{equation}\label{eq:triv-regular}
        f_*\Omega^1_{E^{\natural}/S} = \mathcal{O}_S\nu \oplus \mathcal{O}_S\alpha_0\text{, }\qquad f_*\Omega^2_{E^{\natural}/S}= \mathcal{O}_S \nu \wedge \alpha_0.
    \end{equation}
    It follows from the residue exact sequence \eqref{eq:residue-es} in degree 1 that, for every $n\ge 1$, there is a global section $\alpha_n$ of $f_*\Omega^1_{E^{\natural}/S}(\log \pi^{-1}O)$ such that
    \[
        \Res(\alpha_n) = \frac{t^{n-1}}{(n-1)!}.
    \]
    Our goal is to modify each $\alpha_n$ by a section of $f_*\Omega^1_{E^{\natural}/S}$ so that properties (ii) and (iii) are also satisfied.

    Since the restriction of $\alpha_0$ to $\pi^{-1}O$ vanishes by \eqref{eq:hodge-derham-uve}, the 2-form $\alpha_n\wedge \alpha_0$ is in $f_*\Omega^2_{E^{\natural}/S}$. Thus, by \eqref{eq:triv-regular}, up to adding a multiple of $\nu$ to $\alpha_n$, we can assume that
    \[
        \alpha_n\wedge \alpha_0=0.
    \]
    Let $n\ge 1$. Since
    \[
        \Res(d\alpha_n) = d\Res(\alpha_n) = d(t^{n-1}/(n-1)!) = t^{n-2}dt/(n-2)! = \Res(\nu \wedge \alpha_{n-1}),
    \]
    it follows from the residue exact sequence \eqref{eq:residue-es} in degree 2 and from \eqref{eq:triv-regular} that there exists a global section $r_{n-1}$ of $\mathcal{O}_S$ such that
    \[
        d \alpha_n = \nu \wedge \alpha_{n-1} + r_{n-1}\nu \wedge \alpha_0.
    \]
    For every $n\ge 0$, we set
    \[
        \omega^{(n)}\defeq \alpha_n + r_n\alpha_0.
    \]
    Thus, $\omega^{(0)}$ is a global section of $p_*\Omega^1_{E/S}$ and $\omega^{(n)}$ are global sections of $f_*\Omega^1_{E^{\natural}/S}(\log \pi^{-1}O)$ satisfying (i), (ii), and (iii). Furthermore, since $d\omega^{(1)} = \nu\wedge \omega^{(0)}$, it follows from Lemma \ref{lemma:kronecker-differentials} that $\nu \wedge \omega^{(0)}$ is a trivialisation of $f_*\Omega^2_{E^{\natural}/S}$, so that $\{\nu,\omega^{(0)}\}$ is a trivialisation of $f_*\Omega^1_{E^{\natural}/S}$. Then, \eqref{eq:kronecker-1} and \eqref{eq:kronecker-2} follow from the residue exact sequence \eqref{eq:residue-es}.

    To prove unicity, let $\{\lambda_n\}_{n\ge 0}$ be a family of global sections of $f_*\Omega^1_{E^{\natural}/S}(\log \pi^{-1}O)$ such that $\lambda_0$ trivialises $p_*\Omega^1_{E/S}$ and $\lambda_n$ satisfies (i), (ii), and (iii) for $n\ge 1$. Then, we can write $\lambda_0 = u \omega^{(0)}$ for some $u \in \Gamma(S,\mathcal{O}_S^{\times})$, and, by (i), for every $n\ge 1$,
    \[
        \lambda_n = \omega^{(n)} + a_n\omega^{(0)} + b_n \nu,
    \]
    for some $a_n,b_n \in \Gamma(S,\mathcal{O}_S)$. From (ii), we conclude that $b_n=0$. Finally, property (iii) applied to the family $\{\lambda_n\}_{n\ge 0}$ yields
    \[
        \nu \wedge \omega^{(n-1)} = 
        \begin{cases}
            u\nu\wedge \omega^{(0)} & n=1\\
            \nu \wedge \omega^{(n-1)} + a_{n-1}\nu\wedge \omega^{(0)} & n\ge 2
        \end{cases}
    \]
    which implies that $u=1$ and $a_n=0$ for every $n\ge 1$.
\end{proof}

\begin{remark}\label{eq:change-basis}
    It follows from unicity and from equations (i), (ii), and (iii), that a change in $\nu$ to $\nu'= u\nu + v\omega^{(0)}$, with $u \in \Gamma(S,\mathcal{O}_S^{\times})$ and $v \in \Gamma(S,\mathcal{O}_S)$, changes $\omega^{(n)}$ to ${\omega^{(n)}}' = u^{n-1}\omega^{(n)}$.
\end{remark}

\begin{corollary}\label{coro:local-kronecker}
    With notation as in Theorem \ref{thm:kronecker-differentials}, we have $\omega^{(n)}\wedge \omega^{(m)} = 0$ for every $n,m\ge 0$.
\end{corollary}

\begin{proof}
    Since $E^{\natural}/S$ is smooth of relative dimension 2, in a formal neighbourhood of any point in the smooth relative effective Cartier divisor $\pi^{-1}O$, we can find $S$-coordinates $(x,t)$ such that $dx = \omega^{(0)}$, $dt=\nu$, so that $\pi^{-1}O$ is given by $x=0$. Since $\omega^{(n)}$ has logarithimic singularities along $\pi^{-1}O$, there are power series $F_n(x,t)$, $G_n(x,t)$ with coefficients in $\mathcal{O}_S$ such that $\omega^{(n)} = F_n(x,t)\frac{dx}{x} + G_n(x,t)dt$. By equation (ii) of Theorem \ref{thm:kronecker-differentials}, we have $G_n(x,t) = 0$, so that $\omega^{(n)} = F_n(x,t)\frac{dx}{x}$. The statement follows immediately.
\end{proof}

\subsection{Kronecker subbundles} \label{par:kronecker-subbundle}

Let $Z\subsetneq E$ be a closed subscheme of $E$ given by the union of \emph{torsion} sections $P\in E(S)$. Assume moreover that $Z$ contains the identity section: $O \in Z(S)$.

\begin{example}
    We can always take $Z = O$. If $E/S$ admits a full level $N$ structure, we can also consider $Z = E[N]$, or any other subscheme thereof which is flat over $S$ and contains the identity.
\end{example}

By pulling $Z$ back by the canonical projection $\pi: E^{\natural} \to E$, we obtain a divisor $\pi^{-1}Z$ on $E^{\natural}$, which can be written as a disjoint union:
\[
    \pi^{-1}Z = \bigsqcup_{P \in Z(S)}\pi^{-1}P.
\]
Since $P \in Z(S)$ is torsion and $\pi:E^{\natural} \to E$ is a $\mathbb{G}_a$-bundle, $P$ admits a unique lift to a torsion section $P^{\natural} \in E^{\natural}(S)$ \cite[Lemma C.1.1]{katz77}. Thus, we obtain a canonical trivialisation
\begin{equation}\label{eq:triv-D}
    \pi^{-1}Z \stackrel{\sim}{\To} Z \times_S \mathbb{V}(R^1p_*\mathcal{O}_E),
\end{equation}
given on the component $\pi^{-1}P$ by $x \mapsto (P,x - P^{\natural})$. This yields an isomorphism
\begin{equation}\label{eq:triv-OD}
    f_*\mathcal{O}_{\pi^{-1}Z} \cong p_*\mathcal{O}_Z \otimes \bigoplus_{n\ge 0}(R^1p_*\mathcal{O}_E)^{\otimes n} \cong \bigoplus_{n\ge 0}p_*\mathcal{O}_Z \otimes(R^1p_*\mathcal{O}_E)^{\otimes n},
\end{equation}
where tensor products are taken over $\mathcal{O}_S$.

\begin{theorem}\label{thm:kronecker-subbundle}
    There is a canonical decomposition
    \[
        f_*\Omega^1_{E^{\natural}/S}(\log \pi^{-1}Z) = f_*\Omega^1_{E^{\natural}/S} \oplus \bigoplus_{n\ge 1}\mathcal{K}^{(n)}
    \]
    where $\mathcal{K}^{(n)}$ are subbundles of $f_*\Omega^1_{E^{\natural}/S}(\log \pi^{-1}Z)$ uniquely determined by the following properties:
    \begin{enumerate}[(i)]
        \item Under the identification \eqref{eq:triv-OD}, the residue map
        \[
            \Res : f_*\Omega^1_{E^{\natural}/S}(\log \pi^{-1}Z) \longrightarrow f_*\mathcal{O}_{\pi^{-1}Z}
        \]
        restricts to an isomorphism between $\mathcal{K}^{(n)}$ and $p_*\mathcal{O}_Z\otimes (R^1p_*\mathcal{O}_E)^{\otimes (n-1)}$;
        \item $\mathcal{K}^{(n)}\wedge \mathcal{K}^{(0)}=0$ in $f_*\Omega^2_{E^{\natural}/S}(\log \pi^{-1}Z)$, where
        \[
            \mathcal{K}^{(0)} \defeq p_*\Omega^1_{E/S}
        \]
        is seen as a rank 1 subbundle of $f_*\Omega^1_{E^{\natural}/S}(\log \pi^{-1}Z)$ via pullback by $\pi$;
        \item $d\mathcal{K}^{(n)} = f_*\Omega^1_{E^{\natural}/S}\wedge \mathcal{K}^{(n-1)}$ in $f_*\Omega^2_{E^{\natural}/S}(\log \pi^{-1}Z)$.
    \end{enumerate}
    Moreover, we have
    \[
        f_*\Omega^2_{E^{\natural}/S}(\log \pi^{-1}Z) = \bigoplus_{n\ge 0} f_*\Omega^1_{E^{\natural}/S}\wedge \mathcal{K}^{(n)}.
    \]
\end{theorem}

\begin{proof}
    It suffices to prove existence and unicity locally over $S$. Thus, we may assume that $S$ is affine and that $R^1p_*\mathcal{O}_E=\mathcal{O}_St$ is free. Let $\nu,\omega^{(0)},\omega^{(1)},\ldots$ be as in Theorem \ref{thm:kronecker-differentials}. For any $P \in Z(S)$, let $\tau_{-P^{\natural}}: E^{\natural}\to E^{\natural}$ be the translation by $-P^{\natural}$, and set, for every $n\ge 0$
    \begin{equation}\label{eq:defn-omegaP}
        \omega^{(n)}_{P} \defeq \tau^*_{-P^{\natural}}\omega^{(n)}  \in \Gamma(S, f_*\Omega^1_{E^{\natural}/S}(\log \pi^{-1}P)).
    \end{equation}
    Let us also denote by $t_P =  \tau^*_{-P^{\natural}}t$ the coordinate on $\pi^{-1}P$ induced by $t$ under the isomorphism \eqref{eq:triv-D}. Since $\nu, \omega^{(0)}$ are invariant --- in particular, $\omega_P^{(0)} = \omega^{(0)}$ --- it follows from Theorem \ref{thm:kronecker-differentials} that, for every $n\ge 1$,
    \begin{itemize}
        \item[(a)] $\Res(\omega^{(n)}_P) = t_P^{n-1}/(n-1)!$,
        \item[(b)] $\omega^{(n)}_P\wedge \omega^{(0)} = 0$, and 
        \item[(c)] $d\omega_P^{(n)} = \nu \wedge \omega^{(n-1)}_{P}$.
    \end{itemize}
    For $n\ge 1$, set
    \[
        \mathcal{K}^{(n)} \defeq \bigoplus_{P \in Z(S)}\mathcal{K}^{(n)}_P\text{, }\qquad \mathcal{K}^{(n)}_P \defeq \mathcal{O}_S \omega^{(n)}_P.
    \]
    Then, the residue exact sequence, together with (a), (b), and (c), show that $\mathcal{K}^{(n)}$ so defined satisfies all the properties in the statement.
  
    To prove unicity, let $\{\mathcal{L}_n\}_{n\ge 1}$ be another family of subbundles of $f_*\Omega^1_{E^{\natural}/S}(\log \pi^{-1}Z)$ satisfying (i), (ii), and (iii). By (i), $\mathcal{L}_n$ is trivialised by 1-forms $\lambda_{n,P}$, $P \in Z(S)$, satisfying (a). Then, by the residue exact sequence, there are global sections $a_{n,P}$, $b_{n,P}$ of $\mathcal{O}_S$ such that
    \[
        \lambda_{n,P} = \omega^{(n)}_P + a_{n,P} \omega^{(0)} + b_{n,P}\nu. 
    \]
    It follows from (ii) that $\lambda_{n,P}$ satisfies (b), so that $b_{n,P} = 0$. By (iii), we have
    \[
        d\lambda_{n,P} = c_{n,P}\nu \wedge \lambda_{n-1,P}
    \]
    for some global section $c_{n,P}$ of $\mathcal{O}_S$, where we set $\lambda_{0,P}\defeq \omega^{(0)}$. Thus, also using (c) for $\omega^{(n)}_P$, we obtain, for $n=1$,
    \[
        c_{1,P}\nu \wedge \omega^{(0)} = \nu \wedge \omega^{(0)},
    \]
    and, for $n\ge 2$,
    \[
        c_{n,P}\nu \wedge \omega^{(n-1)}_P + c_{n,P}a_{n-1,P}\nu \wedge \omega^{(0)} = \nu \wedge \omega^{(n-1)}_P.
    \]
    Thus, $c_{1,P} = 1$, and, by taking residues along $\pi^{-1}P$, we see that $c_{n,P} = 1$ for $n\ge 2$. Then, it follows from the same equation that $a_{n-1,P}= 0$. We conclude that $\lambda_{n,P}= \omega^{(n)}_P$, which yields $\mathcal{L}_n = \mathcal{K}^{(n)}$. 
\end{proof}

We shall always denote by
\[
    \mathcal{K}^{(n)} = \bigoplus_{P \in Z(S)} \mathcal{K}^{(n)}_P
\]
the decomposition induced by part (i) of Theorem \ref{thm:kronecker-subbundle} and $p_*\mathcal{O}_Z \cong \bigoplus_{P \in Z(S)}\mathcal{O}_S$. When $R^1p_*\mathcal{O}_E$ is free, a choice of $\nu$ as in Theorem \ref{thm:kronecker-differentials} induces trivialisations $\omega^{(n)}_P$ of $\mathcal{K}^{(n)}_P$ as in the proof of the above theorem, so that
\begin{equation}\label{eq:triv-A1}
    f_*\Omega^1_{E^{\natural}/S}(\log \pi^{-1}Z) = \mathcal{O}_S \nu \oplus \mathcal{O}_S \omega^{(0)} \oplus \bigoplus_{\substack{n\ge 1 \\ P \in Z(S)}}\mathcal{O}_S \omega^{(n)}_P
\end{equation}
\begin{equation}\label{eq:triv-A2}
    f_*\Omega^2_{E^{\natural}/S}(\log \pi^{-1}Z) = \mathcal{O}_S \nu \wedge \omega^{(0)}\oplus \bigoplus_{\substack{n\ge 1 \\ P \in Z(S)}}\mathcal{O}_S \nu \wedge \omega^{(n)}_P.
\end{equation}

\subsection{The dg-algebra of relative logarithmic differentials}\label{par:dg-algebra-relative}

With same hypotheses and notation as in the last sections, consider the dg-algebra over $\mathcal{O}_S$
\[
    \mathcal{A} \defeq f_*\Omega^{\bullet}_{E^{\natural}/S}(\log \pi^{-1}Z),
\]
and consider the submodule of $\mathcal{A}^1$
\[
    \mathcal{K} = \bigoplus_{n\ge 0}\mathcal{K}^{(n)}.
\]
It follows from Theorem \ref{thm:kronecker-subbundle} that the dg-algebra $\mathcal{A}$ is generated by its degree 1 sections. More precisely, we have the following structure result.

\begin{corollary}\label{coro:dg-algebra}
    $\mathcal{A}$ is locally free as an $\mathcal{O}_S$-module, and the wedge product induces an isomorphism of graded $\mathcal{O}_S$-algebras
    \[
        {\bigwedge}^{\bullet} \mathcal{A}^1 / \langle {\bigwedge}^2 \mathcal{K} \rangle \stackrel{\sim}{\To} \mathcal{A}.
    \]
\end{corollary}

\begin{proof}
    That $\mathcal{A}$ is locally free as an $\mathcal{O}_S$-module follows Theorem \ref{thm:kronecker-subbundle} and from the fact that $R^1p_*\mathcal{O}_E$ is locally free (cf. \eqref{eq:triv-A1} and \eqref{eq:triv-A2}).
  
    It follows from Theorem \ref{thm:kronecker-subbundle} (see also Corollary \ref{coro:local-kronecker}) that the wedge product gives a well-defined morphism of graded $\mathcal{O}_S$-algebras ${\bigwedge}^{\bullet} \mathcal{A}^1 / \langle {\bigwedge}^2 \mathcal{K}\rangle \to \mathcal{A}$. This is an isomorphism in degree 0 by Theorem \ref{thm:coleman-laumon}, and is trivially an isomorphism in degree 1. To see that it is also an isomorphism in higher degrees, we may argue locally over $S$. Thus, we may assume that $\mathcal{A}^1/\mathcal{K} \cong R^1p_*\mathcal{O}_E$ is free (of rank 1), and that the short exact sequence
    \[
        \begin{tikzcd}
            0 \arrow{r} & \mathcal{K} \arrow{r} & \mathcal{A}^1 \arrow{r} & \mathcal{A}^1/\mathcal{K} \arrow{r} & 0
        \end{tikzcd}
    \]
    splits:
    \[
        \mathcal{A}^1 \cong (\mathcal{A}^1/\mathcal{K}) \oplus \mathcal{K}.
    \]
    This implies that
    \[
        {\bigwedge}^2\mathcal{A}^1 \cong ((\mathcal{A}^1/\mathcal{K})\otimes \mathcal{K}) \oplus {\bigwedge}^2 \mathcal{K} \cong \mathcal{A}^2 \oplus {\bigwedge}^2 \mathcal{K},
    \]
    and, for $n\ge 3$,
    \[
        {\bigwedge}^n\mathcal{A}^1 \cong ((\mathcal{A}^1/\mathcal{K}) \otimes {\bigwedge}^{n-1}\mathcal{K}) \oplus {\bigwedge}^n \mathcal{K} \subset \im ({\bigwedge}^{n-2}\mathcal{A}^1 \otimes {\bigwedge}^2\mathcal{K} \longrightarrow {\bigwedge}^n \mathcal{A}^1).\qedhere
    \]
\end{proof}

Our next result concerns the cohomology of $\mathcal{A}$. It follows from $\mathcal{A}^0= \mathcal{O}_S$ (Theorem \ref{thm:coleman-laumon}) that $H^0(\mathcal{A}) = \mathcal{O}_S$ and that
\[
    H^1(\mathcal{A}) = \ker(d: \mathcal{A}^1 \To \mathcal{A}^2)\subset \mathcal{A}^1.
\]
Moreover, it follows from Proposition \ref{prop:model-derham} or from Theorem \ref{thm:kronecker-subbundle} that $H^n(\mathcal{A}) = 0$ for $n\ge 2$.

\begin{theorem}\label{thm:projector}
    We have a canonical decomposition of $\mathcal{O}_S$-modules
    \[
        \mathcal{A}^1 =  H^1(\mathcal{A}) \oplus \mathcal{K}^{(1)}_O \oplus \bigoplus_{n\ge 2}\mathcal{K}^{(n)}.
    \]
    Let $\rho^1: \mathcal{A}^1 \to H^1(\mathcal{A})$ be the projector induced by the above decomposition, and
    \[
        \rho: \mathcal{A} \To H^{\bullet}(\mathcal{A})
    \]
    be the induced morphism of graded $\mathcal{O}_S$-algebras ($\rho^0 = \id$, and $\rho^n=0$ for $n\ge 2$). Then, $\rho$ is a dg-quasi-isomorphism.
\end{theorem}

\begin{proof}
    The second assertion follows immediately from the first (cf. Proposition \ref{prop:model-derham}). To prove the decomposition, we may work locally over $S$. Let $\nu,\omega^{(0)},\omega_P^{(1)},\omega^{(2)}_P,\ldots$ be as in \eqref{eq:triv-A1}. It follows from the equations
    \begin{align*}
        d\omega^{(0)} = d\nu = 0\text{, } \qquad d\omega^{(1)}_P = \nu \wedge \omega^{(0)}\text{, }\qquad d\omega^{(n)}_P = \nu \wedge \omega^{(n-1)}_P \qquad (n\ge 2), 
    \end{align*}
    that
    \[
        H^1(\mathcal{A}) = \mathcal{O}_S\nu \oplus \mathcal{O}_S\omega^{(0)}\oplus \bigoplus_{P \in Z(S)\setminus\{O\}}\mathcal{O}_S(\omega^{(1)}_P - \omega^{(1)}_O).
    \]
    To conclude, we simply remark that
    \[
        \mathcal{K}^{(1)} = \bigoplus_{P \in Z(S)}\mathcal{O}_S \omega^{(1)}_P = \mathcal{O}_S\omega^{(1)}_O \oplus\bigoplus_{P \in Z(S)\setminus\{O\}}\mathcal{O}_S(\omega^{(1)}_P - \omega^{(1)}_O). \qedhere
    \]
\end{proof}

\section{Canonical lifts of relative logarithmic differentials}\label{sec:canonical-lifts}

\subsection{Canonical lifts of regular differentials}

Keep the same notation and hypotheses of \S\ref{par:kronecker-subbundle}, and assume moreover that $S$ is smooth over a field $k$ of characteristic zero.

We review some of the results of \cite{FM22}. It follows from the smoothness of $f:E^{\natural}\to S$, that we have an exact sequence of $\mathcal{O}_{E^{\natural}}$-modules
\[
    \begin{tikzcd}
        0 \arrow{r} & f^*\Omega^1_{S/k} \arrow{r} & \Omega^1_{E^{\natural}/k} \arrow{r} & \Omega^1_{E^{\natural}/S} \arrow{r} & 0.
    \end{tikzcd}
\]
By Theorem \ref{thm:coleman-laumon}, we obtain an exact sequence of $\mathcal{O}_S$-modules
\[
    \begin{tikzcd}
        0 \arrow{r} & \Omega^1_{S/k} \arrow{r} & f_*\Omega^1_{E^{\natural}/k} \arrow{r} & f_*\Omega^1_{E^{\natural}/S} \arrow{r} & 0,
    \end{tikzcd}
\]
which admits a canonical splitting, as follows.

\begin{proposition}
    The pullback by the identity section $e\in E^{\natural}(S)$ induces a retraction $e^* : f_*\Omega^1_{E^{\natural}/k} \to \Omega^1_{S/k}$. In particular, if $\mathcal{N} \defeq \ker (e^*)$, we obtain a canonical decomposition
    \[
        f_*\Omega^1_{E^{\natural}/k} = \Omega^1_{S/k} \oplus \mathcal{N}\text{, }\qquad \mathcal{N} \cong f_*\Omega^1_{E^{\natural}/S}.
    \]
\end{proposition}

\begin{proof}
    See  \cite[Theorem 3.2]{FM22}.
\end{proof}

It follows that any section $\omega$ of $f_*\Omega^1_{E^{\natural}/S}$ lifts canonically to a section of $\mathcal{N}\subset f_*\Omega^1_{E^{\natural}/k}$, which we denote by $\widetilde{\omega}$. For the next result, recall that $H^1_{\dR}(E/S)$ is endowed with an integrable $k$-connection, the \emph{Gauss--Manin connection}. Under the identification \eqref{eq:isom-cohomology-E}, we get an integrable $k$-connection
\[
    \nabla : f_*\Omega^1_{E^{\natural}/S} \To \Omega^1_{S/k}\otimes f_*\Omega^1_{E^{\natural}/S}.
\]

\begin{proposition}
    There is a canonical isomorphism $\Omega^1_{S/k}\otimes f_*\Omega^1_{E^{\natural}/S} \cong f_*\Omega^2_{E^{\natural}/k} / \Omega^2_{S/k}$, under which we have
    \[
        \nabla \omega = d \widetilde{\omega} \mod \Omega^2_{S/k}
    \]
    for any section $\omega$ of $f_*\Omega^1_{E^{\natural}/S}$.
\end{proposition}

\begin{proof}
    See \cite[Proposition 2.6]{FM22} and \cite[Proposition 3.5]{FM22}.
\end{proof}

\begin{example}\label{ex:gauss--manin}
    If $\omega^{(0)},\nu$ is a trivialisation of $f_*\Omega^1_{E^{\natural}/S}$ as in \S\ref{subsec:kronecker-differentials}, and $\alpha_{ij} \in \Gamma(S,\Omega^1_{S/k})$ are defined by
    \begin{align*}
        \nabla \omega^{(0)} &= \alpha_{11}\otimes \omega^{(0)} + \alpha_{21}\otimes \nu \\
        \nabla \nu &= \alpha_{12} \otimes \omega^{(0)} + \alpha_{22}\otimes \nu,
    \end{align*}
    then the canonical lifts satisfy (cf. \cite[Remark 3.7]{FM22}):
    \begin{equation}\label{eq:GM-canonical-lifts}
        \begin{aligned}
            d \widetilde{\omega}^{(0)} &= \alpha_{11}\wedge \widetilde{\omega}^{(0)} + \alpha_{21}\wedge \widetilde{\nu} \\
            d \widetilde{\nu} &= \alpha_{12} \wedge \widetilde{\omega}^{(0)} + \alpha_{22}\wedge \widetilde{\nu}.
        \end{aligned}
    \end{equation}
    Note that the above equations hold `on the nose', and not only modulo $\Omega^2_{S/k}$. 
\end{example}

For later reference, we consider the following auxiliary results.

\begin{lemma}\label{lem:restric-D}
    Assume that $S$ is affine, that $R^1p_*\mathcal{O}_E$ is trivial, and let $\omega^{(0)},\nu$, and $t$ be as in Theorem \ref{thm:kronecker-differentials}. Then:
    \[
        \widetilde{\omega}^{(0)}|_{\pi^{-1}O}=-\alpha_{21}t, \qquad \widetilde{\nu}|_{\pi^{-1}O}=dt-\alpha_{22}t.
    \]
\end{lemma}

\begin{proof}
    As $\pi^{-1}O$ is isomorphic to $\mathbb{G}_{a,S}$ via the coordinate $t$, we have
    \[
        \Gamma(S,f_*\Omega^1_{\pi^{-1}O/k}) \cong \Gamma(S,\Omega^1_{S/k})[t] + \Gamma(S,\mathcal{O}_S)[t]dt
    \]
    Thus, since $\omega^{(0)}|_{\pi^{-1}O}=0$ and $\nu|_{\pi^{-1}O}=dt$, we can write
    \begin{equation} \label{eq:restrict-D}
        \widetilde{\omega}^{(0)}\vert_{\pi^{-1}O}=\sum_{n\geq 0}\gamma_nt^n, \qquad \widetilde{\nu}\vert_{\pi^{-1}O}=dt+\sum_{n\geq 0}\delta_nt^n.
    \end{equation}
    for unique $\gamma_n,\delta_n \in \Gamma(S,\Omega^1_{S/k})$ with $\gamma_n = \delta_n = 0$ for $n\gg 0$. The condition $e^*\widetilde{\omega}^{(0)} = e^*\widetilde{\nu} = 0$ implies that $\gamma_0=\delta_0=0$. Plugging \eqref{eq:restrict-D} into \eqref{eq:GM-canonical-lifts}, a short calculation shows that $\gamma_n=\delta_n =0$ for $n\ge 2$, and that $\gamma_1=-\alpha_{21}$ and $\delta_1=-\alpha_{22}$.
\end{proof}

\begin{lemma}\label{lem:invariance}
    Let $\widetilde{\omega} \in \Gamma(S,\mathcal{N})$ and $Q \in E^{\natural}(S)$ be a torsion section. If $\tau_Q: E^{\natural} \to E^{\natural}$ denotes the translation by $Q$, then $\tau_Q^*\widetilde{\omega} = \widetilde{\omega}$. In particular, global sections of $\mathcal{N}$ are invariant under translation by $P^{\natural}\in E^{\natural}(S)$ for every $P \in Z(S)$.
\end{lemma}

\begin{proof}
    Since every section of $f_*\Omega^1_{E^{\natural}/S}$ is invariant (see \S\ref{par:cohomology}), there exists $\beta \in \Gamma(S,\Omega^1_{S/k})$ such that $\tau^*_{Q}\widetilde{\omega} = \widetilde{\omega} + \beta$. By induction, for any $n \ge 1$, we obtain
    \[
        \tau^*_{nQ}\widetilde{\omega} = \widetilde{\omega} + n\beta.
    \]
    By pulling back the above equation by $e$, we get
    \[
        (nQ)^*\widetilde{\omega} = e^*\widetilde{\omega} + n\beta = n\beta,
    \]
    where we have also used that $e^*\widetilde{\omega} = 0$. Since $Q$ is torsion and $S$ is of characteristic zero, we conclude that $\beta=0$.
\end{proof}

\subsection{Canonical lifts of Kronecker differentials} \label{par:canonical-lifts}

Define the \emph{Koszul filtration} $\mathcal{F}^0\supset \mathcal{F}^1 \supset  \cdots $ on the complex $\Omega^{\bullet}_{E^{\natural}/k}(\log \pi^{-1}Z)$ by
\[
    \mathcal{F}^p \defeq \im (f^*\Omega^p_{S/k} \otimes \Omega^{\bullet}_{E^{\natural}/k}(\log \pi^{-1}Z)[-p] \To \Omega^{\bullet}_{E^{\natural}/k}(\log \pi^{-1}Z)), 
\]
where the above map is given by the wedge product. This gives rise to a filtration $f_*\mathcal{F}^{p}$ on $f_*\Omega^{\bullet}_{E^{\natural}/k}(\log \pi^{-1}Z)$ by direct image on each component.

\begin{proposition}\label{prop:koszul-log}
    For every $p\ge 0$, there is a canonical isomorphism of complexes of $\mathcal{O}_S$-modules
    \[
        f_*\mathcal{F}^p/f_*\mathcal{F}^{p+1} \cong \Omega^p_{S/k}\otimes f_*\Omega^{\bullet}_{E^{\natural}/S}(\log \pi^{-1}Z)[-p].
    \]
\end{proposition}

\begin{proof}
    Using that
    \[
        \begin{tikzcd}
            0 \arrow{r} & f^*\Omega^1_{S/k} \arrow{r} & \Omega^1_{E^{\natural}/k}(\log \pi^{-1}Z) \arrow{r} & \Omega^1_{E^{\natural}/S}(\log \pi^{-1}Z) \arrow{r} & 0
        \end{tikzcd}
    \]
    is locally split, and that $\Omega^{\bullet}_{E^{\natural}/k}(\log \pi^{-1}Z) = \bigwedge^{\bullet}\Omega^1_{E^{\natural}/k}(\log \pi^{-1}Z)$, we see that the Koszul filtration satisfies
    \begin{equation} \label{eqn:KoszulGraded}
        \mathcal{F}^p/\mathcal{F}^{p+1}\cong f^*\Omega^p_{S/k}\otimes \Omega^\bullet_{E^\natural/S}(\log \pi^{-1}Z)[-p], \qquad \mbox{for every } p\geq 0.
    \end{equation}
    Thus, by the projection formula, it suffices to check that each $\mathcal{F}^p$ is a complex of $f_*$-acyclic $\mathcal{O}_{E^\natural}$-modules. This statement is local on $S$, so we may assume that $S\rightarrow \Spec k$ is finitely presented, which implies in particular that $\mathcal{F}^{N}=0$ for some $N \ge 0$. We prove the desired assertion, which is trivially true for $p\ge N$, by descending induction on $p$. By \eqref{eqn:KoszulGraded}, we have a short exact sequence
    \[
        \begin{tikzcd}
            0 \arrow{r} & \mathcal{F}^{p+1} \arrow{r} & \mathcal{F}^p \arrow{r} & f^*\Omega^p_{S/k}\otimes \Omega^{\bullet}_{E^{\natural}/S}(\log \pi^{-1}Z)[-p] \arrow{r} & 0.
        \end{tikzcd}
    \]
    Note that $f^*\Omega^p_{S/k}\otimes \Omega^{\bullet}_{E^{\natural}/S}(\log \pi^{-1}Z)[-p]$ is a complex of $f_*$-acyclic $\mathcal{O}_{E^{\natural}}$-modules by the same argument in the proof of Proposition \ref{prop:model-derham}. Thus, if $\mathcal{F}^{p+1}$ is also a complex of $f_*$-acyclic $\mathcal{O}_{E^{\natural}}$-modules, so is $\mathcal{F}^p$ by the long exact sequence in cohomology. 
\end{proof}

\begin{theorem}\label{thm:lift-kronecker}
    Assume that $S$ is affine, that $R^1p_*\mathcal{O}_E$ is free, and let $\nu,\omega^{(0)},\omega^{(1)},\ldots$ be as in Theorem \ref{thm:kronecker-differentials}. There is a unique family $\{\widetilde{\omega}^{(n)}\}_{n\ge 1}$ of global sections of $f_*\Omega^1_{E^{\natural}/k}(\log \pi^{-1}O)$ such that, for every $n\ge 1$, $\widetilde{\omega}^{(n)}$ is a lift of $\omega^{(n)}$ satisfying
    \[
        \widetilde{\omega}^{(n)}\wedge \widetilde{\nu}\wedge \widetilde{\omega}^{(0)} \equiv n \alpha_{21}\wedge \widetilde{\nu}\wedge \widetilde{\omega}^{(n+1)} \mod f_*\mathcal{F}^2,
    \]
    where $\alpha_{21} \in \Gamma(S,\Omega^1_{S/k})$ is a coefficient of the Gauss--Manin connection as in Example \ref{ex:gauss--manin}.
\end{theorem}

\begin{proof}
    We first prove unicity. Any other lift of $\omega^{(n)}$ would be of the form $\widetilde{\omega}^{(n)}+\beta_n$, for some $\beta_n \in \Gamma(S,\Omega^1_{S/k})$. By the condition in the statement, we deduce that
    \[
        \beta_n\wedge \widetilde{\nu}\wedge \widetilde{\omega}^{(0)} \equiv 0 \mod f_*\mathcal{F}^2.
    \]
    In other words, the image of $\beta_n\wedge \widetilde{\nu}\wedge \widetilde{\omega}^{(0)}$ in $f_*\mathcal{F}^{1}/f_*\mathcal{F}^2$ vanishes. By Proposition \ref{prop:koszul-log}, this means that $\beta_n \otimes \nu \wedge \omega^{(0)}=0$ in $\Omega^1_{S/k}\otimes f_*\Omega^{2}_{E^{\natural}/S}(\log \pi^{-1}O)$, which implies that $\beta_n=0$, since $\nu\wedge \omega^{(0)}, \nu \wedge \omega^{(1)}, \ldots $ is a trivialisation of $f_*\Omega^2_{E^{\natural}/S}(\log \pi^{-1}O)$ by Theorem \ref{thm:kronecker-differentials}.

    For the existence, let $\varphi_n \in \Gamma(S,f_*\Omega^1_{E^\natural/k}(\log \pi^{-1}O))$ be any lift of $\omega^{(n)}$. By Proposition \ref{prop:koszul-log} and the fact that $f_*\Omega^2_{E^\natural/S}(\log \pi^{-1}O)$ is trivialised by $\nu\wedge\omega^{(0)},\nu\wedge\omega^{(1)},\ldots$, there exist unique $\gamma_{n,i} \in \Gamma(S,\Omega^1_{S/k})$ such that
    \[
        \varphi_n\wedge\widetilde{\nu}\wedge\widetilde{\omega}^{(0)} \equiv \sum_{i\geq 0}\gamma_{n,i}\wedge \widetilde{\nu}\wedge\varphi_i \mod f_*\mathcal{F}^2.
    \]
    We take residues along $\pi^{-1}O$ on both sides of the above equation. By Theorem \ref{thm:kronecker-differentials} and Lemma \ref{lem:restric-D}, we get on the one hand
    \[
        \operatorname{Res}(\varphi_n\wedge\widetilde{\nu}\wedge\widetilde{\omega}^{(0)})=\frac{t^n}{(n-1)!}\alpha_{21}\wedge dt+\frac{t^{n+1}}{(n-1)!}\alpha_{22}\wedge\alpha_{21},
    \]
    and, on the other hand,
    \[
        \operatorname{Res}(\gamma_{n,i}\wedge\widetilde{\nu}\wedge\varphi_i)=
        \begin{cases}
            0 & i=0\\
            \frac{t^{i-1}}{(i-1)!}\gamma_{n,i}\wedge dt+\frac{t^i}{(i-1)!}\alpha_{22}\wedge\gamma_{n,i} & i\geq 1.
        \end{cases}
    \]
    In particular, we get
    \[
        \frac{t^n}{(n-1)!}\alpha_{21}\wedge dt=\sum_{i\geq 1}\frac{t^{i-1}}{(i-1)!}\gamma_{n,i}\wedge dt,
    \]
    so that $\gamma_{n,n+1} = n\alpha_{21}$ and $\gamma_{n,i} = 0$ for $i \not\in \{0,n+1\}$. We conclude that
    \[
        \varphi_n\wedge\widetilde{\nu}\wedge\widetilde{\omega}^{(0)}\equiv \gamma_{n,0}\wedge \widetilde{\nu}\wedge\widetilde{\omega}^{(0)}+n\alpha_{21}\wedge\widetilde{\nu}\wedge\varphi_{n+1} \mod f_*\mathcal{F}^2.
    \]
    Thus, $\widetilde{\omega}^{(n)}\defeq \varphi_n-\gamma_{n,0}$ are lifts of $\omega^{(n)}$ satisfying the equation in the statement.  
\end{proof}

\begin{remark}
    If we consider $\nu' = u\nu + v\omega^{(0)}$ as in Remark \ref{eq:change-basis}, then, by unicity, the canonical lifts of the corresponding Kronecker differentials are $\widetilde{\omega}^{(n)}{}' = u^{n-1}\widetilde{\omega}^{(n)}$.
\end{remark}

\begin{theorem}\label{thm:canonical-lift-kronecker-subbundle}
    Let $\mathcal{K}^{(n)}$ be the Kronecker subbundles of $f_*\Omega^1_{E^{\natural}/S}(\log \pi^{-1}Z)$ defined in Theorem \ref{thm:kronecker-subbundle}. For every $n\ge 1$, there is a unique lift of $\mathcal{K}^{(n)}$ to a subbundle $\mathcal{L}^{(n)}$ of $f_*\Omega^1_{E^{\natural}/k}(\log \pi^{-1}Z)$ such that
    \begin{equation}\label{eq:defn-Ln}
        \mathcal{L}^{(n)}\wedge \mathcal{N} \wedge \mathcal{N} \equiv d \mathcal{N} \wedge \mathcal{L}^{(n+1)} \mod f_*\mathcal{F}^2.
    \end{equation}
    Moreover, we have
    \[
        f_*\Omega^1_{E^{\natural}/k}(\log \pi^{-1}Z) = \Omega^1_{S/k} \oplus \mathcal{N} \oplus \bigoplus_{n\ge 1}\mathcal{L}^{(n)}.
    \]
\end{theorem}

\begin{proof}
    We may work locally over $S$, so that the hypotheses of Theorem \ref{thm:lift-kronecker} are satisfied. The proof is similar to that of Theorem \ref{thm:kronecker-subbundle}: for every $P \in Z(S)$, we set
    \begin{equation}\label{eq:def-omegaP-lift}
        \widetilde{\omega}^{(n)}_P\defeq \tau_{-P^{\natural}}^*\widetilde{\omega}^{(n)}.
    \end{equation}
    Using that $\widetilde{\nu}$ and $\widetilde{\omega}^{(0)}$ are invariant under translation by $P^{\natural}$ (Lemma \ref{lem:invariance}), we obtain
    \[
        d\widetilde{\omega}^{(n)}_P\wedge \widetilde{\nu}\wedge \widetilde{\omega}^{(0)} \equiv n \alpha_{21}\wedge \widetilde{\nu}\wedge \widetilde{\omega}_P^{(n+1)} \mod f_*\mathcal{F}^2.
    \]
    Thus, the subbundles
    \[
        \mathcal{L}^{(n)} \defeq \bigoplus_{P \in Z(S)}\mathcal{L}_P^{(n)}\text{, }\qquad \mathcal{L}_P^{(n)} = \mathcal{O}_S \widetilde{\omega}_P^{(n)},
    \]
    satisfy \eqref{eq:defn-Ln}. The second assertion follows immediately from the exactness of
    \[
        \begin{tikzcd}
            0 \arrow{r} & \Omega^1_{S/k} \arrow{r} & f_*\Omega^1_{E^{\natural}/k}(\log \pi^{-1}Z) \arrow{r} & f_*\Omega^1_{E^{\natural}/S}(\log \pi^{-1}Z) \arrow{r} & 0
        \end{tikzcd}
    \]
    and from Theorem \ref{thm:kronecker-subbundle}.

    For unicity, let $\mathcal{L}^{(n)}{}'$ be another family of subbundles of $f_*\Omega^1_{E^{\natural}/S}(\log \pi^{-1}Z)$ satisfying \eqref{eq:defn-Ln}. Since $\mathcal{L}^{(n)}{}'$ is isomorphic to $\mathcal{K}^{(n)} = \bigoplus_{P \in Z(S)}\mathcal{K}^{(n)}_P$, we have a decomposition $\mathcal{L}^{(n)}{}' = \bigoplus_{P \in Z(S)}\mathcal{L}_P^{(n)}{}'$. Let $\widetilde{\omega}_P^{(n)}{}'$ be the trivialisation of $\mathcal{L}_P^{(n)}{}'$ corresponding to the trivialisation $\omega^{(n)}_P$ of $\mathcal{K}_P^{(n)}$. Let $\beta^n_P \in \Gamma(S,\Omega^1_{S/k})$ be such that
    \[
        \widetilde{\omega}^{(n)}_P{}' = \widetilde{\omega}^{(n)}_P + \beta^n_P.
    \]
    Since $\pi^{-1}Z= \bigsqcup_{P \in Z(S)}\pi^{-1}P$, and $\widetilde{\omega}_P^{(n)}{}'$ has singularities only along the component $\pi^{-1}P$, equation \eqref{eq:defn-Ln} implies that, for every $P \in Z(S)$, we have
    \[
        \mathcal{O}_S \widetilde{\omega}_P^{(n)}{}' \wedge \widetilde{\nu} \wedge \widetilde{\omega}^{(0)} \equiv \mathcal{O}_S \alpha_{21} \wedge \widetilde{\nu} \wedge \widetilde{\omega}^{(n+1)}_P{}' \mod f_*\mathcal{F}^2.
    \]
    Under the identification $f_*\mathcal{F}^1/f_*\mathcal{F}^2 \cong \Omega^1_{S/k} \otimes f_*\Omega^{\bullet}_{E^{\natural}/S}(\log \pi^{-1}Z)$ of Proposition \ref{prop:koszul-log}, we deduce that
    \[
        \beta^n_P \otimes \nu \wedge \omega^{(0)} \in \mathcal{O}_S \alpha_{21}\otimes \nu \wedge \omega^{(n+1)}_P. 
    \]
    This is only possible if $\beta^n_P=0$, since $\nu \wedge \omega^{(0)}, \nu \wedge\omega^{(1)}_P, \ldots$ trivialise $f_*\Omega^2_{E^{\natural}/S}(\log \pi^{-1}Z)$.
\end{proof}

\section{Relative elliptic KZB connections} \label{sec:relative-kzb}

\subsection{Reminders on the bar construction}\label{par:bar-construction}

We work in the category of $\mathcal{O}_S$-modules for some scheme $S$. All tensor products are taken over $\mathcal{O}_S$.

Let $\mathcal{A}$ be a graded-commutative dg-algebra over $\mathcal{O}_S$, and assume that $\mathcal{A}$ is \emph{connected}: $\mathcal{A}^0 = \mathcal{O}_S$. We denote by
\[
    \mathcal{I} \defeq \bigoplus_{n\ge 1} \mathcal{A}^n
\]
the kernel of the augmentation $\mathcal{A} \to \mathcal{O}_S$ given by projection onto the component of degree 0. Local sections of a tensor power $\mathcal{I}^{\otimes n}$ will be written in `bar notation':
\[
    a_1 \otimes \cdots \otimes a_n \eqdef [a_1 | \cdots | a_n].
\]
The \emph{bar construction} associated to $\mathcal{A}$ is the total complex of $\mathcal{O}_S$-modules $(B^{\bullet}(\mathcal{A}),d_B)$ associated to the double complex $(B^{\bullet,\bullet}(\mathcal{A}),d_1,d_2)$ given by
\[
    B^{-s,t}(\mathcal{A}) \defeq (\mathcal{I}^{\otimes s})^t\text{, }\qquad s,t\ge 0,
\]
\[
    d_1: B^{-s,t}(\mathcal{A})\To B^{-s,t+1}(\mathcal{A})\text{, }\qquad  [a_1|\cdots|a_n] \longmapsto \sum_{i=1}^n(-1)^i[Ja_1|\cdots |Ja_{i-1}|da_i|a_{i+1}| \cdots |a_n] 
\]
\[
    d_2: B^{-s,t}(\mathcal{A}) \To B^{-s+1,t}(\mathcal{A})\text{, }\qquad [a_1|\cdots |a_n] \longmapsto \sum_{i=1}^{n-1}(-1)^{i-1}[Ja_1|\cdots |Ja_{i-1}|Ja_i \wedge a_{i+1}|a_{i+2}|\cdots |a_n]
\]
where $J : \mathcal{I} \to \mathcal{I}$ is the involution acting by $(-1)^j$ in degree $j$.

The \emph{length filtration} on $B(\mathcal{A})$ is the increasing exhaustive filtration by the $\mathcal{O}_S$-submodules
\[
    L_nB(\mathcal{A}) \defeq \bigoplus_{m=0}^n \mathcal{I}^{\otimes m}\text{, }\qquad n \ge 0.
\]
Note that $d_1$ sends $L_nB(\mathcal{A})$ to itself, while $d_2$ sends it to $L_{n-1}B(\mathcal{A})$. In particular, the bar differential $d_B = d_1+d_2$ preserves the length filtration.

The bar construction is functorial: if $\varphi: \mathcal{A}_1\rightarrow \mathcal{A}_2$ is a morphism of connected graded-commutative dg-algebras, then
\[
    B(\varphi): B(\mathcal{A}_1)\longrightarrow B(\mathcal{A}_2), \qquad B(\varphi)\defeq \bigoplus_{n\geq 0}\overline{\varphi}^{\otimes n}
\]
is a morphism of complexes of $\mathcal{O}_S$-modules, where $\overline{\varphi}: \mathcal{I}_1\rightarrow \mathcal{I}_2$ is obtained from $\varphi$ by restriction to the kernel of the augmentation. The next result shows that, under a suitable K\"unneth-type condition, the functor $B$ also preserves quasi-isomorphisms.

\begin{lemma}\label{lemma:bar-quasi-isom}
    Let $\varphi: \mathcal{A}_1\to\mathcal{A}_2$ be a dg-quasi-isomorphism between connected graded-commutative dg-algebras. Assume that for all $m,n\geq 0$ the canonical maps
    \begin{equation} \label{eqn:canonicalmap}
        \bigoplus_{i_1+\cdots+i_n=m}H^{i_1}(\mathcal{I}_j)\otimes \cdots\otimes H^{i_n}(\mathcal{I}_j) \longrightarrow H^m(\mathcal{I}_j^{\otimes n})
    \end{equation}
    are isomorphisms, for $j=1,2$. Then, the induced map $B(\varphi): B(\mathcal{A}_1)\rightarrow B(\mathcal{A}_2)$ is a quasi-isomorphism.
\end{lemma}

\begin{proof}
    It suffices to prove that the induced map $L_nB(\mathcal{A}_1)\rightarrow L_nB(\mathcal{A}_2)$ is a quasi-isomorphism for every $n\geq 0$. We proceed by induction on $n$, the base case $n=0$ being trivial. Now fix $n\geq 1$ and consider the following commutative diagram with exact rows:
    \[
        \begin{tikzcd}
            0\arrow{r}&L_{n-1}B(\mathcal{A}_1)\arrow{r}\arrow{d}&L_nB(\mathcal{A}_1)\arrow{r}\arrow{d}&L_nB(\mathcal{A}_1)/L_{n-1}B(\mathcal{A}_1)\arrow{r}\arrow{d}&0\\
            0\arrow{r}&L_{n-1}B(\mathcal{A}_2)\arrow{r}&L_nB(\mathcal{A}_2)\arrow{r}&L_nB(\mathcal{A}_2)/L_{n-1}B(\mathcal{A}_2)\arrow{r}&0.
        \end{tikzcd}
    \]
    The vertical arrow on the left is a quasi-isomorphism by induction hypothesis. That the vertical arrow on the right is a quasi-isomorphism follows from the isomorphisms \eqref{eqn:canonicalmap}. We conclude by an application of the five lemma.
\end{proof}

\begin{remark}\label{rmk:kunneth}
    For the hypotheses of Lemma \ref{lemma:bar-quasi-isom} to be satisfied, it is sufficient that all of $\mathcal{I}^n_j$, $d\mathcal{I}^n_j$, and $H^n(\mathcal{I}_j)$ are flat $\mathcal{O}_S$-modules ($n\ge 0$, $j=1,2$); see \cite[Theorem 3.6.3]{weibel94}.
\end{remark}

We state without proof the following standard result.

\begin{proposition}
    The cohomology in degree 0 of the bar construction
    \[
        H^0(B(\mathcal{A})) = \ker(d_B:B^0(\mathcal{A}) \To B^1(\mathcal{A}))
    \]
    is naturally equipped with the structure of a filtered commutative Hopf algebra over $\mathcal{O}_S$, given by:
    \begin{itemize}
        \item A commutative multiplication $\shuffle: H^0(B(\mathcal{A}))\otimes H^0(B(\mathcal{A}))\rightarrow H^0(B(\mathcal{A}))$, the \emph{shuffle product}, which is defined on local sections by
        \[
            [a_1|\cdots|a_m] \shuffle [a_{m+1}|\cdots|a_{m+n}]=\sum_{\sigma}[a_{\sigma(1)}|\cdots|a_{\sigma(m+n)}],
        \]
        where the sum ranges over all permutations $\sigma$ of $\{1,\ldots,m+n\}$ such that $\sigma^{-1}$ is strictly increasing on $\{1,\ldots,m\}$ and $\{m+1,\ldots,m+n\}$. The unit for $\shuffle$ is $1 \in B^{0,0}(\mathcal{A})=\mathcal{O}_S$.    
        \item A comultiplication $\Delta: H^0(B(\mathcal{A}))\rightarrow H^0(B(\mathcal{A}))\otimes H^0(B(\mathcal{A}))$, the \emph{deconcatenation coproduct}, which is defined on local sections by
        \[
            \Delta([a_1|\cdots|a_n])=[a_1|\cdots|a_n] \otimes 1+1\otimes [a_1|\cdots|a_n]+\sum_{r=1}^{n-1} [a_1|\cdots|a_r]\otimes [a_{r+1}|\cdots|a_n].
        \]
        The counit for $\Delta$ is the augmentation map $\varepsilon: H^0(B(\mathcal{A}))\rightarrow \mathcal{O}_S$.  
        \item An antipode $\sigma: H^0(B(\mathcal{A}))\rightarrow H^0(B(\mathcal{A}))$, given on local sections by
        \[
            \sigma([a_1|\cdots|a_n])=(-1)^n[a_n|\cdots|a_1].
        \]
        \item A length filtration $L_nH^0(B(\mathcal{A})) \defeq L_nB(\mathcal{A}) \cap H^0(B(\mathcal{A}))$.
    \end{itemize}
\end{proposition}

\begin{example}\label{ex:non-commutative}
    If $\mathcal{A}^n = 0$ for $n\ge 2$, then $d_B = 0$, so that
    \[
        H^0(B(\mathcal{A})) = B^0(\mathcal{A}) = T^c\mathcal{A}^1.
    \]
    Here, $T^c\mathcal{A}^1 = \bigoplus_{n\ge 0}(\mathcal{A}^1)^{\otimes n}$ denotes the \emph{tensor coalgebra} on $\mathcal{A}^1$, with the above structure of filtered Hopf algebra over $\mathcal{O}_S$. If moreover $\mathcal{A}^1 = \mathcal{O}_S\alpha_1 \oplus \cdots \oplus \mathcal{O}_S\alpha_r$, then we can identify
    \[
        H^0(B(\mathcal{A})) \cong \mathcal{O}_S \langle \alpha_1,\ldots,\alpha_r\rangle,
    \]
    where the length filtration $L_n$ is spanned by non-commutative polynomials of total degree $\le n$.
\end{example}

In general, denote by $\pr_n : H^0(B(\mathcal{A})) \to (\mathcal{A}^1)^{\otimes n}$ the natural projection onto the component of pure length $n$. Comodules for the Hopf algebra $H^0(B(\mathcal{A}))$ can be characterised as follows. 

\begin{proposition}\label{prop:comodule-bar}
    Let $\mathcal{E}$ be a vector bundle over $S$, and $\rho: \mathcal{E} \to H^0(B(\mathcal{A})) \otimes \mathcal{E}$ be an $\mathcal{O}_S$-morphism. Write $\rho = \sum_{n\ge 0}\rho_n$, where $\rho_n = (\pr_n \otimes \id)\circ \rho : \mathcal{E} \to (\mathcal{A}^1)^{\otimes n}\otimes \mathcal{E} $. Then, $\rho$ is a comodule structure if and only if
    \begin{enumerate}[(i)]
        \item $\rho_0 = \id_{\mathcal{E}}$, and
        \item $\rho_n$ is the $n$-fold composition
        \[
            [\rho_1]^n \defeq (\id_{(\mathcal{A}^1)^{\otimes n-1}}\otimes \rho_1)\circ \cdots \circ (\id_{\mathcal{A}^1}\otimes \rho_1)\circ \rho_1:\mathcal{E} \longrightarrow (\mathcal{A}^1)^{\otimes n}\otimes \mathcal{E},
        \]
        for every $n\ge 1$.
    \end{enumerate}
    Moreover, given an $\mathcal{O}_S$-linear map $\omega: \mathcal{E} \to \mathcal{A}^1 \otimes \mathcal{E}$, there exists a comodule structure $\rho : \mathcal{E} \to H^0(B(\mathcal{A})) \otimes \mathcal{E}$ satisfying $\rho_1 = \omega$ if and only if $\omega$ is \emph{locally nilpotent} (i.e.,  locally over $S$, we have $[\omega]^n = 0$ for $n\gg 0$), and
    \[
        d\omega + \omega\wedge \omega =0
    \]
    in $\mathcal{A}^2 \otimes \mathcal{E}nd(\mathcal{E})$.
\end{proposition}

\begin{proof}
    Since the counit $\varepsilon$ of $H^0(B(\mathcal{A}))$ is equal to $\pr_0$, statement (i) is equivalent to the comodule axiom $(\varepsilon \otimes \id)\circ \rho = \id$. Let
    \[
        \delta_{i,j} : (\mathcal{A}^1)^{\otimes i+j} \stackrel{\sim}{\To} (\mathcal{A}^1)^{\otimes i} \otimes (\mathcal{A}^1)^{\otimes j}\text{, } \qquad [a_1| \cdots | a_{i+j}] \longmapsto [a_1 | \cdots |a_i] \otimes [a_{i+1}| \cdots | a_{i+j}]
    \]
    be the `deconcatenation isomorphisms'. Using the expression $\rho = \sum_{n\ge 0}\rho_n$ and the definition of the deconcatenation coproduct $\Delta$ of $H^0(B(\mathcal{A}))$, we obtain
    \[
        (\Delta \otimes \id)\circ \rho = \sum_{n\ge 0} (\Delta \otimes \id)\circ \rho_n = \sum_{n\ge 0}\sum_{i+j = n}(\delta_{i,j} \otimes \id)\circ \rho_n = \sum_{i,j\ge 0}(\delta_{i,j} \otimes \id)\circ \rho_{i+j}
    \]
    and
    \[
         (\id\otimes \rho)\circ \rho  = \left( \id \otimes \sum_{i\ge 0}\rho_i\right)\circ \sum_{j\ge 0}\rho_j = \sum_{i,j\ge 0} (\id\otimes\rho_i)\circ \rho_j.
    \]
    By induction, this shows that the comodule axiom $(\Delta \otimes \id)\circ \rho = (\id\otimes \rho)\circ \rho$ is equivalent to (ii).

    For the last assertion, note that any $\mathcal{O}_S$-linear map $\omega : \mathcal{E} \to \mathcal{A}^1\otimes \mathcal{E}$ defines a $\mathcal{O}_S$-linear map
    \[
      \rho \defeq ( [\omega]^n)_{\ge 0} : \mathcal{E} \longrightarrow \prod_{n\ge 0}(\mathcal{A}^1)^{\otimes n}\otimes \mathcal{E}.
    \]
    We regard $T^c \mathcal{A}^1 = \bigoplus_{n\ge 0} (\mathcal{A}^1)^{\otimes n}$ as a submodule of $\prod_{n\ge 0}(\mathcal{A}^1)^{\otimes n}$. Since $\mathcal{E}$ is an $\mathcal{O}_S$-module of finite type, we see that $\rho$ factors through $T^c\mathcal{A}^1 \otimes  \mathcal{E}$ if and only if, locally over $S$, we have $[\omega]^n = 0$ for every sufficiently large $n$. Finally, since $H^0(B(\mathcal{A})) = \ker(d_B)$ and $\mathcal{E}$ is flat, the image of $\rho$ is contained in $H^0(B(\mathcal{A})) \otimes \mathcal{E}$ if and only if $(d_B\otimes \id)\circ \rho =0$. In bar notation, we have:
    \begin{align*}
        (d_B \otimes \id) \circ \rho &= \sum_{n\ge 1}d_B \underbrace{[\omega | \cdots | \omega]}_{\text{length }n}= -\sum_{n\ge 1}\left(\sum_{i=1}^n\underbrace{[\omega |\cdots | \overbrace{d\omega}^{i\text{th position}} |\cdots |\omega]}_{\text{length }n} + \sum_{i=1}^{n-1}\underbrace{[\omega | \cdots  | \overbrace{\omega \wedge \omega}^{i\text{th position}} | \cdots |\omega}_{\text{length }n-1}] \right)\\
        &= -\sum_{n\ge 1}\underbrace{[\omega | \cdots |\overbrace{ d\omega + \omega \wedge \omega}^{i\text{th position}}  | \cdots |\omega}_{\text{length }n-1}],
    \end{align*}
    so that the identity $(d_B\otimes \id)\circ \rho =0$ is equivalent to $d\omega + \omega \wedge \omega  = 0$.
\end{proof}

\subsection{Fundamental Hopf algebra} \label{subsec:fundamental-Hopf}

Let $S$ be a scheme of characteristic zero, $p:E \to S$ be an elliptic curve, and $Z$ be a divisor on $E/S$ as in \S\ref{par:kronecker-subbundle}. Let $f:E^{\natural} \to S$ be the universal vector extension of $E/S$ and $\pi:E^{\natural} \to E$ be the natural projection.

\begin{definition}\label{def:fundamental-hopf}
    The \emph{de Rham fundamental Hopf algebra} of $E/S$ punctured at $Z$ is the filtered Hopf algebra over $\mathcal{O}_S$ defined by
    \[
        \mathcal{H}_{E/S,Z} \defeq H^0(B(f_*\Omega^{\bullet}_{E^{\natural}/S}(\log \pi^{-1}Z))).
    \]
\end{definition}

The affine group scheme over $S$ corresponding to $\mathcal{H}_{E/S,Z}$ can be regarded as a base point free version of the `relative unipotent de Rham fundamental group' of $E\setminus Z$ over $S$  (cf. \cite{CdPS19}). We refer to Appendix \ref{appendix:tannakian} for precise comparison theorems in the case where $S$ is the spectrum of a field.

\begin{example}
    Assume that $S$ is affine, that $R^1p_*\mathcal{O}_E$ is trivial, and let $\nu,\omega^{(0)},\omega^{(1)}_P,\omega^{(2)}_P,\ldots$ be as in \eqref{eq:triv-A1}. Then, a section $\xi$ of $\mathcal{H}_{E/S,Z}$ is an $\mathcal{O}_S$-linear combination of words in these 1-forms, satisfying $d_B\xi =0$. For instance,
    \[
        [\omega^{(0)}|\omega^{(0)}]\text{,}\qquad [\omega^{(0)}| \nu] + [\omega_P^{(1)}]\text{, }\qquad [\omega^{(1)}_{P}-\omega^{(1)}_O|\nu]+ [\omega^{(2)}-\omega_O^{(2)}]
    \]
    are sections of length 2, and
    \begin{align*}
         [\omega^{(0)} | \nu | \omega^{(0)} | \nu | \nu] + [\omega_P^{(1)} | \omega^{(0)} | \nu | \nu] &- [\omega^{(0)} | \omega_P^{(1)} | \nu | \nu] + [\omega^{(0)} | \nu | \omega_P^{(1)} | \nu] + 2[\omega_P^{(1)} | \omega^{(1)} | \nu]\\
        &\ \ \ - 2[\omega^{(0)} | \omega_P^{(2)} | \nu] + [\omega^{(0)} | \nu | \omega_P^{(2)}] + 3 [\omega_P^{(1)} | \omega_P^{(2)}] - 3 [\omega^{(0)} | \omega_P^{(3)}]
    \end{align*}
    is a section of length 5.
\end{example}

\begin{theorem}\label{thm:tensor-coalgebra}
    The projector $\rho$ from Theorem \ref{thm:projector} induces an isomorphism of filtered Hopf algebras over $\mathcal{O}_S$
    \[
        \mathcal{H}_{E/S,Z} \stackrel{\sim}{\To} T^cH^1_{\dR}((E\setminus Z)/S).
    \]
\end{theorem}

\begin{proof}
    This follows immediately from the fact that $\rho$ is a dg-quasi-isomorphism (Theorem \ref{thm:projector}) and from Lemma \ref{lemma:bar-quasi-isom}. To verify the Künneth-type hypotheses of Lemma \ref{lemma:bar-quasi-isom} (cf. Remark \ref{rmk:kunneth}), we apply Theorem \ref{thm:kronecker-subbundle} and Theorem \ref{thm:projector}.
\end{proof}

\begin{example}
    Locally over $S$, the isomorphism of the above theorem is given explicitly by writing a section of $\mathcal{H}_{E/S,Z}$ as an $\mathcal{O}_S$-linear combination of words in $\omega^{(0)},\nu,\omega^{(1)}_{P}-\omega^{(1)}_O, \omega^{(1)}_O,\omega^{(n)}_P$ (where $n\ge 2$ and $P \in Z(S)$), and by sending $\omega^{(1)}_O,\omega^{(n)}_P$ to zero. For instance, the length 2 section
    \[
        [\omega^{(0)}| \nu] + [\omega_P^{(1)}] = [\omega^{(0)}| \nu] + [\omega_P^{(1)} - \omega^{(1)}_O] + [\omega^{(1)}_O]
    \]
    of $\mathcal{H}_{E/S,Z}$ is sent to the length 2 section  $[\omega^{(0)}| \nu] + [\omega_P^{(1)} - \omega^{(1)}_O]$ of $T^cH^1_{\dR}((E\setminus Z)/S)$.
\end{example}

We also deduce from the above theorem that the formation of $\mathcal{H}_{E/S,Z}$ commutes with every base change in $S$.

\begin{corollary}\label{coro:base-change}
    For any morphism of schemes $\varphi:S'\to S$, the natural map
    \[
        \varphi^*\mathcal{H}_{E/S,Z} \To \mathcal{H}_{(E\times_SS')/S, Z\times_SS'}
    \]
    is an isomorphism of filtered Hopf algebras over $\mathcal{O}_{S'}$.
\end{corollary}

\begin{proof}
    It is well known that, for proper and smooth schemes, the formation of the de Rham cohomology commutes with arbitrary base change. Using the residue exact sequence
    \[
        \begin{tikzcd}
            0 \arrow{r} & H^1_{\dR}(E/S) \arrow{r} & H^1_{\dR}((E\setminus Z)/S) \arrow{r}{\Res}& H^0_{\dR}(Z/S)\arrow{r} & H^2_{\dR}(E/S) \arrow{r} & 0
        \end{tikzcd}
    \]
    we deduce that the formation of $H^1_{\dR}((E\setminus Z)/S)$ commutes with arbitrary base change. To conclude we apply Theorem \ref{thm:tensor-coalgebra} and the fact that the formation of the universal vector extension also commutes with arbitrary base change (cf. \S \ref{par:def-uve}). 
\end{proof}

\subsection{Elliptic KZB connection}\label{par:kzb-connection}

We keep the notation of \S\ref{subsec:fundamental-Hopf}. It follows from Theorem \ref{thm:tensor-coalgebra} that there are canonical isomorphisms
\begin{equation}\label{eq:isomorphism-graded-length}
    L_n\mathcal{H}_{E/S,Z}/L_{n-1}\mathcal{H}_{E/S,Z} \cong H^1_{\dR}((E\setminus Z)/S)^{\otimes n}.
\end{equation}
In particular, each $L_n\mathcal{H}_{E/S,Z}$ is a vector bundle over $S$. The \emph{continuous dual} of $\mathcal{H}_{E/S,Z}$ is the $\mathcal{O}_S$-module
\[
    \mathcal{H}_{E/S,Z}^{\vee} \defeq \lim_n \, (L_n\mathcal{H}_{E/S,Z})^{\vee}\text{, }\qquad (L_n\mathcal{H}_{E/S,Z})^{\vee} = \mathcal{H}om_{\mathcal{O}_S}(L_n\mathcal{H}_{E/S,Z},\mathcal{O}_S),
\]
with the dual structure of a completed Hopf algebra over $\mathcal{O}_S$ (cf. \cite[\S3.2.6]{BGF22}). For instance, its multiplication is given by
\[
    \Delta^{\vee} : \mathcal{H}_{E/S,Z}^{\vee}\hat{\otimes}\mathcal{H}_{E/S,Z}^{\vee} \longrightarrow \mathcal{H}_{E/S,Z}^{\vee}\text{,}\qquad \Delta^{\vee} \defeq \lim_n \, (\Delta|_{L_n\mathcal{H}_{E/S,Z}})^{\vee}.
\]
Regarding the restriction of the antipode $\sigma$ to $L_n\mathcal{H}_{E/S,Z}$ as a global section of $L_n\mathcal{H}_{E/S,Z} \otimes (L_n\mathcal{H}_{E/S,Z})^{\vee}$, we get a global section
\[
    \hat{\sigma}  =  \lim_n \sigma|_{L_n\mathcal{H}_{E/S,Z}} \in \Gamma(S, \mathcal{H}_{E/S,Z}\hat{\otimes}\mathcal{H}_{E/S,Z}^{\vee}).
\]

\begin{definition}
    The \emph{KZB form} of $E/S$ punctured at $Z$ is the global section
    \[
        \omega_{E^{\natural}/S,Z} \defeq (\pr_1\otimes \id)(\hat{\sigma}) \in \Gamma(S, f_*\Omega^1_{E^{\natural}/S}(\log \pi^{-1}Z) \hat{\otimes} \mathcal{H}_{E/S,Z}^{\vee}),
    \]
    where $\pr_1: \mathcal{H}_{E/S,Z} \to f_*\Omega^1_{E^{\natural}/S}(\log \pi^{-1}Z)$ is the projection onto the component of length 1.
\end{definition}

By letting $\omega_{E^{\natural}/S,Z}$ act on $\mathcal{H}_{E/S,Z}^{\vee}$ by left multiplication via $\Delta^{\vee}$, we can also regard it as an $\mathcal{O}_S$-linear map
\[
    \omega_{E^{\natural}/S,Z}: \mathcal{H}_{E/S,Z}^{\vee} \longrightarrow f_*\Omega^1_{E^{\natural}/S}(\log \pi^{-1}Z) \hat{\otimes} \mathcal{H}_{E/S,Z}^{\vee}.
\]

\begin{proposition}\label{prop:univ-property-kzb}
    We have
    \[
        d\omega_{E^{\natural}/S,Z} + \omega_{E^{\natural}/S,Z}\wedge \omega_{E^{\natural}/S,Z} = 0.
    \]
    Moreover, if we denote
    \[
        1 \defeq \lim_n\,  (\varepsilon|_{L_n\mathcal{H}_{E/S,Z}})^{\vee} \in \Gamma(S,\mathcal{H}_{E/S,Z}^{\vee}),
    \]
    then the triple $(\mathcal{H}^{\vee}_{E/S,Z},\omega_{E^{\natural}/S,Z},1)$ satisfies the following universal property: for every triple $(\mathcal{E},\omega,e)$, where $\mathcal{E}$ is a vector bundle over $S$, $\omega: \mathcal{E} \to f_*\Omega^1_{E^{\natural}/S}(\log \pi^{-1}Z)\otimes \mathcal{E}$ is a locally nilpotent $\mathcal{O}_S$-linear map satisfying $d\omega + \omega \wedge \omega =0$, and $e \in \Gamma(S,\mathcal{E})$, there is a unique $\mathcal{O}_S$-linear map $\varphi : \mathcal{H}^{\vee}_{E/S,Z} \to \mathcal{E}$ such that $(\id \otimes \varphi) \circ \omega_{E^{\natural}/S,Z} = \omega \circ \varphi$ and $\varphi(1) = e$.
\end{proposition}

\begin{proof}
    Let $\rho: \mathcal{E} \to \mathcal{H}_{E/S,Z}\otimes \mathcal{E}$ be the comodule structure corresponding to $\omega$ as in Proposition \ref{prop:comodule-bar}. Working locally over $S$, we may assume that there is some $n_0\ge 0$ such that $\rho(\mathcal{E}) \subset L_{n_0}\mathcal{H}_{E/S,Z}\otimes \mathcal{E}$. For every $n\ge n_0$, define
    \begin{equation}\label{eq:phin}
        \varphi_n : (L_n\mathcal{H}_{E/S,Z})^{\vee} \To \mathcal{E}\text{, }\qquad  \lambda \longmapsto (\lambda \otimes \id)\circ (\sigma\otimes \id)\circ \rho(e).
    \end{equation}
    Then, a straightforward computation shows that $\varphi \defeq \lim_n \varphi_n : \mathcal{H}_{E/S,Z}^{\vee} \to \mathcal{E}$ is an $\mathcal{O}_S$-linear map satisfying the properties in the statement. 
    To prove unicity, let $\rho_{E/S,Z}: \mathcal{H}_{E/S,Z}^{\vee} \to \mathcal{H}_{E/S,Z}\hat{\otimes} \mathcal{H}_{E/S,Z}^{\vee}$ be the completed comodule structure corresponding to $\omega_{E^{\natural}/S,Z}$ via Proposition \ref{prop:comodule-bar}. It is given by left multiplication by $\hat{\sigma}$. If $\varphi': \mathcal{H}_{E/S,Z}^{\vee} \to \mathcal{E}$ is an $\mathcal{O}_S$-morphism satisfying $\varphi'(1)=e$ and  $(\id \otimes \varphi') \circ \omega_{E^{\natural}/S,Z} = \omega \circ \varphi'$, then it follows from Proposition \ref{prop:comodule-bar} that this last equation can be lifted to
    \[
        (\id \otimes \varphi') \circ \rho_{E/S,Z} = \rho \circ \varphi'.
    \]
    Thus,
    \[
        \rho(e) = (\id \otimes \varphi')\circ \rho_{E/S,Z}(1) = (\id \otimes \varphi')(\hat{\sigma}).
    \]
    Let $n\ge n_0$ be as above. For any section $\lambda$ of $(L_n\mathcal{H}_{E/S,Z})^{\vee}$, we have
    \[
        (\lambda \otimes \id)\circ \rho (e) = (\lambda \otimes \id) \circ (\id \otimes \varphi')(\hat{\sigma}) = \varphi' \circ (\lambda \otimes \id)(\hat{\sigma}) = \varphi'_n(\lambda \circ \sigma|_{L_n\mathcal{H}_{E/S,Z}}),
    \]
    where in the last equality we regard $\lambda \circ \sigma|_{L_n\mathcal{H}_{E/S,Z}}$ as a section of $(L_n\mathcal{H}_{E/S,Z})^{\vee}$. Finally, using that $\sigma$ is an involution, we conclude that $\varphi'_n$ is given by the same formula as $\varphi_n$ in \eqref{eq:phin}.
\end{proof}

\begin{definition}\label{def:vertical-kzb}
    Consider the completed pullback $f^*\mathcal{H}_{E/S,Z}^{\vee} \defeq \lim_n \, f^*(L_n\mathcal{H}_{E/S,Z})^{\vee}$. The \emph{relative elliptic KZB connection} of $E/S$ punctured at $Z$ is the $S$-connection
    \[
        \nabla_{E^{\natural}/S,Z} : f^*\mathcal{H}^{\vee}_{E/S,Z} \longrightarrow \Omega^1_{E^{\natural}/S}(\log \pi^{-1}Z) \hat{\otimes} f^*\mathcal{H}_{E/S,Z}^{\vee}\text{, }\qquad \nabla_{E^{\natural}/S,Z} = d + \omega_{E^{\natural}/S,Z}.
    \]
\end{definition}

\begin{proposition}\label{prop:base-change-vkzb}
    The formation of $\nabla_{E^{\natural}/S,Z}$ is compatible with arbitrary base change in $S$.
\end{proposition}

\begin{proof}
    This follows from Corollary \ref{coro:base-change} and from the fact that $\omega_{E^{\natural}/S,Z}$ is induced by the antipode of the Hopf algebra $\mathcal{H}_{E/S,Z}$.
\end{proof}

We shall now give explicit formulas for the relative elliptic KZB connection. Assuming that $S$ is affine and that $R^1p_*\mathcal{O}_E$ is trivial, let $\nu,\omega^{(0)},\omega^{(1)}_P,\omega^{(2)}_P,\ldots$ be as in \eqref{eq:triv-A1}. Recall that
\[
    \{\nu,\omega^{(0)}\} \cup \{\omega^{(1)}_P - \omega^{(1)}_O : P \in Z(S)\setminus \{O\}\}
\]
trivialises the vector bundle $H^1(f_*\Omega^{\bullet}_{E^{\natural}/S}(\log \pi^{-1}Z))$ over $S$ (cf. proof of Theorem \ref{thm:projector}). We denote by
\[
    \{a,b\} \cup \{b_P : P \in Z(S) \setminus \{O\}\}
\]
the dual trivialisation of $H^1(f_*\Omega^{\bullet}_{E^{\natural}/S}(\log \pi^{-1}Z))^{\vee}$. By Theorem \ref{thm:tensor-coalgebra} (cf. Example \ref{ex:non-commutative}), we have
\[
    \mathcal{H}_{U/S}^{\vee} \cong \mathcal{O}_S \langle \!\langle a,b,b_P : P \in Z(S) \setminus\{O\}\rangle \! \rangle. 
\]
Note that Hopf algebra unit $1$ as defined in Proposition \ref{prop:univ-property-kzb} gets identified with the constant $1$. It will be more convenient to write
\[
    \mathcal{O}_S \langle \!\langle a,b,b_P : P \in Z(S) \setminus\{O\}\rangle \! \rangle \cong \frac{ \mathcal{O}_S \langle \!\langle a,b,c_P : P \in Z(S)\rangle \! \rangle}{\langle \sum_{P \in Z(S)}c_P - [a,b] \rangle},
\]
where the isomorphism is given by sending $b_P$ to $c_P$ for every $P \in Z(S) \setminus\{O\}$.

\begin{theorem}\label{thm:explicit-kzb-form}
    With the above notation, we have
    \begin{equation}\label{eq:explicit-kzb-form}
        \omega_{E^{\natural}/S,Z} =  - \nu \otimes a - \omega^{(0)}\otimes b - \sum_{n\ge 1}\sum_{P \in Z(S)}\omega^{(n)}_P\otimes \ad_a^{n-1}c_P,
    \end{equation}
    where $\ad_ax = [a,x] = ax-xa$.
\end{theorem}

\begin{proof}
    Call $\omega'$  the right-hand side of \eqref{eq:explicit-kzb-form}. We show that $\omega'$ satisfies the universal property for the KZB form. For this, let $(\mathcal{E},\omega,e)$ be a triple as in Proposition \ref{prop:univ-property-kzb}. As $\nu,\omega^{(0)},\omega^{(n)}_P$ ($n\ge 1$, $P \in Z(S)$) trivialise $f_*\Omega^1_{E^{\natural}/S}(\log \pi^{-1}Z)$, we can uniquely write
    \[
        \omega = - \nu \otimes A - \omega^{(0)}\otimes B - \sum_{n\ge 1}\sum_{P \in Z(S)}\omega^{(n)}_P \otimes C^{(n)}_P,
    \]
    where $A,B,C^{(n)}_P$ are nilpotent endomorphisms of $\mathcal{E}$ (with $C^{(n)}_P = 0$ for $n\gg 0$). Since
    \[
        0 = d\omega + \omega \wedge \omega = \nu \wedge \omega^{(0)} \otimes \left([A,B]-\sum_{P \in Z(S)}C^{(1)}_P\right) + \sum_{n\ge 1}\nu \wedge \omega^{(n)}_P \otimes \left([A,C^{(n)}_P]-C^{(n+1)}_P\right), 
    \]
    we conclude that
    \[
        \sum_{P \in Z(S)}C^{(1)}_P = [A,B]\text{, }\qquad C_P^{(n)} = \ad_A^{n-1}C^{(1)}_P.
    \]
    Thus, there is a unique $\mathcal{O}_S$-morphism $\varphi:\mathcal{H}^{\vee}_{E/S,Z} \to \mathcal{E}$ satisfying $\varphi(1)= e$, $\varphi(a) = Ae$, $\varphi(b)=Be$, and $\varphi(c_P)=C^{(1)}_Pe$ for every $P \in Z(S)$. To finish the proof, it suffices to remark that these conditions are equivalent to $\varphi(1)=e$ and $(\id \otimes \varphi) \circ \omega' = \omega \circ \varphi$.
\end{proof}

\begin{example}
    When $Z=O$, we have $\mathcal{H}_{E/S,O}^{\vee} \cong \mathcal{O}_S \langle \! \langle a,b\rangle\!\rangle$. Since $c_O = [a,b]=\ad_ab$, the formula for $\omega_{E^{\natural}/S,Z}$ becomes
    \[
        \omega_{E^{\natural}/S,O} = - \nu \otimes a - \sum_{n\ge 0}\omega^{(n)}\otimes \ad^n_ab,
    \]
    where $\omega^{(n)} = \omega^{(n)}_O$.
\end{example}

\section{Absolute elliptic KZB connections}\label{sec:absolute-kzb}

\subsection{Bar construction relative to a dg-algebra}\label{par:relative-bar}

Let $S$ be a smooth scheme over a field $k$ of characteristic zero. In this paragraph, we consider a variant of the bar construction relative to the dg-algebra $\Omega \defeq \Omega^{\bullet}_{S/k}$.

To lighten the notation, tensor products without subscripts are over $\mathcal{O}_S$. If $\mathcal{F}$ and $\mathcal{G}$ are bimodules over the (non-commutative) $\mathcal{O}_S$-algebra $\Omega$, we denote
\[
    \mathcal{F} \otimes_{\Omega}\mathcal{G} \defeq (\mathcal{F} \otimes\mathcal{G})/ \mathcal{R},
\]
where $\mathcal{R}$ is the submodule generated by $(f\wedge \omega)\otimes g - f \otimes (J\omega \wedge g)$, with $f,g,\omega$ sections of $\mathcal{F},\mathcal{G},\Omega$ respectively.

Let $\mathcal{C} = \bigoplus_{n\ge 0}\mathcal{C}^n$ be a graded $\mathcal{O}_S$-algebra equipped with a $k$-linear differential $d: \mathcal{C} \to \mathcal{C}[1]$ making it a commutative dg-algebra over $k$. Assume that $\mathcal{C}^0 = \mathcal{O}_S$ and that $\mathcal{C}$ contains $\Omega$ both as a graded subalgebra over $\mathcal{O}_S$ and as a dg-subalgebra over $k$. The \emph{relative bar construction} on $\mathcal{C}$ is defined as follows. Let $\mathcal{J} = \bigoplus_{n\ge 1}\mathcal{C}^n$ be the kernel of the projection on the component of degree zero $\mathcal{C} \to \mathcal{O}_S$, and set
\[
    B_{\Omega}^{-s,t}(\mathcal{C}) \defeq (\mathcal{J}^{\otimes_{\Omega} s})^t\text{, }\qquad s,t\ge 0.
\]
Here, we also denote local sections of a tensor power $\mathcal{J}^{\otimes_{\Omega}n}$ in `bar notation':
\[
    c_1 \otimes \cdots \otimes c_n \eqdef [c_1 | \cdots | c_n].
\]
With the above notion of tensor product, one can readily check that
\[
    d_1: B_{\Omega}^{-s,t}(\mathcal{C})\To B_{\Omega}^{-s,t+1}(\mathcal{C})\text{, }\qquad  [c_1|\cdots|c_n] \longmapsto \sum_{i=1}^n(-1)^i[Jc_1|\cdots |Jc_{i-1}|dc_i|c_{i+1}| \cdots |c_n] 
\]
\[
    d_2: B_{\Omega}^{-s,t}(\mathcal{C}) \To B^{-s+1,t}_{\Omega}(\mathcal{C})\text{, }\qquad [c_1|\cdots |c_n] \longmapsto \sum_{i=1}^{n-1}(-1)^{i-1}[Jc_1|\cdots |Jc_{i-1}|Jc_i \wedge c_{i+1}|c_{i+2}|\cdots |c_n]
\]
are well-defined (even though $d$ is only $k$-linear), so that we obtain a double complex $(B^{\bullet,\bullet}_{\Omega}(\mathcal{C}),d_1,d_2)$. The relative bar construction is the associated total complex $(B^{\bullet}_{\Omega}(\mathcal{C}),d_B)$.

\begin{remark}
    Similarly to the usual bar construction, as an $\mathcal{O}_S$-module, $B_{\Omega}(\mathcal{C})$ is simply the tensor module
    \[
        B_{\Omega}(\mathcal{C}) =  \bigoplus_{n\ge 0} \mathcal{J}^{\otimes_{\Omega} n},
    \]
    with a shift on the grading: $\mathrm{deg}([c_1|\cdots |c_n]) = \sum_{i=1}^n(\mathrm{deg}(c_i)-1)$.
\end{remark}

\subsection{Koszul filtration on the relative bar construction}

We take
\[
    \mathcal{C} \defeq f_*\Omega^{\bullet}_{E^{\natural}/k}(\log \pi^{-1}Z),
\]
with hypotheses and notation as in \S\ref{subsec:fundamental-Hopf}. The Koszul filtration by $\mathcal{O}_{E^{\natural}}$-submodules $\mathcal{F}^0\supset \mathcal{F}^1 \supset \cdots$ on $\Omega^{\bullet}_{E^{\natural}/k}(\log \pi^{-1}Z)$ (cf. \S\ref{par:canonical-lifts}) induces by direct image a filtration by $\mathcal{O}_S$-submodules $f_*\mathcal{F}^0\supset f_*\mathcal{F}^1 \supset \cdots $ on $\mathcal{C}$. Note that each $f_*\mathcal{F}^p$ is actually a $\Omega$-submodule of $\mathcal{C}$. Let $F^{\bullet}\mathcal{J}$ be the induced filtration on $\mathcal{J} \subset \mathcal{C}$, so that:
\[
    F^p\mathcal{J} =
    \begin{cases}
        \mathcal{J} & p=0 \\
        f_*\mathcal{F}^p & p\ge 1.
    \end{cases}
\]
This induces a filtration on $B_{\Omega}(\mathcal{C})$:
\[
    F^pB_{\Omega}(C) = \bigoplus_{n\ge 1} \sum_{p_1 + \cdots + p_n=p}\im(F^{p_1}\mathcal{J}\otimes_{\Omega}\cdots \otimes_{\Omega}F^{p_n}\mathcal{J} \To \mathcal{J}^{\otimes_{\Omega} n}).
\]

In what follows, let $\mathcal{A} = f_*\Omega^{\bullet}_{E^{\natural}/S}(\log \pi^{-1}Z)$ and $\mathcal{I} = \bigoplus_{n\ge 1}\mathcal{A}^n$ (cf. \S\ref{par:bar-construction}). By Proposition \ref{prop:koszul-log}, we have $\mathcal{J}/F^1\mathcal{J} \cong \mathcal{I}$. Let
\[
    B_{\Omega}(\mathcal{C}) \To B(\mathcal{A})
\]
be the natural map induced by the quotient map $\mathcal{J} \to \mathcal{I}$.

\begin{lemma}\label{lemma:sequence-bar}
    The sequence
    \[
        \begin{tikzcd}
            0 \arrow{r} & F^1B_{\Omega}(\mathcal{C}) \arrow{r} & B_{\Omega}(\mathcal{C}) \arrow{r} & B(\mathcal{A}) \arrow{r} & 0
        \end{tikzcd}
    \]
    is exact.
\end{lemma}

\begin{proof}
    By flatness of $\mathcal{J}$, $F^1\mathcal{J}$, and of $\mathcal{I}$, we have, for every $n\ge 1$, an exact sequence of $\mathcal{O}_S$-modules
    \[
        \begin{tikzcd}
            0 \arrow{r} & \bigoplus_{i=1}^n \mathcal{J}^{\otimes i-1}\otimes F^1\mathcal{J} \otimes \mathcal{J}^{\otimes n-i} \arrow{r} & \mathcal{J}^{\otimes n} \arrow{r} & \mathcal{I}^{\otimes n} \arrow{r} & 0.
        \end{tikzcd}
    \]
    Since the map $\mathcal{J}^{\otimes n} \to \mathcal{I}^{\otimes n}$ factors through $\mathcal{J}^{\otimes_{\Omega} n} \to \mathcal{I}^{\otimes n}$, we obtain an exact sequence of $\mathcal{O}_S$-modules
    \[
        \begin{tikzcd}
            0 \arrow{r} & \bigoplus_{i=1}^n \mathcal{J}^{\otimes_{\Omega} i-1}\otimes_{\Omega} F^1\mathcal{J} \otimes_{\Omega} \mathcal{J}^{\otimes_{\Omega} n-i} \arrow{r} & \mathcal{J}^{\otimes_{\Omega} n} \arrow{r} & \mathcal{I}^{\otimes n} \arrow{r} & 0.
        \end{tikzcd}
    \]
    This proves that the sequence in the statement is an exact sequence of $\mathcal{O}_S$-modules.
\end{proof}

Let
\[
    \sigma_{n,i}: \mathcal{I}^{\otimes i-1}\otimes (\Omega^1[1] \otimes \mathcal{I}) \otimes \mathcal{I}^{\otimes n-i} \stackrel{\sim}{\To} \Omega^1[1] \otimes \mathcal{I}^{\otimes n}
\]
be the `Koszul sign rule' isomorphism given by
\[
    [a_1 | \cdots |a_{i-1}|\omega \otimes a_i | a_{i+1}|\cdots |a_n]\longmapsto (-1)^{i-1}\omega \otimes [Ja_1 | \cdots |Ja_{i-1}|a_i |\cdots |a_n].
\]

\begin{lemma}\label{lemma:projection}
    For $n\ge 1$ and $1\le i \le n$, let
    \[
        \pi_{n,i}: \mathcal{J}^{\otimes i-1}\otimes F^1\mathcal{J} \otimes \mathcal{J}^{\otimes n-i} \To \Omega^1[1] \otimes \mathcal{I}^{\otimes n}
    \]
    be the projection given by $F^1\mathcal{J} \to F^1\mathcal{J}/F^2\mathcal{J} \cong \Omega^1[1] \otimes \mathcal{A} \to \Omega^1[1] \otimes \mathcal{I}$ (cf. Proposition \ref{prop:koszul-log})  on the $i$th factor and by $\mathcal{J} \to \mathcal{J}/F^1\mathcal{J} \cong \mathcal{I}$ on the other factors, composed with $\sigma_{n,i}$. Then:
    \begin{enumerate}[(i)]
        \item The sum
        \[
            \pi_n = \sum_{1\le i \le n}\pi_{n,i} : \bigoplus_{i=1}^n\mathcal{J}^{\otimes i-1}\otimes F^1\mathcal{J} \otimes \mathcal{J}^{\otimes n-i} \To \Omega^1[1] \otimes \mathcal{I}^{\otimes n}
        \]
        factors through an $\mathcal{O}_S$-linear map
        \[
            \pi_n : \sum_{i=1}^n \im(\mathcal{J}^{\otimes_{\Omega} i-1}\otimes_{\Omega} F^1\mathcal{J} \otimes_{\Omega} \mathcal{J}^{\otimes_{\Omega} n-i} \To \mathcal{J}^{\otimes_{\Omega} n}) \longrightarrow \Omega^1[1] \otimes \mathcal{I}^{\otimes n}
        \]
        \item The induced $\mathcal{O}_S$-linear map
        \[
            \pi: F^1B_{\Omega}(\mathcal{C}) \To \Omega^1[1] \otimes B(\mathcal{A}) \cong \Omega^1\otimes B(\mathcal{A})[-1]
        \]
        is a morphism of complexes over $k$.
    \end{enumerate}
\end{lemma}

\begin{proof}
    For clarity, we only give details for the case $n=2$; the general case is similar.

    To prove (i), we first remark that, from the flatness of $F^1\mathcal{J}$, $\mathcal{J}$, and $\mathcal{J}/F^1\mathcal{J} \cong \mathcal{I}$, we obtain an exact sequence
    \[
        \begin{tikzcd}
            0 \arrow{r} & (F^1\mathcal{J})^{\otimes 2} \arrow{r} & (F^1\mathcal{J}\otimes \mathcal{J})\oplus (\mathcal{J} \otimes F^1\mathcal{J}) \arrow{r} & \mathcal{J}^{\otimes 2},
        \end{tikzcd}
    \]
    where the injection is given by $x\mapsto (x,-x)$. For $i=1,2$, we have $\pi_{2,i}((F^1\mathcal{J})^{\otimes 2}) = 0$, since there is always a projection $\mathcal{J} \to \mathcal{J}/F^1\mathcal{J}$ in one of the factors. This shows that  $\pi_2$ factors through
    \[
        \pi_2: \im((F^1\mathcal{J}\otimes \mathcal{J}) \oplus (\mathcal{J}\otimes F^1\mathcal{J})  \To \mathcal{J}^{\otimes 2})= F^1\mathcal{J}\otimes \mathcal{J} + \mathcal{J}\otimes F^1\mathcal{J} \To \Omega^1[1]\otimes \mathcal{I}^{\otimes 2}.
    \]
    We are left to show that $\pi_2$ factors through the image of $\mathcal{F}^1\mathcal{J}\otimes \mathcal{J} + \mathcal{J} \otimes F^1\mathcal{J}$ in $\mathcal{J}^{\otimes_{\Omega}2}$. By definition, the kernel of $\mathcal{J}^{\otimes 2} \to \mathcal{J}^{\otimes_{\Omega}2}$ is generated by sections of the form
    \[
        x = (c_1\wedge \omega) \otimes c_2 - c_1 \otimes(J\omega \wedge c_2),
    \]
    with $c_1,c_2$ sections of $\mathcal{J}$ and $\omega$ a section of $\Omega^{n}$ with $n\ge 1$. Every such element is in $\mathcal{F}^1\mathcal{J}\otimes \mathcal{J} + \mathcal{J} \otimes F^1\mathcal{J}$, so we only need to prove that $\pi_2(x) = 0$. Using that $\Omega^i = \bigwedge^i \Omega^1$, we can reduce to the case where $\omega$ is a section of $\Omega^1$: if $\omega = \omega'\wedge \omega''$, then we can write
    \begin{align*}
        x =& ((c_1\wedge \omega')\wedge \omega'')\otimes c_2 - (c_1\wedge \omega')\otimes (J\omega''\wedge c_2)\\
        &+ (c_1\wedge \omega') \otimes (J\omega''\wedge c_2) - c_1 \otimes (J\omega'\wedge (J\omega'' \wedge c_2)).
    \end{align*}
    Finally, assuming that $\omega$ is a section of $\Omega^1$, we have $x = (c_1\wedge \omega) \otimes c_2 + c_1 \otimes(\omega \wedge c_2)$, so that
    \begin{align*}
        \pi(x) &= \pi_{2,1}((c_1\wedge \omega) \otimes c_2) + \pi_{2,2}(c_1\otimes (\omega \wedge c_2))\\
           &= \pi_{2,1}((\omega \wedge Jc_1) \otimes c_2) + \pi_{2,2}(c_1\otimes (\omega \wedge c_2))\\
           &= \omega \otimes (Jc_1\otimes c_2) - \omega \otimes (Jc_1\otimes c_2) = 0.
    \end{align*}
    This ends the proof of (i).

    To prove (ii), we simply compute:
    \begin{align*}
        d_{B}([\omega \wedge c_1 | c_2 |\cdots |c_n]) =& -[d\omega \wedge c_1 - \omega \wedge dc_1 | c_2 |\cdots |c_n]\\ & - \sum_{i=2}^n(-1)^i[\omega \wedge Jc_1|\cdots |Jc_{i-1}|dc_i|c_{i+1}| \cdots |c_n] \\
        &- \sum_{i=1}^n(-1)^{i-1}[\omega \wedge Jc_1|\cdots |Jc_{i-1}|Jc_i \wedge c_{i+1}|c_{i+2}|\cdots |c_n]
    \end{align*}
    Thus,
    \begin{align*}
        \pi_2(d_{B}([\omega \wedge c_1 |c_2| \cdots |c_n])) = -(\id\otimes d_B)(\pi_2([\omega \wedge c_1 |c_2| \cdots |c_n])).
    \end{align*}
    This finishes the proof, as every section of $F^1B_{\Omega}(\mathcal{C})$ is a combination of sections of the form $[\omega \wedge c_1 |c_2| \cdots |c_n]$.
\end{proof}

\subsection{Gauss--Manin connection on the fundamental Hopf algebra}\label{par:GM-connection-fundamental}

Consider the splitting
\[
    \begin{tikzcd}
        0 \arrow{r} & \Omega^1 \arrow{r} & \mathcal{C}^1 \arrow{r} & \mathcal{A}^1 \arrow{r}\arrow[bend right]{l} & 0,
    \end{tikzcd}
\]
induced by the canonical lift of Kronecker differentials, as in Section \ref{sec:canonical-lifts}. Locally,
\[
    \nu \longmapsto \widetilde{\nu}\text{, }\qquad \omega^{(n)}_P\longmapsto \widetilde{\omega}^{(n)}_P\text{, }\qquad n\ge 0\text{, } P \in Z(S).
\]
It induces a splitting
\[
    \begin{tikzcd}
        0 \arrow{r} & F^1B^0_{\Omega}(\mathcal{C}) \arrow{r} & B^0_{\Omega}(\mathcal{C}) \arrow{r} & B^0(\mathcal{A}) \arrow{r}\arrow[bend right]{l}[swap]{s} & 0.
    \end{tikzcd}
\]
given by
\[
    s[a_1|\cdots |a_n]= [\widetilde{a}_1|\cdots |\widetilde{a}_n].
\]

\begin{lemma}
    If $\xi$ is a section of $H^0(B(\mathcal{A}))$, then $d_{B}\circ s (\xi) \in F^1B^1_{\Omega}(\mathcal{C})$.
\end{lemma}

\begin{proof}
    This follows immediately from the commutative diagram with exact rows
    \[
        \begin{tikzcd}
            0 \arrow{r} & F^1B^0_{\Omega}(\mathcal{C})\arrow{d} \arrow{r} & B^0_{\Omega}(\mathcal{C})\arrow{d}{d_{B}} \arrow{r} & B^0(\mathcal{A})\arrow{d}{d_B} \arrow{r}\arrow[bend right]{l}[swap]{s} & 0\\
            0 \arrow{r} & F^1B^1_{\Omega}(\mathcal{C}) \arrow{r} & B^1_{\Omega}(\mathcal{C}) \arrow{r} & B^1(\mathcal{A}) \arrow{r} & 0
        \end{tikzcd}
    \]
    given by Lemma \ref{lemma:sequence-bar}.
\end{proof}

Thus, we can define a $k$-linear map
\[
    \delta : H^0(B(\mathcal{A})) \To \Omega^1\otimes B^0(\mathcal{A}) \text{, }\qquad \delta = -\pi \circ d_{B}\circ s.
\]

\begin{lemma}
    The image of $\delta$ is contained in $\Omega^1\otimes H^0(B(\mathcal{A}))$.
\end{lemma}

\begin{proof}
    By definition, $(\id \otimes d_B)\circ \delta = -(\id \otimes d_B) \circ \pi \circ d_B \circ s$. Then, using that $\pi$ is a morphism of $k$-complexes (part (ii) of Lemma \ref{lemma:projection}), we get  $-(\id \otimes d_B) \circ \pi \circ d_B \circ s = \pi \circ d_B\circ d_B \circ s  = 0$.
\end{proof}

In the notation of Definition \ref{def:fundamental-hopf}, we obtain a $k$-linear map
\[
    \delta : \mathcal{H}_{E/S,Z} \To \Omega_{S/k}^1 \otimes \mathcal{H}_{E/S,Z}.
\]

\begin{theorem}\label{thm:non-abelian-gm}
    The above defined map $\delta$ is an integrable $k$-connection on the $\mathcal{O}_S$-module $\mathcal{H}_{E/S,Z}$. Moreover:
    \begin{enumerate}[(i)]
        \item The connection $\delta$ preserves the length filtration.
        \item For every $n\ge 1$, the induced connection on $L_n \mathcal{H}_{E/S,Z}/L_{n-1}\mathcal{H}_{E/S,Z}$ gets identified with the $n$th tensor power of the Gauss--Manin connection on $H^1_{\dR}((E\setminus Z)/S)$ under the isomorphism \eqref{eq:isomorphism-graded-length}. In particular, $\delta$ is regular singular at infinity.
        \item The deconcatenation coproduct $\Delta$ and the antipode  $\sigma$ are horizontal for $\delta$, namely:
        \[
            (\id \otimes \Delta)\circ \delta = (\delta\otimes \id + \id \otimes \delta)\circ \Delta\text{, }\qquad\delta \circ  \sigma = (\id \otimes \sigma) \circ \delta.
        \]
        \item The formation of $\delta$ is compatible with base change of the form $S'\to S$, where $S'$ is a smooth $k$-scheme, and with extension of scalars $k\subset k'$.
   \end{enumerate}
\end{theorem}

\begin{proof}
    To show that $\delta$ is a connection, let $\gamma = \sum_{j}[c_{1}^j|\cdots |c_{n_j}^j]$ be a section of $B^0_{\Omega}(\mathcal{C})$ such that $d_{B}\gamma$ lies in $F^1B^1_{\Omega}(\mathcal{C})$. For a section $r$ of $\mathcal{O}_S$, we have
    \[
        d_{B}(r\gamma) = \sum_{j}d_{B}[rc_{1}^j|\cdots |c_{n_j}^j] =-\sum_{j}[dr\wedge c_{1}^j|\cdots |c_{n_j}^j] +rd_{B}\gamma. 
    \]
    In particular, $d_{B}(r\gamma)$ is also a section of $F^1B^1_{\Omega}(\mathcal{C})$, and
    \[
        \pi \circ d_{B}(r\gamma) = -dr\otimes \overline{\gamma} + r \pi \circ d_B(\gamma),
    \]
    where $\overline{\gamma}$ denotes the image of $\gamma$ in $B^0(\mathcal{A})$. If $\gamma = s(\xi)$ for some section $\xi$ of $\mathcal{H}_{E/S,Z} = H^0(B(\mathcal{A}))$, this shows that
    \[
        \delta(r\xi) = dr\otimes \xi + r\delta(\xi).
    \]
    Thus, $\delta$ is a $k$-connection.

    Note that the definition of $\delta$ only involves the canonical splitting $s$, which is compatible with base change, and natural operations on the bar construction. Therefore, (iv) follows from Corollary \ref{coro:base-change}. The compatibility with base change immediately implies the integrability of $\delta$, since the moduli stack of elliptic curves is a 1-dimensional smooth Deligne--Mumford stack, and every connection on a smooth curve is integrable.

    Properties (i) and (iii) are straightforward to verify. We are left to prove (ii). Recall that, by Katz and Oda's construction \cite{KO68}, the Gauss--Manin connection $\nabla$ on $H^1_{\dR}((E\setminus Z)/S)$ can be explicitly described as follows. Under the isomorphism
    \[
        H^1_{\dR}((E\setminus Z)/S)\cong  H^1(\mathcal{A}) = \ker (d:\mathcal{A}^1 \longrightarrow \mathcal{A}^2)
    \]
    of Proposition \ref{prop:model-derham} (cf. \S\ref{par:dg-algebra-relative}), for a cohomology class given by a closed differential form $\omega$ in $\mathcal{A}^1$, let $\omega'$ be any lift of $\omega$ to an absolute differential form in  $\mathcal{C}^1$; then
    \begin{equation}\label{eq:gauss-manin-katz-oda}
        \nabla \omega = d\omega' \mod f_*\mathcal{F}^2 = \sum_{i=1}^n\alpha_i \otimes \omega_i
    \end{equation}
    for unique $\alpha_i$ in $\Omega^1$ and $\omega_i$ in $\mathcal{A}^1$ closed (see Proposition \ref{prop:koszul-log}).

    Now, given $[\omega_1| \cdots |  \omega_n]$ in $H^1(\mathcal{A})^{\otimes n}$, let $\xi$ be a section of $L_n\mathcal{H}_{E/S,Z}$ mapping to $\omega_1\otimes \cdots\otimes \omega_n$ under the isomorphism \eqref{eq:isomorphism-graded-length}. Since $\mathcal{H}_{E/S,Z}\subset B^0(\mathcal{A}) = T^c \mathcal{A}^1$, we can write $\xi = \xi_n + \xi_{n-1} + \cdots +\xi_0$, where each $\xi_i$ is of pure length $i$, and $\xi_n = [\omega_1 | \cdots |\omega_n]$. We have
    \[
        (-d_B \circ s)(\xi) = \sum_{i=1}^n[\widetilde{\omega}_1|\cdots |d\widetilde{\omega}_i |\cdots |\widetilde{\omega}_n] + \text{lower length}.
    \]
    Then, since each $\widetilde{\omega}_i$ is a lift of $\omega_i$ to an absolute differential in $\mathcal{C}^1$, it follows from \eqref{eq:gauss-manin-katz-oda} and from the definition of $\pi$ that
    \[
        \delta(\xi) = \nabla^{\otimes n}([\omega_1|\cdots| \omega_n]) + \text{lower length}.
    \]
    The last statement of (ii) follows from the regularity at infinity of the Gauss--Manin connection \cite[II, \S7]{deligne70} and from the fact that regularity is preserved by extensions \cite[II, Proposition 4.6.(i)]{deligne70}.
\end{proof}

\subsection{Lifting the relative elliptic KZB connection}\label{par:universal-kzb}

The absolute elliptic KZB connection will be given by combining the relative connection  $\nabla_{E^{\natural}/S,Z}$ (Definition \ref{def:vertical-kzb}) with the dual of $\delta$:
\[
    \delta^{\vee} : \mathcal{H}^{\vee}_{E/S,Z} \longrightarrow \Omega^1_{S/k}\hat{\otimes}\mathcal{H}^{\vee}_{E/S,Z}.
\]
For this, let
\[
    \widetilde{\omega}_{E^{\natural}/S,Z} = (s \otimes \id)(\omega_{E^{\natural}/S,Z}) \in \Gamma(S, f_*\Omega^1_{E^{\natural}/k}(\log \pi^{-1}Z) \hat{\otimes} \mathcal{H}_{E/S,Z}^{\vee})
\]
be the canonical lift of the KZB form, and let it act on $\mathcal{H}_{E/S,Z}^{\vee}$ by left multiplication. 

\begin{definition}
    The \emph{absolute elliptic KZB connection} of $E/S/k$ punctured at $Z$ is the $k$-connection
    \[
        \nabla_{E^{\natural}/S/k,Z} : f^*\mathcal{H}^{\vee}_{E/S,Z} \longrightarrow \Omega^1_{E^{\natural}/k}(\log \pi^{-1}Z) \hat{\otimes} f^*\mathcal{H}^{\vee}_{E/S,Z}\text{, }\qquad \nabla_{E^{\natural}/S/k,Z}= f^*\delta^{\vee} + \widetilde{\omega}_{E^{\natural}/S,Z}.
    \]
\end{definition}

\begin{proposition}\label{prop:base-change-kzb}
    The formation of $\nabla_{E^{\natural}/S/k,Z}$ is compatible with every base change of the form $S' \to S$, where $S'$ is a smooth $k$-scheme, and with extension of scalars $k\subset k'$.
\end{proposition}

\begin{proof}
    This follows immediately from Proposition \ref{prop:base-change-vkzb} and Theorem \ref{thm:non-abelian-gm}.
\end{proof}

To prove integrability, consider the following lemma.

\begin{lemma}\label{lemma:integrability}
    Let $A$ be a $k$-algebra (not necessarily commutative) and $(\Omega,d,\wedge)$ be a dg-algebra over $k$. Let $\varphi \in \Omega^1 \otimes \mathrm{Der}_k(A)$ and $\alpha \in \Omega^1 \otimes A$. We identify $A$ as a subspace of $\End_k(A)$ by left multiplication. Then, the following equation holds in $\Omega^2 \otimes \End_k(A)$:
    \[
        d(\varphi + \alpha) + (\varphi + \alpha)\wedge (\varphi+\alpha) = (d\varphi + \varphi \wedge \varphi) + (d\alpha + \alpha \wedge \alpha) + \varphi(\alpha),
    \]
    where $\varphi(\alpha)$ is the element of $\Omega^2 \otimes A$ given by `evaluating' $\varphi$ at $\alpha$.
\end{lemma}

\begin{proof}
    It suffices to prove that $\varphi \wedge \alpha + \alpha \wedge \varphi = \varphi(\alpha)$. Since this equation is linear in $\varphi$ and in $\alpha$, we can assume that $\varphi = \omega \otimes \partial$ and $\alpha = \eta \otimes a$. Then,
    \begin{align*}
        \varphi \wedge \alpha + \alpha \wedge \varphi &= \omega \wedge \eta \otimes (\partial \circ a) + \eta \wedge \omega \otimes (a \circ \partial)\\ &= \omega \wedge \eta \otimes (\partial \circ a - a \circ \partial) = \omega \wedge \eta \otimes \partial(a) = \varphi(\alpha),
    \end{align*}
    where we used that $\partial$ is a derivation in the penultimate equality above.
\end{proof}

\begin{theorem}\label{thm:integrability}
    The $k$-connection $\nabla_{E^{\natural}/S/k,Z}$ is integrable. 
\end{theorem}

\begin{proof}
    We may work locally over $S$ and assume that $R^1p_*\mathcal{O}_E$ is trivial. Moreover, by Proposition \ref{prop:base-change-kzb}, we can also assume that $\dim S = 1$ (cf. proof of Theorem \ref{thm:non-abelian-gm}). Then, with notation as in \eqref{eq:triv-A1} and \S\ref{par:kzb-connection}, we identify
    \[
        \mathcal{H}^{\vee}_{E/S,Z} \cong \mathcal{O}_S \hat{\otimes}A\text{, }\qquad A = \frac{ k \langle \!\langle a,b,c_P : P \in Z(S)\rangle \! \rangle}{\langle \sum_{P \in Z(S)}c_P - [a,b] \rangle}
    \]
    and we write
    \[
        \delta^{\vee} = d + \Phi\text{, }\qquad \Phi \in \Gamma(S,\Omega^1_{S/k})\otimes \End_k(A),
    \]
    so that $\nabla_{E^{\natural}/S/k,Z}$ is a $k$-connection on $\mathcal{O}_{E^{\natural}}\hat{\otimes}A$, given by
    \[
        \nabla_{E^{\natural}/S/k,Z} = d + \Phi + \widetilde{\omega}_{E^{\natural}/S,Z}.
    \]
    Since the multiplication in $\mathcal{H}^{\vee}_{E/S,Z}$ is given by the dual of the deconcatenation coproduct $\Delta$, it follows from Theorem \ref{thm:non-abelian-gm} that $\Phi \in \Gamma(S,\Omega^1_{S/k})\otimes \mathrm{Der}_k(A)$. Moreover, since  $\delta^{\vee}$ is integrable, we conclude from Lemma \ref{lemma:integrability} that the integrability of $\nabla_{E^{\natural}/S/k,Z}$ is equivalent to the equation
    \begin{equation}\label{eq:integrability-kzb}
        d\widetilde{\omega}_{E^{\natural}/S,Z} + \widetilde{\omega}_{E^{\natural}/S,Z}\wedge \widetilde{\omega}_{E^{\natural}/S,Z} + \Phi(\widetilde{\omega}_{E^{\natural}/S,Z}) = 0.
    \end{equation}

    Recall from \S\ref{par:kzb-connection} that $\hat{\sigma}$ is the element of $\Gamma(S,\mathcal{H}_{E/S,Z})\hat{\otimes} A \cong \Gamma(S, \mathcal{H}_{E/S,Z}\hat{\otimes}\mathcal{H}_{E/S,Z}^{\vee})$ corresponding to the antipode $\sigma$ of $\mathcal{H}_{E/S,Z}$, and that $\omega_{E^{\natural}/S,Z} = (\pr_1\otimes \id)(\hat{\sigma})$. Since left multiplication by $\hat{\sigma}$ defines a (completed) $\mathcal{H}_{E/S,Z}$-comodule structure on $\mathcal{H}_{E/S,Z}^{\vee}$, it follows from Proposition \ref{prop:comodule-bar} that
    \[
        \hat{\sigma} = \sum_{n\ge 0}\underbrace{[\omega_{E^{\natural}/S,Z} | \cdots | \omega_{E^{\natural}/S,Z}]}_{\text{length }n}.
    \]
    Since $\sigma : \mathcal{H}_{E/S,Z} \to \mathcal{H}_{E/S,Z}$ is horizontal for the connection $\delta$ by Theorem \ref{thm:non-abelian-gm}, and since $\delta^{\vee}$ is defined as the dual of $\delta$, we have
    \begin{equation}\label{eq:sigma-horizontal}
        (\delta \otimes \id + \id \otimes \delta^{\vee})(\hat{\sigma}) = 0.
    \end{equation}
    On the one hand, using the definition of $\delta = -\pi \circ d_B \circ s$, we get:
    \[
        (\delta \otimes \id)(\hat{\sigma}) = \sum_{n\ge 1}\sum_{i=1}^n (\pi \otimes \id)(\underbrace{[\widetilde{\omega}_{E^{\natural}/S,Z}| \cdots | \overbrace{d\widetilde{\omega}_{E^{\natural}/S,Z} + \widetilde{\omega}_{E^{\natural}/S,Z}\wedge \widetilde{\omega}_{E^{\natural}/S,Z}}^{i\text{th position}} | \cdots | \widetilde{\omega}_{E^{\natural}/S,Z}]}_{\text{length }n})
    \]
    On the other hand, using that $\Phi$ has coefficients in $k$-derivations of $A$, we obtain:
    \[
        (\id \otimes\delta^{\vee})(\hat{\sigma}) = \sum_{n\ge 1}\sum_{i=1}^n\underbrace{[\omega_{E^{\natural}/S,Z}|\cdots | \overbrace{\Phi(\omega_{E^{\natural}/S,Z})}^{i\text{th position}}| \cdots| \omega_{E^{\natural}/S,Z}]}_{\text{length }n}.
    \]
    Thus, equation \eqref{eq:sigma-horizontal} is equivalent to
    \[
        (\pi\otimes \id) (d\widetilde{\omega}_{E^{\natural}/S,Z} + \widetilde{\omega}_{E^{\natural}/S,Z}\wedge \widetilde{\omega}_{E^{\natural}/S,Z} + \Phi(\widetilde{\omega}_{E^{\natural}/S,Z})) = 0.
    \]
    To conclude, we simply remark that the hypothesis $\dim S =1$ implies that $F^2\mathcal{J} = f_*\mathcal{F}^2 = 0$, so that $\pi$ is injective on $F^1\mathcal{J} = f_*\mathcal{F}^1$ (cf. Proposition \ref{prop:koszul-log}), which yields \eqref{eq:integrability-kzb}.
\end{proof}

\section{Analytic formulae}\label{sec:analytic-formuli}

\subsection{Kronecker differentials}\label{par:analytic-kronecker}

Let $E$ be an elliptic curve over $\mathbb{C}$ and let $\tau \in \mathfrak{H}$ be such that $E^{\an} \cong \mathbb{C}/(\mathbb{Z} + \tau \mathbb{Z})$, so that $H_1(E^{\an},\mathbb{Z}) \cong \mathbb{Z} + \tau \mathbb{Z}$. fder the basis $(\omega,\eta)$ of $H^1_{\dR}(E/\mathbb{C})$ given by
\begin{equation}\label{eq:analytic-omega-eta}
    \omega = dz\text{, }\qquad \eta = \left(\wp_{\tau}(z) + G_2(\tau)\right)dz,
\end{equation}
where $\wp_{\tau}(z)$ is the Weierstrass $\wp$-function associated to the lattice $\mathbb{Z}+\tau \mathbb{Z}$, and $G_2(\tau) = \sum_{n}\sum_{m}'(m+nz)^{-2}$ is the Eisenstein series of weight 2 and level 1. It follows from Example \ref{ex:uniformisation} and \cite[Lemma A1.3.9]{katz72} that
\begin{equation}\label{eq:explicit-uve-analytic}
    E^{\natural,\an} \cong \mathbb{C}^2/L_{\tau}\text{, }\qquad L_{\tau} = \{(m + n \tau, 2\pi i n) \in \mathbb{C}^2 : m,n \in \mathbb{Z}\}.
\end{equation}
If $z,w$ denote the coordinates on $\mathbb{C}^2$, then the basis $(\omega,\eta)$ of $H^1_{\dR}(E/\mathbb{C})$ corresponds under the isomorphism \eqref{eq:isom-cohomology-E} to the basis
\[
    \omega^{(0)}\defeq dz\text{, }\qquad \nu \defeq dw
\]
of $\Gamma(E^{\natural}, \Omega^1_{E^{\natural}/\mathbb{C}})$.

\begin{proposition}\label{prop:analytic-kronecker-diff}
    Let
    \[
        F_{\tau}(z,x) \defeq \frac{\theta'_{\tau}(0)\theta_{\tau}(z+x)}{\theta_{\tau}(z)\theta_{\tau}(x)},
    \]
    where
    \[
        \theta_{\tau}(z) = \sum_{n \in \mathbb{Z}}(-1)^n q^{\frac12\left(n + \frac12\right)^2}e^{ 2\pi i(n + \frac12)z}\text{, }\qquad q = e^{2\pi i \tau},
    \]
    denotes Jacobi's odd theta function. Consider the functions $\varphi^{(n)}_{\tau}(z,w)$ defined by the generating series
    \[
        e^{wx}F_{\tau}(z,x) = \sum_{n \ge 0}\varphi^{(n)}_{\tau}(z,w) x^{n-1}.
    \]
    Then the Kronecker differentials associated to $\nu=dw$ (Theorem \ref{thm:kronecker-differentials}) are given by
    \[
        \omega^{(n)} = \varphi^{(n)}_{\tau}(z,w)dz\text{, }\qquad n\ge 0.
    \]
\end{proposition}

\begin{proof}
    It follows from the transformation property
    \begin{equation}\label{eq:transformation-Kronecker}
        F_{\tau}(z+m+n\tau,x) = e^{-2\pi i n x}F_{\tau}(z,x)
    \end{equation}
    of the Kronecker function (cf. \cite[Equations (10) and (11)]{LR07}) that $\varphi^{(n)}_{\tau}(z,w)dz$ are well-defined 1-forms on $E^{\natural}$ with logarithmic poles along the divisor $\pi^{-1}O$ (which is explicitly given by the equation $z=0$ under \eqref{eq:explicit-uve-analytic}). A straightforward computation shows that they satisfy properties (i), (ii), and (iii) of Theorem \ref{thm:kronecker-differentials}. By unicity, we conclude that $\omega^{(n)} = \varphi^{(n)}_{\tau}(z,w)dz$. 
\end{proof}

\begin{corollary}\label{coro:omegaP-analytic}
    Let $P \in E(\mathbb{C})$ be a torsion point represented by $\alpha+\beta\tau$, for some $\alpha,\beta \in \mathbb{Q}$, under $E^{\an} \cong \mathbb{C}/(\mathbb{Z} + \tau \mathbb{Z})$. Then, for every $n\ge 0$,
    \[
        \omega_P^{(n)} = \varphi^{(n)}_{\tau}(z-\alpha-\beta\tau, w-2\pi i \beta)dz.
    \]
\end{corollary}

\begin{proof}
    It suffices to use \eqref{eq:defn-omegaP} and to note that $P^{\natural} \in E^{\natural}(\mathbb{C})$ is represented by $(\alpha+\beta\tau,2\pi i\beta)$ under \eqref{eq:explicit-uve-analytic}.
\end{proof}

\subsection{Analytic canonical lifts}\label{par:analytic-canonical-lifts}

We work in the category of complex analytic spaces. Let $p:\mathcal{E}_{\mathfrak{H}} \to \mathfrak{H}$ be the universal framed elliptic curve over the upper half-plane $\mathfrak{H}$, with fibre at $\tau \in \mathfrak{H}$ given by $p^{-1}(\tau) = \mathbb{C}/(\mathbb{Z} + \tau \mathbb{Z}) \eqdef \mathcal{E}_{\tau} $, and let $\mathcal{E}^{\natural}$ be its universal vector extension. Explicitly, we can write
\[
    \mathcal{E}_{\mathfrak{H}}^{\natural} \cong (\mathbb{C}^2 \times \mathfrak{H}) / L,
\]
where $L \to \mathfrak{H}$ is the `relative lattice' with fibre at $\tau \in \mathfrak{H}$ given by $L_{\tau}$ as in \eqref{eq:explicit-uve-analytic}.  We denote by $f: \mathcal{E}_{\mathfrak{H}}^{\natural} \to \mathfrak{H}$ the structure map, and by $\pi : \mathcal{E}_{\mathfrak{H}}^{\natural} \to \mathcal{E}_{\mathfrak{H}}$ is the canonical projection, which in this case is induced by $(z,w)\mapsto z$. Note that the divisor $\pi^{-1}O$ is given by $z=0$, and the identity section $e\in \mathcal{E}_{\mathfrak{H}}^{\natural}(\mathfrak{H})$ is given by $z=w=0$.

In  \S\ref{par:analytic-kronecker}, we have described the differentials
\begin{equation}\label{eq:relative-diff-analytic}
    \nu = dw\text{, }\qquad  \omega^{(0)} = dz\text{, } \qquad  \omega^{(1)} = \varphi_{\tau}^{(1)}(z,w)dz\text{, }\qquad \ldots 
\end{equation}
which we now regard as global sections of $f_*\Omega^1_{\mathcal{E}_{\mathfrak{H}}^{\natural}/\mathfrak{H}}(\log \pi^{-1}O)$. Next, we describe their canonical lifts to `absolute differentials' on $\mathcal{E}_{\mathfrak{H}}^{\natural}$, i.e., global sections of $f_*\Omega^1_{\mathcal{E}_{\mathfrak{H}}^{\natural}}(\log \pi^{-1}O)$. As our algebraic results do not immediately apply to the analytic category due to the failure of GAGA (cf. Remark \ref{rmk:gaga}), we shall first define these canonical lifts by explicit formulas, and then characterise them via unicity statements which mimic their algebraic counterparts. Set
\[
    \widetilde{\nu} \defeq dw\text{, }\qquad \widetilde{\omega}^{(0)} \defeq dz - w\frac{d\tau}{2\pi i}.
\]

\begin{proposition}\label{prop:lift-regular-analytic}
    The above formulas yield well-defined absolute 1-forms $\widetilde{\nu},\widetilde{\omega}^{(0)} \in \Gamma(\mathcal{E}_{\mathfrak{H}}^{\natural}, \Omega^1_{\mathcal{E}_{\mathfrak{H}}^{\natural}})$ lifting the relative 1-forms $\nu,\omega^{(0)} \in \Gamma(\mathcal{E}_{\mathfrak{H}}^{\natural},\Omega^1_{\mathcal{E}^{\natural}_{\mathfrak{H}}/\mathfrak{H}})$. Moreover, we have
    \begin{equation}\label{eq:restriction-analytic}
        e^*\widetilde{\omega}^{(0)} = e^*\widetilde{\nu} = 0
    \end{equation}
    and
    \begin{equation}\label{eq:GM-analytic}
        d\widetilde{\omega}^{(0)} = \frac{d\tau}{2\pi i} \wedge \widetilde{\nu}\text{, } \qquad d\widetilde{\nu} = 0,
    \end{equation}
    and these properties characterise $\widetilde{\nu},\widetilde{\omega}^{(0)}$ uniquely among lifts of $\nu,\omega^{(0)}$.
\end{proposition}

\begin{proof}
    For the first claim, it suffices to check that the above formulas for $\widetilde{\nu}$ and $\widetilde{\omega}^{(0)}$ are invariant under the action of the lattice $L$. For instance:
    \[
        d(z + m+ n \tau) - (w + 2\pi i n) \frac{d\tau}{2\pi i} = dz - w \frac{d\tau}{2\pi i}.
    \]
    Equations \eqref{eq:restriction-analytic} and \eqref{eq:GM-analytic} are straightforward from the explicit formulas. The last claim follows from the following fact: if $\gamma$ is a global section of $\Omega^1_{\mathcal{E}^{\natural}}$ of the form
    \[
        \gamma = f(z,w,\tau)d\tau\text{, }\qquad f \in \Gamma(\mathcal{E}_{\mathfrak{H}}^{\natural},\mathcal{O}_{\mathcal{E}_{\mathfrak{H}}^{\natural}}),
    \]
    satisfying
    \[
        0 = e^*\gamma = f(0,0,\tau)d\tau
    \]
    and
    \[
        0 = d\gamma = \frac{\partial f}{\partial z}(z,w,\tau)dz\wedge d\tau + \frac{\partial f}{\partial w}(z,w,\tau)dw\wedge d\tau
    \]
    then $\gamma =0$. 
\end{proof}

\begin{remark}\label{rmk:gauss-manin-analytic-frame}
    Note that $(\omega^{(0)},\nu)$ corresponds to the frame $(\omega,\eta) = (dz, (\wp_{\tau}(z) + G_2(\tau))dz)$ of the analytic de Rham cohomology $H^1_{\dR}(\mathcal{E}_{\mathfrak{H}}/\mathfrak{H})$. The Gauss--Manin connection in this frame is given by (cf. \cite[\S A1]{katz72})
    \[
        \nabla \omega = \frac{d\tau}{2\pi i}\otimes \eta\text{, }\qquad \nabla \eta = 0,
    \]
    so that equations \eqref{eq:GM-analytic} are the analytic versions of \eqref{eq:GM-canonical-lifts}.
\end{remark}

Define lifts of $\omega^{(n)}$ ($n\ge 1$) to absolute logarithmic 1-forms by
\[
    \widetilde{\omega}^{(n)} \defeq \varphi^{(n)}_{\tau}(z,w)\left(dz - w \frac{d\tau}{2\pi i} \right) + n\varphi^{(n+1)}_{\tau}(z,w) \frac{d\tau}{2\pi i},
\]
where $\varphi^{(n)}_{\tau}(z,w)$ is as in Proposition \ref{prop:analytic-kronecker-diff}.

\begin{proposition}\label{prop:lift-kronecker-analytic}
    For every $n\ge 1$, the above formula yields a well-defined absolute logarithmic 1-form $\widetilde{\omega}^{(n)} \in \Gamma(\mathcal{E}_{\mathfrak{H}}^{\natural},\Omega^1_{\mathcal{E}_{\mathfrak{H}}^{\natural}}(\log \pi^{-1}O))$ which lifts the relative logarithmic 1-form $\omega^{(n)}$. Moreover, we have
    \begin{equation}\label{eq:lift-kronecker-analytic}
        \widetilde{\omega}^{(n)} \wedge \widetilde{\nu}\wedge \widetilde{\omega}^{(0)} = n\frac{d\tau}{2\pi i}\wedge \widetilde{\nu}\wedge \widetilde{\omega}^{(n+1)}
    \end{equation}
    and this property characterises $\widetilde{\omega}^{(n)}$ uniquely among lifts of $\omega^{(n)}$.
\end{proposition}

\begin{proof}
    The first claim follows from Proposition \ref{prop:lift-regular-analytic} and equation \eqref{eq:transformation-Kronecker}. Equation \eqref{eq:lift-kronecker-analytic} follows immediately from the explicit formulas. To show that it characterises $\widetilde{\omega}^{(n)}$ uniquely among lifts of $\omega^{(n)}$, we let
    \[
        \widetilde{\omega}^{(n)}{}' = \widetilde{\omega}^{(n)} + f_n(z,w,\tau)d\tau\text{, }\qquad f_n(z,w,\tau) \in \Gamma(\mathcal{E}_{\mathfrak{H}}^{\natural},\mathcal{O}_{\mathcal{E}_{\mathfrak{H}}^{\natural}}),
    \]
    be other lifts of $\omega^{(n)}$ satisfying \eqref{eq:lift-kronecker-analytic}. Then
    \[
        (\widetilde{\omega}^{(n)} +f_n(z,w,\tau)d\tau)\wedge \widetilde{\nu}\wedge \widetilde{\omega}^{(0)} = n\frac{d\tau}{2\pi i}\wedge \widetilde{\nu}\wedge (\widetilde{\omega}^{(n+1)}+f_{n+1}(z,w,\tau)d\tau),
    \]
    if and only if
    \[
         f_n(z,w,\tau)d\tau\wedge dw\wedge dz = 0,
    \]
    so that $f_n =0$ for every $n\ge 1$.
\end{proof}

As an application, we can use the above unicity statements to prove an algebraicity result.

\begin{proposition}\label{prop:algebraicity-kronecker}
    Consider the pullback diagram
    \begin{equation}\label{eq:pullback-weierstrass-EC}
        \begin{tikzcd}
            \mathcal{E}_{\mathfrak{H}}\arrow{d} \arrow{r}{\psi}& E^{\an}\arrow{d}\\
            \mathfrak{H} \arrow{r}[swap]{s} & S^{\an}\arrow[phantom]{lu}{\square}
        \end{tikzcd}
    \end{equation}
    where
    \begin{itemize}
        \item $S = \Spec \mathbb{C}[g_2,g_3,(g_2^3-27g_3^2)^{-1}]$, and $E/S$ is the universal Weierstrass elliptic curve given by the equation $y^2z = 4x^3-g_2xz^2-g_3z^3$,
        \item the map $s:\mathfrak{H} \to S^{\an}$ is given by $s(\tau) = (g_2(\tau),g_3(\tau))$, where $g_2(\tau) = 60\sum_{m,n}'(m+n\tau)^{-4}$ and $g_3(\tau) = 140\sum_{m,n}'(m+n\tau)^{-6}$,
        \item the map $\psi:\mathcal{E}_{\mathfrak{H}}\to E^{\an}$ is given on each fibre by
        \[
            \psi_{\tau} :\mathcal{E}_{\tau} = \mathbb{C}/(\mathbb{Z} + \tau \mathbb{Z}) \stackrel{\sim}{\To} E^{\an}_{s(\tau)}\text{, }\qquad [z] \longmapsto 
            \begin{cases}
                (\wp_{\tau}(z):\wp_{\tau}'(z):1) & [z]\neq 0\\
                (0:1:0) & [z]= 0. 
            \end{cases}
        \] 
\end{itemize}
Consider the frame $(\omega_{\alg},\eta_{\alg}) = (dx/y,xdx/y)$ of the algebraic de Rham cohomology $H^1_{\dR}(E/S)$, and let $\omega_{\alg}^{(n)} \in \Gamma(E^{\natural},\Omega^1_{E^{\natural}/S}(\log \pi^{-1}O))$ and $\widetilde{\omega}_{\alg}^{(n)} \in \Gamma(E^{\natural}, \Omega^1_{E^{\natural}/\mathbb{C}}(\log \pi^{-1}O))$ ($n\ge 1$) be the corresponding Kronecker differentials and their canonical lifts (cf. Remark \ref{rmk:pairing}). Then,
\[
    \psi^*\omega_{\alg}^{(n)} = \omega^{(n)}\text{, } \qquad \psi^*\widetilde{\omega}_{\alg}^{(n)} = \widetilde{\omega}^{(n)}
\]
for every $n\ge 1$.
\end{proposition}

\begin{proof}
    Let $(\omega_{\alg}^{(0)},\nu_{\alg})$ be the frame of $f_*\Omega^1_{E^{\natural}/S}$ corresponding to $(\omega_{\alg},\eta_{\alg})$. Since the frame $(\omega^{(0)},\nu)$ of $f_*\Omega^1_{\mathcal{E}_{\mathfrak{H}}/\mathfrak{H}}$ corresponds to $(\omega,\eta)$, which is given by the formula \eqref{eq:analytic-omega-eta}, we obtain
    \[
        \psi^*\omega_{\alg}^{(0)} = dz = \omega^{(0)}\text{, }\qquad \psi^*\nu_{\alg} = dw - G_2dz = \nu - G_2\omega^{(0)}.
    \]
    It follows from Remark \ref{eq:change-basis} and from the unicity part of Theorem \ref{thm:kronecker-differentials} that $\psi^*\omega_{\alg}^{(n)}$ and $\omega^{(n)}$ agree in every fibre of $f:\mathcal{E}_{\mathfrak{H}}^{\natural} \to \mathfrak{H}$, so that
    \[
        \psi^*\omega_{\alg}^{(n)} = \omega^{(n)}
    \]
    for every $n\ge 1$.

    Let $(\alpha_{ij})_{1\le i,j\le 2}$ be the matrix of the Gauss--Manin connection $\nabla$ on $H^1_{\dR}(E/S)$ in the frame $(\omega_{\alg},\eta_{\alg})$, so that
    \[
        \nabla \omega_{\alg} = \alpha_{11}\otimes \omega_{\alg} + \alpha_{21}\otimes \eta_{\alg}
    \]
    \[
        \nabla \eta_{\alg} = \alpha_{12}\otimes \omega_{\alg} + \alpha_{22}\otimes \eta_{\alg}.
    \]
    It follows from Remark \ref{rmk:gauss-manin-analytic-frame} that the Gauss--Manin connection $\nabla$ on $H^1_{\dR}(\mathcal{E}_{\mathfrak{H}}/\mathfrak{H})$ satisfies
    \[
        \nabla \psi^*\omega_{\alg} = \nabla \omega = \frac{d\tau}{2\pi i}\otimes \eta = \frac{d\tau}{2\pi i} \otimes (\psi^*\eta_{\alg} + G_2 \psi^*\omega_{\alg}) = G_2 \frac{d\tau}{2\pi i}\otimes \psi^*\omega_{\alg} + \frac{d\tau}{2\pi i}\otimes \psi^*\eta_{\alg}
    \]
    \[
        \nabla \psi^*\eta_{\alg} = \nabla (\eta - G_2\omega) = -dG_2\otimes \omega -G_2\frac{d\tau}{2\pi i}\otimes \eta = \left(-dG_2 -G_2^2\frac{d\tau}{2\pi i} \right)\otimes \psi^*\omega_{\alg} - G_2\frac{d\tau}{2\pi i}\otimes \psi^*\eta_{\alg}.
    \]
    Since the formation of the Gauss--Manin connection commutes with base change, we conclude from the above equations that
    \[
        \begin{pmatrix}
            s^*\alpha_{11} & s^*\alpha_{12}\\
            s^*\alpha_{21} & s^*\alpha_{22}
        \end{pmatrix}
        =
        \begin{pmatrix}
          G_2\frac{d\tau}{2\pi i} & -dG_2 - G_2^2\frac{d\tau}{2\pi i}\\
            \frac{d\tau}{2\pi i} & -G_2\frac{d\tau}{2\pi i}
        \end{pmatrix}.
    \]
    Then, using the equations of Example \ref{ex:gauss--manin}, one can check that
    \[
        d(\psi^*\widetilde{\omega}^{(0)}_{\alg}) = \frac{d\tau}{2\pi i}\wedge (\psi^*\widetilde{\nu}_{\alg} + G_2\widetilde{\omega}^{(0)})\text{, } \qquad d(\psi^*\widetilde{\nu}_{\alg} + G_2\widetilde{\omega}^{(0)}) = 0,
    \]
    and we conclude from Proposition \ref{prop:lift-regular-analytic} that
    \[
        \psi^*\widetilde{\omega}^{(0)}_{\alg} = \widetilde{\omega}^{(0)}\text{, }\qquad \psi^*\widetilde{\nu}_{\alg} = \widetilde{\nu} - G_2 \widetilde{\omega}^{(0)}.
    \]
  
    By pulling back the defining equation for $\widetilde{\omega}_{\alg}^{(n)}$ in Theorem \ref{thm:lift-kronecker}, we obtain
    \[
        \psi^*\widetilde{\omega}_{\alg}^{(n)} \wedge (\widetilde{\nu} - G_2 \widetilde{\omega}^{(0)})\wedge \widetilde{\omega}^{(0)} = n\frac{d\tau}{2\pi i} \wedge (\widetilde{\nu} - G_2 \widetilde{\omega}^{(0)}) \wedge \psi^*\widetilde{\omega}_{\alg}^{(n+1)}
    \]
    (note that $f_*\mathcal{F}^2=0$ since $\mathfrak{H}$ is 1-dimensional). Clearly,
    \[
        \psi^*\widetilde{\omega}_{\alg}^{(n)} \wedge (\widetilde{\nu} - G_2 \widetilde{\omega}^{(0)})\wedge \widetilde{\omega}^{(0)} = \psi^*\widetilde{\omega}_{\alg}^{(n)} \wedge \widetilde{\nu}\wedge \widetilde{\omega}^{(0)}.
    \]
    Since $\psi^*\widetilde{\omega}_{\alg}^{(n+1)}$ differs from $\widetilde{\omega}^{(n+1)}$ by an element of $\Gamma(\mathcal{E}_{\mathfrak{H}}^{\natural},\mathcal{O}_{\mathcal{E}_{\mathfrak{H}}^{\natural}})d\tau$ and $\widetilde{\omega}^{(0)} \wedge \widetilde{\omega}^{(n+1)}$ is a multiple of $d\tau$, we obtain
    \[
        n\frac{d\tau}{2\pi i} \wedge (\widetilde{\nu} - G_2 \widetilde{\omega}^{(0)}) \wedge \psi^*\widetilde{\omega}_{\alg}^{(n+1)} = n\frac{d\tau}{2\pi i} \wedge \widetilde{\nu}\wedge \psi^*\widetilde{\omega}_{\alg}^{(n+1)}.
    \]
    Thus, by Proposition \ref{prop:lift-kronecker-analytic}, we conclude that $\psi^*\widetilde{\omega}^{(n)}_{\alg} = \widetilde{\omega}^{(n)}$ for every $n\ge 1$.
\end{proof}

Let $P \in \mathcal{E}_{\mathfrak{H}}(\mathfrak{H})$ be a torsion section. Then, there are $\alpha,\beta \in \mathbb{Q}$ such that $P(\tau)$ is represented by $\alpha+\beta\tau$ under $\mathcal{E}_{\tau} \cong \mathbb{C}/(\mathbb{Z} + \tau \mathbb{Z})$ for every $\tau \in \mathfrak{H}$. The next result gives the canonical lift of $\omega^{(n)}_P$, defined in Corollary \ref{coro:omegaP-analytic}.

\begin{corollary}\label{coro:canonical-lift-omega-P}
    Set
    \[
         \widetilde{\omega}^{(n)}_{P} \defeq  \varphi^{(n)}_{\tau}(z - \alpha - \beta\tau, w-2\pi i \beta) \left(dz - w \frac{d\tau}{2\pi i} \right) + n\varphi^{(n+1)}_{\tau}(z-\alpha-\beta\tau,w-2\pi i \beta) \frac{d\tau}{2\pi i}.
    \]
    Let $T$ be an $S$-scheme and assume that the diagram \eqref{eq:pullback-weierstrass-EC} factors as
    \[
        \begin{tikzcd}
            \mathcal{E}_{\mathfrak{H}}\arrow{d} \arrow{r}{\psi_t}\arrow[bend left]{rr}{\psi}& E^{\an}_T \arrow{r}\arrow{d} & E^{\an}\arrow{d}\\
            \mathfrak{H} \arrow{r}[swap]{t}\arrow[bend right]{rr}[swap]{s} &T^{\an}\arrow{r}\arrow[phantom]{lu}{\square} & S^{\an}\arrow[phantom]{lu}{\square}
        \end{tikzcd}
    \]
    Let $P_{\alg} \in E_T(T)$ be a torsion section and $\widetilde{\omega}^{(n)}_{\alg, P_{\alg}}$ be the corresponding Kronecker differentials on $E_T$ with logarithmic singularities along $\pi^{-1}P_{\alg}$. If $P = t^*P_{\alg}$, then $\psi_t^*\widetilde{\omega}^{(n)}_{\alg, P_{\alg}}=\widetilde{\omega}^{(n)}_P$.
\end{corollary}

\begin{proof}
    The unique lift to a torsion section $P^{\natural} \in \mathcal{E}_{\mathfrak{H}}^{\natural}(\mathfrak{H})$ is such that $P^{\natural}(\tau)$ is represented by $(\alpha+\beta\tau,2\pi i \beta)$ under $\mathcal{E}^{\natural}_{\tau}\cong \mathbb{C}^2/L_{\tau}$ for every $\tau \in \mathfrak{H}$. One can check that that $\widetilde{\omega}^{(0)}$ and $\widetilde{\nu}$ are invariant under translation by $-P^{\natural}$ (cf. Lemma \ref{lem:invariance}), so that $\widetilde{\omega}^{(n)}_P = \tau^*_{-P^{\natural}}\widetilde{\omega}^{(n)}$. Then, the statement follows immediately from \eqref{eq:def-omegaP-lift} and from the previous theorem.
\end{proof}

\subsection{Level $N$ elliptic KZB connection}

Let $N\ge 1$ be an integer and denote by $\Gamma_N \defeq \mathcal{E}_{\mathfrak{H}}[N](\mathfrak{H})$ the group of $N$-torsion sections of $p:\mathcal{E}_{\mathfrak{H}} \to \mathfrak{H}$. Note that $\Gamma_N \cong (N^{-1}\mathbb{Z})^2/\mathbb{Z}^2$.

Consider the completed Hopf algebra over $\mathcal{O}_{\mathfrak{H}}$ given by
\[
    \mathcal{A}_N\defeq \frac{\mathcal{O}_{\mathfrak{H}}\langle \! \langle a,b,c_P : P \in \Gamma_N \rangle \! \rangle}{\langle \sum_{P \in \Gamma_N}c_P - [a,b]\rangle}.
\]
If $E$ is an (algebraic) complex elliptic curve such that $E^{\an} \cong \mathcal{E}_{\tau}$, then it follows from the discussion in \S\ref{par:kzb-connection} that the fibre of $\mathcal{A}_{N}$ at $\tau$ is isomorphic to the dual of the fundamental Hopf algebra of $E/\mathbb{C}$ punctured at $Z=E[N]$:
\[
    \mathcal{A}_{N,\tau} \cong \mathcal{H}_{E/\mathbb{C},E[N]}^{\vee}.
\]
Then, the relative elliptic KZB connection is given by
\[
    \overline{\nabla}_{N} : f^*\mathcal{A}_N \longrightarrow \Omega^1_{\mathcal{E}_{\mathfrak{H}}^{\natural}/\mathfrak{H}}(\log \pi^{-1}\mathcal{E}_{\mathfrak{H}}[N]) \hat{\otimes} f^*\mathcal{A}_N\text{, }\qquad \overline{\nabla}_{N} = d + \omega_{N},
\]
where
\[
    \omega_{N} = -\nu \otimes a -\omega^{(0)}\otimes b - \sum_{n\ge 1}\sum_{P \in \Gamma_N}\omega^{(n)}_P\otimes \ad_a^{n-1}c_P,
\]
with $\nu$, $\omega^{(0)}$, and $\omega_P^{(n)}$ as in the above paragraphs. If $P(\tau)$ is represented by $\alpha+\beta\tau$, let us denote
\[
    k_P(z,w,x) \defeq e^{(w-2\pi i \beta)x}F_{\tau}(z-\alpha-\beta\tau,x) - \frac{1}{x},
\]
the dependence on $\tau$ being omitted in the notation for simplicity. Then, it follows from the results of \S\ref{par:analytic-kronecker} that we can rewrite
\[
    \omega_{N} = -dw\otimes a -dz \otimes \left(b + \sum_{P \in \Gamma_N}k_P(z,w,\ad_a)c_P \right).
\]

According to \S\ref{par:universal-kzb}, the relative connection $\overline{\nabla}_N$ lifts to an absolute elliptic KZB connection, which we shall denote more simply as
\[
    \nabla_{N} : f^*\mathcal{A}_{N} \longrightarrow \Omega^1_{\mathcal{E}_{\mathfrak{H}}^{\natural}}(\log \pi^{-1}\mathcal{E}_{\mathfrak{H}}[N]) \hat{\otimes} f^*\mathcal{A}_N,
\]
and which is an integrable connection explicitly given by (cf. proof of Theorem \ref{thm:integrability})
\[
    \nabla_N = d + \widetilde{\omega}_N + \Phi_N,
\]
where $\widetilde{\omega}_N$ is the canonical lift of the KZB form $\omega_N$, and $\Phi_N$ is the matrix of the dual Gauss--Manin connection $\delta^{\vee}$ (cf. \S\ref{par:GM-connection-fundamental}). Next, we determine $\nabla_N$ explicitly.

\begin{lemma}
    For every $P \in \Gamma_N$, set
    \[
        g_P(z,w,x) \defeq \frac{\partial}{\partial x}k_P(z,w,x) - wk_P(z,w,x).
    \]
    We have
    \[
        \widetilde{\omega}_N = -dw\otimes a -dz\otimes \left(b + \sum_{P \in \Gamma_N}k_P(z,w,\ad_a)c_P\right) - \frac{d\tau}{2\pi i}\otimes \left(-wb + \sum_{P \in \Gamma_N}g_P(z,w,\ad_a)c_P\right).
    \]
\end{lemma}

\begin{proof}
    By definition,
    \[
        \widetilde{\omega}_N = -\widetilde{\nu} \otimes a -\widetilde{\omega}^{(0)}\otimes b - \sum_{n\ge 1}\sum_{P \in \Gamma_N}\widetilde{\omega}^{(n)}_P\otimes \ad_a^{n-1}c_P.
    \]
    We put together the explicit expressions for the canonical lifts: $\widetilde{\nu} = dw$, $\widetilde{\omega}^{(0)} = dz - w \frac{d\tau}{2\pi i}$, and, by Corollary \ref{coro:canonical-lift-omega-P},
    \begin{align*}
        \sum_{n\ge 1}\widetilde{\omega}^{(n)}_Px^{n-1} =& 
        \sum_{n\ge 1}  \varphi^{(n)}_{\tau}(z - \alpha - \beta\tau, w-2\pi i \beta)x^{n-1}dz \\
        &+ \sum_{n\ge 1}\left(n\varphi^{(n+1)}_{\tau}(z-\alpha-\beta\tau,w-2\pi i \beta)x^{n-1} -  w \varphi^{(n)}_{\tau}(z - \alpha - \beta\tau, w-2\pi i \beta)x^{n-1}\right)\frac{d\tau}{2\pi i}\\
        =& k_P(z,w,x)dz + g_P(z,w,x)\frac{d\tau}{2\pi i}.\qedhere
    \end{align*}
\end{proof}

For $Q \in \Gamma_N$, let us define $A_{m,Q}(\tau)$ by the generating series
\[
    g_{-Q}(0,0,x) = k'_{-Q}(0,0,x) = \sum_{m\ge 0}A_{m,Q}(\tau)x^m,
\]
where $k'_{-Q}(z,w,x) = \frac{\partial}{\partial x}k_{-Q}(z,w,x)$. When $Q = O$, we have
\[
    g_O(0,0,x) = \frac{\partial}{\partial x}\left(\frac{\partial}{\partial x}\log \theta_{\tau}(x)\right)  + \frac{1}{x^2} = -\left(\wp_{\tau}(x) -\frac{1}{x^2}\right) = \sum_{k\ge 2}(2k-1)G_{2k}(\tau)x^{2k-2},
\]
so that $A_{m,O}$ are level 1 Eisenstein series. For general $Q$, the functions $A_{m,Q}$ are Eisenstein series of level $N$; see \cite[Proposition 10.1]{hopper21}. 

\begin{theorem}\label{thm:computation-universal-KZB}
    We have
    \[
        \Phi_N = -\frac{d\tau}{2\pi i}\otimes \left(b\frac{\partial}{\partial a} + \frac{1}{2}\sum_{Q \in \Gamma_N}\sum_{m\ge 0}A_{m,Q}(\tau)\delta_{m,Q} \right)
    \]
    where
    \[
        \delta_{m,Q} \defeq \sum_{\substack{i,j\ge 0 \\ i+j = m-1}}\sum_{P \in \Gamma_N}[(-\ad_a)^ic_P,\ad_a^jc_{P-Q}]\frac{\partial}{\partial b} + \sum_{P \in \Gamma_N}[c_P,\ad_a^mc_{P-Q}+(-\ad_a)^mc_{P+Q}]\frac{\partial}{\partial c_P}.
    \]
\end{theorem}

The proof is based on the following lemma.

\begin{lemma}\label{lemma:computation-universal-KZB}
    Let $A_N = \mathbb{C}\langle \!\langle a,b,c_P : P \in \Gamma_N\rangle \! \rangle /\langle \sum_{P \in \Gamma_N}c_P - [a,b]\rangle$, so that $\mathcal{A}_N = \mathcal{O}_{\mathfrak{H}} \hat{\otimes}A_N$, and let $\Psi \in \Gamma(\mathfrak{H},\Omega^1_{\mathfrak{H}})\otimes \operatorname{Der}_{\mathbb{C}}(A_N)$. The following are equivalent:
    \begin{enumerate}[(i)]
        \item The connection $d + \widetilde{\omega}_N + \Psi$ is integrable,
        \item $d\widetilde{\omega}_N + \widetilde{\omega}_N\wedge \widetilde{\omega}_N + \Psi(\widetilde{\omega}_N) = 0$,
        \item $\Psi = \Phi_N$.
    \end{enumerate}
\end{lemma}

\begin{proof}
    The equivalence between (i) and (ii) follows immediately from Lemma \ref{lemma:integrability}. Since $\nabla_N = d + \widetilde{\omega}_N + \Phi_N$ is integrable, (ii) is equivalent to
    \begin{equation}\label{eq:derivation-KZB-form}
        \Psi(\widetilde{\omega}_N) = \Phi_N(\widetilde{\omega}_N).
    \end{equation}
    To finish, it is enough to verify that \eqref{eq:derivation-KZB-form} implies (iii). Let us write $\Phi_N = -\frac{d\tau}{2\pi i}\otimes \partial_N$ and $\Psi = -\frac{d\tau}{2\pi i}\otimes D$. Then, \eqref{eq:derivation-KZB-form} is equivalent to
    \begin{equation}\label{eq:derivation-KZB-form-equiv}
        \begin{aligned}
            &\frac{d\tau}{2\pi i}\wedge \widetilde{\nu}\otimes Da + \frac{d\tau}{2\pi i}\wedge \widetilde{\omega}^{(0)}\otimes Db + \sum_{n\ge 1}\sum_{P \in \Gamma_N}\frac{d\tau}{2\pi i}\wedge \widetilde{\omega}_P^{(n)}\otimes D \ad_a^{n-1}c_P\\
            =&\frac{d\tau}{2\pi i}\wedge \widetilde{\nu}\otimes \partial_Na + \frac{d\tau}{2\pi i}\wedge \widetilde{\omega}^{(0)}\otimes \partial_Nb + \sum_{n\ge 1}\sum_{P \in \Gamma_N}\frac{d\tau}{2\pi i}\wedge \widetilde{\omega}_P^{(n)}\otimes \partial_N \ad_a^{n-1}c_P
        \end{aligned}
    \end{equation}
    Since
    \[
        \frac{d\tau}{2\pi i}\wedge \widetilde{\nu}, \frac{d\tau}{2\pi i}\wedge \widetilde{\omega}^{(0)}, \frac{d\tau}{2\pi i}\wedge \widetilde{\omega}^{(n)}_P\text{, }\qquad P \in \Gamma_N, \qquad n \ge 1,
    \]
    trivialise $\mathcal{F}^{1,2} \cong f^*\Omega^1_{\mathfrak{H}}\otimes \Omega^1_{\mathcal{E}_{\mathfrak{H}}^{\natural}/\mathfrak{H}}(\log \pi^{-1}\mathcal{E}_{\mathfrak{H}}[N])$ (cf. \eqref{eqn:KoszulGraded}), it follows from \eqref{eq:derivation-KZB-form-equiv} that $Da = \partial_Na$, $Db=\partial_Nb$, and $Dc_P=\partial_Nc_P$ for all $P \in \Gamma_N$; thus $\Psi = \Phi_N$.
\end{proof}

\begin{proof}[Proof of Theorem \ref{thm:computation-universal-KZB}]
    The proof is a long computation; we merely indicate the main steps. Let
    \[
        \Psi \defeq -\frac{d\tau}{2\pi i}\otimes D\text{, }\qquad  D \defeq \left(b\frac{\partial}{\partial a} + \frac{1}{2}\sum_{Q \in \Gamma_N}\sum_{m\ge 0}A_{m,Q}(\tau)\delta_{m,Q} \right)
    \]
    By Lemma \ref{lemma:computation-universal-KZB}, it suffices to prove that
    \[
        d\widetilde{\omega}_N + \widetilde{\omega}_N\wedge \widetilde{\omega}_N + \Psi(\widetilde{\omega}_N) = 0.
    \]
  
    It follows from the `mixed heat equation' \cite[Proposition 5.(ii)]{BL11} that
    \[
        2\pi i \frac{\partial}{\partial \tau}k_P(z,w,x) = \frac{\partial }{\partial z}g_P(z,w,x),
    \]
    so that
    \begin{equation}\label{eq:computation-domega}
        \begin{aligned}
            d\widetilde{\omega}_N =& -dw\wedge dz \otimes \left(\sum_{P \in \Gamma_N}\frac{\partial }{\partial w}k_P(z,w,\ad_a)c_P \right)\\
            &- dw\wedge \frac{d\tau}{2\pi i}\otimes \left(-b + \sum_{P \in \Gamma_N}\frac{\partial }{\partial w}g_P(z,w,\ad_a)c_P \right)\\
            =& -dw\wedge dz \otimes \left( \ad_ab + \sum_{P \in \Gamma_N}\ad_ak_P(z,w,\ad_a)c_P \right)\\
            &- dw\wedge \frac{d\tau}{2\pi i}\otimes \left( -b - w\ad_ab + \sum_{P \in \Gamma_N}\ad_ag_P(z,w,\ad_a)c_P \right),
        \end{aligned}
    \end{equation}
    where in the last equality we used the equations $\frac{\partial }{\partial w }k_P(z,w,x) = xk_P(x,w,x)+1$, $\frac{\partial }{\partial x}g_P(z,w,x) = xg_P(z,w,x) -w$, and $\sum_{P \in \Gamma_N}c_P = \ad_ab$. Now, we have
    \begin{equation}\label{eq:computation-omegasquared}
        \begin{aligned}
            \widetilde{\omega}_N\wedge \widetilde{\omega}_N = & dw\wedge dz \otimes \left(\ad_ab + \sum_{P\in \Gamma_N} \ad_ak_P(z,w,\ad_a)c_P \right) \\
            &+ dw\wedge \frac{d\tau}{2\pi i} \otimes \left(-w\ad_ab + \sum_{P\in \Gamma_N} \ad_ag_P(z,w,\ad_a)c_P \right)\\
            &+dz\wedge \frac{d\tau}{2\pi i} \otimes \left[b + \sum_{P\in \Gamma_N} k_P(z,w,\ad_a)c_P,-wb + \sum_{Q\in \Gamma_N} g_Q(z,w,\ad_a)c_Q \right].
        \end{aligned}
    \end{equation}
    By putting \eqref{eq:computation-domega} and \eqref{eq:computation-omegasquared} together, we get
    \begin{equation}\label{eq:curvature-omega-tilde}
        \begin{aligned}
            d\widetilde{\omega}_N+ \widetilde{\omega}_N\wedge \widetilde{\omega}_N =& -\frac{d\tau}{2\pi i}\wedge dw \otimes b \\
            &-\frac{d\tau}{2\pi i}\wedge dz \otimes \left[b + \sum_{P\in \Gamma_N} k_P(z,w,\ad_a)c_P,-wb + \sum_{Q\in \Gamma_N} g_Q(z,w,\ad_a)c_Q \right]
        \end{aligned}
    \end{equation}

    We borrow the following notation from \cite{LR07}: for  $f(x,y) = \sum_{i,j\ge 0}f_{i,j}x^iy^j \in \mathbb{C}[\![x,y]\!]$ and $t,r,s \in A_N$, we set
    \[
        f(x,y)[\![r,s]\!]_t \defeq \sum_{i,j\ge 0}f_{i,j}[\ad^i_tr,\ad_t^js].
    \]
    Note that
    \begin{equation}\label{eq:useful-identity}
        f(x,y)[\![r,s]\!]_t = -f(y,x)[\![s,r]\!]_t.
    \end{equation}
    To finish the proof, we are left to show that the right-hand side of \eqref{eq:curvature-omega-tilde} is equal to $-\Psi(\widetilde{\omega}_N)$, or equivalently that
    \begin{equation}\label{eq:main-equation}
        \begin{aligned}
            D\left(b + \sum_{P\in \Gamma_N}k_P(z,w,\ad_a)c_P \right) &= \left[b + \sum_{P\in \Gamma_N} k_P(z,w,\ad_a)c_P,-wb + \sum_{Q\in \Gamma_N} g_Q(z,w,\ad_a)c_Q \right]\\
            & = \ad_b\sum_{P\in \Gamma_N}k_P'(z,w,\ad_a)c_P + \frac{1}{2}\sum_{P,Q \in \Gamma_N}f_{P,Q}(x,y)[\![c_Q,c_P]\!]_a,
        \end{aligned}
    \end{equation}
    where
    \[
        f_{P,Q}(x,y) = k_Q(z,w,x) k'_P(z,w,y) - k_P(z,w,y)k'_Q(z,w,x).
    \]

    We now compute the left-hand side of \eqref{eq:main-equation}. Using the symmetry $g_{-Q}(0,0,x) = g_Q(0,0,-x)$, we get by direct computation:
    \[
        D(b) = \frac{1}{2}\sum_{P,Q \in \Gamma_N}\frac{g_{P-Q}(0,0,y)-g_{Q-P}(0,0,x)}{x+y}[\![c_Q,c_P]\!]_a.
    \]
    Applications of \cite[Lemma 3.1.4]{LR07} and \eqref{eq:useful-identity} show that
    \[
        D\left(\sum_{P\in \Gamma_N}k_P(z,w,\ad_a)c_P \right) = \ad_b\sum_{P \in \Gamma_N}k'_P(z,w,\ad_a)c_P + \frac{1}{2}\sum_{P,Q \in \Gamma_N}h_{P,Q}(x,y)[\![c_Q,c_P]\!]_a,
    \]    
    where
    \begin{align*}
        h_{P,Q}(x,y) =& \frac{k_P(z,w,x+y) - k_P(z,w,y) - xk_P'(z,w,y)}{x^2}\\
        &- \frac{k_Q(z,w,x+y) - k_Q(z,w,x) - yk_Q'(z,w,x)}{y^2}\\
        &+k_Q(z,w,x+y)g_{P-Q}(0,0,y) - k_P(z,w,x+y)g_{Q-P}(0,0,x).
    \end{align*}
    Thus, to establish \eqref{eq:main-equation}, it suffices to prove that, for every $P,Q \in \Gamma_N$, we have
    \begin{align*}
        f_{P,Q}(x,y) - h_{P,Q}(x,y) - \frac{g_{P-Q}(0,0,y)-g_{Q-P}(0,0,x)}{x+y} = 0,
    \end{align*}
    which is equivalent to
    \begin{align*}
        &\left(k'_{P-Q}(0,0,y) - \frac{1}{y^2} \right)\left(k_Q(z,w,x+y) + \frac{1}{x+y} \right) - \left(k'_{Q-P}(0,0,x) - \frac{1}{x^2} \right)\left(k_P(z,w,x+y)+ \frac{1}{x+y} \right) \\
        &+\left(k'_Q(z,w,x)-\frac{1}{x^2} \right)\left(k_P(z,w,y) + \frac{1}{y} \right) -\left(k'_P(z,w,y)-\frac{1}{y^2} \right)\left(k_Q(z,w,x) + \frac{1}{x} \right) = 0.
    \end{align*}
    This, in turn, follows immediately from the formula \cite[Equation (3.3)]{BL11} in the case $P=Q$, and from the Fay identity \cite[Proposition 5.(iii)]{BL11} in the case $P\neq Q$.
\end{proof}

\appendix

\section{Unipotent connections and Tannakian theory}\label{appendix:tannakian}

We fix a base field $k$ of characteristic zero, which will be implicit throughout this appendix.

\subsection{Unipotent connections}\label{app:unipotent}

Let $X$ be a smooth $k$-scheme and $D$ be a normal crossings divisor in $X$. Recall that a \emph{vector bundle with integrable connection over $X$ with logarithmic singularities along $D$} is a pair $(\mathcal{E},\nabla)$, where $\mathcal{E}$ is a vector bundle over $X$ and $\nabla : \mathcal{E} \to \Omega^1_{X/k}(\log D) \otimes_{\mathcal{O}_X} \mathcal{E}$ is an integrable logarithmic $k$-connection on $\mathcal{E}$. A morphism $(\mathcal{E},\nabla) \to (\mathcal{E}',\nabla')$ is a horizontal $\mathcal{O}_X$-linear map $f:\mathcal{E} \to \mathcal{E}'$, meaning that  $\nabla' \circ f = (\id \otimes f)\circ \nabla$. We thus obtain a category which we denote by $\VIC(X,\log D)$. When $D$ is empty, it is denoted by $\VIC(X)$.

\begin{example}
    If $\mathcal{E} =\mathcal{O}_X\otimes_kF$ for some $k$-vector space $F$, then we can write $\nabla = d + \omega$ for a unique $\omega \in \Gamma(X,\Omega^1_{X/k}(\log D)) \otimes_k \End_k(F)$. Integrability amounts to the equation $d\omega + \omega \wedge \omega =0$ in $\Gamma(X,\Omega^2_{X/k}(\log D)) \otimes_k \End_k(F)$. In general, a \emph{frame} of a connection $(\mathcal{E},\nabla)$ is an isomorphism $(\mathcal{O}_X^{\oplus n},d + \omega) \cong (\mathcal{E},\nabla)$, and $\omega \in M_{n\times n}(\Gamma(X,\Omega^1_{X/k}(\log D)))$ is the \emph{matrix of $\nabla$} in this frame. 
\end{example}

Recall that the category $\VIC(X,\log D)$ is $k$-linear and admits the usual multilinear operations, such as duals and tensor products. A \emph{horizontal section} of $(\mathcal{E},\nabla)$ is some $s \in \Gamma(X,\mathcal{E})$ satisfying $\nabla s = 0$; it can also be thought as a morphism $(\mathcal{O}_X,d) \to (\mathcal{E},\nabla)$. In particular, a morphism $(\mathcal{E},\nabla) \to (\mathcal{E}',\nabla')$ is the same as a horizontal section of $(\mathcal{E},\nabla)^{\vee}\otimes (\mathcal{E}',\nabla')$.

\begin{definition}
    We say that an object $(\mathcal{E},\nabla)$ of $\VIC(X,\log D)$ is \emph{unipotent} if it admits a finite filtration
    \[
        0 = (\mathcal{E}_0,\nabla_0) \subseteq (\mathcal{E}_1,\nabla_1) \subseteq \cdots \subseteq (\mathcal{E}_n,\nabla_n) = (\mathcal{E},\nabla)
    \]
    such that, for every $1\le i \le n$, the quotient $(\mathcal{E}_i,\nabla_i)/(\mathcal{E}_{i-1},\nabla_{i-1})$ is isomorphic to $(\mathcal{O}_X \otimes F_i, d \otimes \id)$ for some $k$-vector space $F_i$. The full subcategory of $\VIC(X,\log D)$ given by unipotent objects is denoted by $\UVIC(X,\log D)$ (or $\UVIC(X)$, when $D$ is empty). The smallest $n$ for which such a filtration exists is the \emph{index of unipotency of $(\mathcal{E},\nabla)$}.
\end{definition}

\subsection{Local form}

In this paragraph, we give a local characterisation of unipotent connections.  

\begin{lemma}\label{lemma:subcomplex}
    Let $A$ be a commutative $k$-algebra and $\Omega^{\bullet}\hookrightarrow \Omega_{A/k}^{\bullet}$ be a subcomplex of $k$-modules such that
    \begin{equation}\tag{$\mathrm{C}_n$}
        \begin{split}
            H^n(\Omega^{\bullet}) \stackrel{\sim}{\rightarrow} H^{n}(\Omega^{\bullet}_{A/k}) \text{ is an isomorphism and } H^{n+1}(\Omega^{\bullet}) \hookrightarrow H^{n+1}(\Omega^{\bullet}_{A/k}) \text{ is injective}
        \end{split}
    \end{equation}
    for some $n\ge 0$. Given $\omega \in \Omega^n_{A/k}$, the following are equivalent:
    \begin{enumerate}
        \item $d\omega \in \Omega^{n+1}$.
        \item There exists $\nu \in \Omega_{A/k}^{n-1}$ such that $\omega + d\nu \in \Omega^n$.
    \end{enumerate}
\end{lemma}

In practice, we only use this lemma when $\Omega^{\bullet}\hookrightarrow \Omega^{\bullet}_{A/k}$ is a quasi-isomorphism, so that the conditions ($\mathrm{C}_n$) are satisfied for every $n\ge 0$.

\begin{proof}
    Only the direction $(1) \Rightarrow (2)$ is non-trivial. The form $d\omega \in \Omega^{n+1}$ is closed and defines a cohomology class in $H^{n+1}(\Omega^{\bullet})$. As $H^{n+1}(\Omega^{\bullet}) \hookrightarrow H^{n+1}(\Omega^{\bullet}_{A/k})$ is injective, and $d\omega$ is exact in $\Omega^{\bullet}_{A/k}$, it must also be exact in $\Omega^{\bullet}$. Thus, there exists $\eta \in \Omega^n$ such that $d\eta = d\omega$. Since $\omega-\eta \in \Omega^n_{A/k}$ is closed, it defines a cohomology class in $H^n(\Omega^{\bullet}_{A/k})$. Finally, from the isomorphism $H^n(\Omega^{\bullet}) \stackrel{\sim}{\to} H^{n}(\Omega^{\bullet}_{A/k})$, we obtain $\nu \in \Omega^{n-1}_{A/k}$ such that $\omega - \eta + d\nu \in \Omega^n$; as $\eta \in \Omega^n$, we conclude that $\omega + d\nu \in \Omega^n$.
\end{proof}

\begin{theorem}\label{thm:local-form}
    Let $X=\Spec A$ be a smooth affine $k$-scheme, $(\mathcal{E},\nabla)$ be an object of $\UVIC(X)$, and  $\Omega^{\bullet}\hookrightarrow \Omega^{\bullet}_{A/k}$ be a subcomplex of $k$-modules satisfying condition $(\mathrm{C}_1)$. Then $\mathcal{E}$ is trivial and there exists a frame
    \[
        (\mathcal{O}_X^{\oplus n},d + \omega) \cong (\mathcal{E},\nabla)
    \]
    in which the matrix $\omega$ is strictly upper triangular and has all of its entries in $\Omega^1$.
\end{theorem}

\begin{proof}
    We proceed by induction on the rank $n$ of $\mathcal{E}$. The base case $n=1$ is trivial, since $(\mathcal{E},\nabla)$ must be isomorphic to $(\mathcal{O}_X,d)$. Assume that the statement holds in rank $\le n-1$. By unipotency of $(\mathcal{E},\nabla)$, there is an exact sequence
    \[
        \begin{tikzcd}
            0 \rar & (\mathcal{E}',\nabla') \rar & (\mathcal{E},\nabla) \rar & (\mathcal{O}_X,d) \rar & 0,
        \end{tikzcd}
    \]
    where $(\mathcal{E}',\nabla')$ is an object of $\UVIC(X)$, with $\mathcal{E}'$ of rank $n-1$. By induction hypothesis, $\mathcal{E}'$ is trivial and admits a frame $e':(\mathcal{O}_{X}^{\oplus n-1},d + \omega') \stackrel{\sim}{\to} (\mathcal{E}',\nabla|_{\mathcal{E}'})$ in which  $\omega'$ is strictly upper-triangular and has all of its entries in $\Omega^1$.

    As $X$ is affine, there is a splitting
    \[
        \begin{tikzcd}
            0 \rar & \mathcal{E}' \rar & \mathcal{E} \rar & \mathcal{O}_X \rar \arrow[bend right]{l}[above]{e_n}  & 0    ,  
        \end{tikzcd}
    \]
    so that $\mathcal{E}$ is trivial and admits the frame $(e',e_n) : (\mathcal{O}_{X}^{\oplus n}, d+ \omega) \stackrel{\sim}{\to} (\mathcal{E},\nabla)$, in which 
    \[
        \omega = \left(\begin{array}{ccc|c}
                     &  & & \omega_{1\, n}\\
                     & \omega'_{ij} & & \vdots \\
                     &  & & \omega_{n-1\, n}\\\hline
                     0& \cdots  & 0 & 0
                   \end{array}\right)
    \]
    where $\omega_{ij} = \omega_{ij}' \in \Omega^1$ whenever $j<n$.
    
    To finish, we explain how to inductively modify $e_n$ so that $\omega_{in} \in \Omega^1$ for all $1\le i \le n$. Note that $\omega_{nn} = 0 \in \Omega^1$. By descending induction in $i$, assume that $\omega_{kn} \in \Omega^1$ for all $i+1\le k \le n$. The integrability equation $d\omega + \omega \wedge \omega = 0$ at the entry $(i,n)$ means that
    \[
        d\omega_{in} + \sum_{k=i+1}^{n-1}\omega_{ik}\wedge \omega_{kn}=0.
    \]
    Since, for every $i+1\le k \le n-1$, both $\omega_{ik}$ and $\omega_{kn}$ belong to $\Omega^1$, we have $\sum_{k=i+1}^{n-1}\omega_{ik}\wedge \omega_{kn} \in \Omega^2$. By Lemma \ref{lemma:subcomplex}, there exists $g \in A$ such that $\omega_{in} + dg \in \Omega^1$. Thus
    \[
        \nabla(e_n + ge_i') = \sum_{k=1}^{i-1}(\omega_{kn} + g\omega_{ki})\otimes e'_k  + (\omega_{in} + dg)\otimes e'_{i}  + \sum_{k=i+1}^{n-1}\omega_{kn} \otimes e_k' 
    \]
    and we conclude that the matrix $\widetilde{\omega}$ of $\nabla$ in the frame $(e', e_n + ge'_i)$ satisfies $\widetilde{\omega}_{kn} \in \Omega^1$ for all $i\le k \le n$.
\end{proof}

\subsection{Canonical extension}

Let $X$ be a smooth $k$-scheme and $D$ be a normal crossings divisor in $X$. We say that an object $(\mathcal{E},\nabla)$ of $\VIC(X,\log D)$ (resp. $\VIC(X\setminus D)$) is \emph{locally unipotent along $D$} if, for every $x \in D$, there exists an open neighbourhood $V$ of $x$ such that the restriction $(\mathcal{E},\nabla)|_V$ (resp. $(\mathcal{E},\nabla)|_{V\setminus D}$) is unipotent. For simplicity, let us denote the corresponding full subcategories of $\VIC(X,\log D)$ and $\VIC(X\setminus D)$ by $L_1$ and $L_2$.

\begin{theorem}[cf. {\cite[Proposition 5.2]{deligne70}}]\label{thm:canonical-extension}
    If $j: X\setminus D \to X$ denotes the inclusion, then the restriction functor
    \[
        j^*: L_1 \longrightarrow L_2\text{, }\qquad (\mathcal{E},\nabla)\longmapsto (\mathcal{E},\nabla)|_{X\setminus D}
    \]
    is an equivalence of tensor categories.
\end{theorem}

\begin{proof}
    Since
    \[
        \Hom((\mathcal{E}',\nabla'),(\mathcal{E},\nabla)) \cong \Hom((\mathcal{O}_X,d), (\mathcal{E}',\nabla')^{\vee} \otimes (\mathcal{E},\nabla))
    \]
    and since $\Hom((\mathcal{O}_X,d), (\mathcal{E},\nabla))$ is canonically isomorphic to $\Gamma(X,\mathcal{E}^{\nabla})$, where $\mathcal{E}^{\nabla}$ denotes the subsheaf of horizontal sections of $\mathcal{E}$, to show that $j^* : L_1 \to L_2$ is fully faithful, it is enough to prove that
    \begin{equation}\label{eq:restriction-horizontal}
        \Gamma(X, \mathcal{E}^{\nabla}) \longrightarrow \Gamma(X\setminus D, \mathcal{E}^{\nabla})\text{, }\qquad s \longmapsto s|_{X\setminus D}
    \end{equation}
    is bijective for every $(\mathcal{E},\nabla)$ in $L_1$.
    
    As $\mathcal{E}$ is locally free, the injectivity of \eqref{eq:restriction-horizontal} follows from the fact that $D$ is locally defined by a torsion-free section of $\mathcal{O}_X$. Granted the injectivity, we can prove the surjectivity of \eqref{eq:restriction-horizontal} locally. We can thus assume that $(\mathcal{E},\nabla)$ admits a frame $e: (\mathcal{O}_X^{\oplus n}, d+\omega) \stackrel{\sim}{\to} (\mathcal{E},\nabla)$ in which $\omega$ is strictly upper triangular (Theorem \ref{thm:local-form}). We must prove that a horizontal section $s \in \Gamma(X\setminus D, \mathcal{E}^{\nabla})$ extends to $X$. By writing $s = \sum_{j=1}^ns_j \otimes e_j$, with $s_j \in \Gamma(X\setminus D,\mathcal{O}_X)$, we get
    \[
         0 = \nabla(s) = \sum_{i=1}^n (ds_i + \sum_{j=i+1}^ns_j \omega_{ij}) \otimes e_i.
    \]
    Since $ds_n=0$, $s_n$ extends to $X$. By descending induction on $i$, it follows from the above equation that $ds_i$ has at most logarithmic singularities along $D$, so that $s_i$ extends to $X$ for every $1\le i \le n$. This finishes the proof that \eqref{eq:restriction-horizontal} is bijective.
    
    We are left to prove that $j^*:L_1 \to L_2$ is essentially surjective, we use that it is fully faithful to reduce it to a local statement. We can thus assume that $X$ is affine and that $(\mathcal{E},\nabla)$ is in $\UVIC(X\setminus D)$. Since we are in characteristic zero, the injection $\Omega^{\bullet}_{X/k}(\log D) \hookrightarrow j_*\Omega^{\bullet}_{(X\setminus D)/k}$ is a quasi-isomorphism (see \cite[Corollaire 3.14, Remarque 3.16]{deligne70}), so that we can apply Theorem \ref{thm:local-form} to find a frame $e: (\mathcal{O}_{X\setminus D}^{\oplus n}, d+\omega) \stackrel{\sim}{\to}(\mathcal{E},\nabla)$ in which $\omega_{ij} \in \Gamma(X,\Omega^1_{X/k}(\log D))$ for every $1\le i,j\le n$. Then, $(\mathcal{O}_X^{\oplus n}, d + \omega)$ is an extension of $(\mathcal{E},\nabla)$.
\end{proof}

Given an object $(\mathcal{E},\nabla)$ of $L_2$, the unique object $(\overline{\mathcal{E}},\overline{\nabla})$ of $L_1$ such that $(\overline{\mathcal{E}},\overline{\nabla})|_{X\setminus D} = (\mathcal{E},\nabla)$ is called the \emph{canonical extension} of $(\mathcal{E},\nabla)$. This yields a quasi-inverse to the restriction $j^*: L_1 \to L_2$, which can be checked to be additive, exact, and tensor.

\begin{corollary}\label{coro:canonical-ext}
    With the above notation, the restriction functor
    \begin{equation}\label{eq:restriction-uvic}
        \UVIC(X,\log D) \longrightarrow \UVIC(X\setminus D)\text{, }\qquad (\mathcal{E},\nabla)\longmapsto (\mathcal{E},\nabla)|_{X\setminus D}
    \end{equation}
    is an equivalence of tensor categories.
\end{corollary}

\begin{proof}
    Since $\UVIC(X,\log D)$ (resp. $\UVIC(X\setminus D)$) is a full subcategory of $L_1$ (resp. $L_2$), it follows immediately from Theorem \ref{thm:canonical-extension} that \eqref{eq:restriction-uvic} is fully faithful. To see that it is essentially surjective, we must check that, for any object $(\mathcal{E},\nabla)$ of $\UVIC(X\setminus D)$, its canonical extension $(\overline{\mathcal{E}},\overline{\nabla})$ is in $\UVIC(X,\log D)$. If $(\mathcal{E},\nabla) = (\mathcal{O}_{X\setminus D}\otimes F, d \otimes \id)$ for some $k$-vector space $F$, then it follows from the unicity of the canonical extension that $(\overline{\mathcal{E}},\overline{\nabla}) = (\mathcal{O}_X\otimes F,d \otimes \id)$. Since the canonical extension functor is exact, the general case follows by induction on the index of unipotency of $(\mathcal{E},\nabla)$.
\end{proof}

\subsection{$\mathbb{A}^1$-invariance}\label{par:A1-invariance}

By an \emph{affine bundle} we mean a morphism of $k$-schemes $Y \to X$ which is, locally over $X$, of the form $\mathbb{A}_U^m \to U$.

\begin{theorem}\label{thm:A1-invariance}
    If $X$ is a smooth $k$-scheme and $\pi:Y \to X$ is an affine bundle, then the pullback functor
    \[
        \pi^* : \UVIC(X) \longrightarrow \UVIC(Y),\qquad (\mathcal{E},\nabla) \longmapsto (\pi^*\mathcal{E},\pi^*\nabla),
    \]
    is an equivalence of tensor categories.
\end{theorem}

\begin{proof}
    By the same initial argument of the proof of Theorem \ref{thm:canonical-extension}, to verify that $\pi^*$ is fully faithful, it is enough to check that, for every object $(\mathcal{E},\nabla)$ of $\UVIC(X)$, the pullback map on horizontal sections
    \begin{equation}\label{eq:pullback-horizontal}
        \Gamma(X,\mathcal{E}^{\nabla}) \longrightarrow \Gamma(Y, (\pi^*\mathcal{E})^{\pi^*\nabla})
    \end{equation}
    is bijective.
    
    The injectivity of \eqref{eq:pullback-horizontal} immediately follows from the faithful flatness of $\pi$. As in Theorem \ref{thm:canonical-extension}, granted the injectivity, we can reduce the proof of surjectivity to a local statement. Thus, we can assume that $X= \Spec B$ is affine, and that $Y = \mathbb{A}^m_X$. By induction on $m$, we can further assume that $m=1$, so that $Y = \Spec A$, with $A=B[t]$. Set $M = \Gamma(X,\mathcal{E})$; the map \eqref{eq:pullback-horizontal} then becomes the inclusion of $k$-vector spaces
    \[
        M^{\nabla} \to M[t]^{\pi^*\nabla}.
    \]
    An element of $M[t]$ is of the form $q = \sum_{n\ge 0}x_nt^n$ for some $x_n \in M$. If $q$ is horizontal for $\pi^*\nabla$, then
    \[
        0 = \pi^*\nabla(q) = \sum_{n\ge 0}t^n\nabla(x_n) + \sum_{n\ge 1}dt\otimes nt^{n-1}x_n ,
    \]
    and we must have $\sum_{n\ge 1}dt\otimes nt^{n-1}x_n = 0$. Since $k$ is of characteristic zero, we get $x_n=0$ for every $n\ge 1$. Thus, $q=x_0$ is in the image of \eqref{eq:pullback-horizontal}.
    
    We first prove that $\pi^*:\UVIC(X) \to \UVIC(Y)$ is essentially surjective locally on $X$. Let $(\mathcal{E},\nabla)$ be an object of $\UVIC(Y)$. We use the notation from the last paragraph: $X=\Spec B$ and $Y=\Spec A$, with $A=B[t]$. Since $\Omega^{\bullet}_{B/k} \hookrightarrow \Omega^{\bullet}_{A/k}$ is a quasi-isomorphism, we can apply Theorem \ref{thm:local-form} to find a frame $(\mathcal{O}_Y^{\oplus n},d+\omega) \stackrel{\sim}{\to} (\mathcal{E},\nabla)$ in which $\omega$ has all of its entries in $\Omega^1_{B/k}$. Thus, $\pi^*(\mathcal{O}_X^{\oplus n}, d + \omega) \cong (\mathcal{E},\nabla)$.
    
    In general, let $(\mathcal{E},\nabla)$ be an object of $\UVIC(Y)$. Since we already know that $\pi^*$ is fully faithful, the above local constructions glue, yielding a (locally unipotent) vector bundle with integrable $k$-connection $(\mathcal{E}',\nabla')$ on $X$ satisfying $\pi^*(\mathcal{E}',\nabla') \cong (\mathcal{E},\nabla)$. We are left to check that $(\mathcal{E}',\nabla')$ is unipotent. If $(\mathcal{E},\nabla) = (\mathcal{O}_X\otimes F,d \otimes \id)$ for some $k$-vector space $F$, then $(\mathcal{E}',\nabla') \cong (\mathcal{O}_Y\otimes F, d\otimes \id)$ by the fully faithfulness of $\pi^*$. Since $\pi$ is faithfully flat, the pullback $\pi^*$ is an exact functor, so that the general case follows by induction on the index of relative unipotency of $(\mathcal{E},\nabla)$.
\end{proof}

The above statement also admits a logarithmic version. We keep the above notation and let $D$ be a normal crossings divisor in $X$.

\begin{theorem}\label{thm:pullback-equiv}
    With the above notation, the pullback functor
    \[
        \pi^*: \UVIC(X, \log D) \longrightarrow \UVIC(Y, \log \pi^{-1}D)\text{, }\qquad (\mathcal{E},\nabla) \longmapsto (\pi^*\mathcal{E},\pi^*\nabla)
    \]
    is an equivalence of tensor categories.
\end{theorem}

\begin{proof}
    One can prove it directly, as in the proof of Theorem \ref{thm:A1-invariance}, or derive it as a corollary Theorems \ref{thm:canonical-extension} and \ref{thm:A1-invariance}. Indeed, let $j:X\setminus D \to X$ be the inclusion. Then, the diagram of pullback functors
    \[
        \begin{tikzcd}
            \UVIC(X,\log D) \arrow{d}[swap]{\pi^*} \arrow{r}{j^*} & \UVIC(X\setminus D) \arrow{d}\\
            \UVIC(Y,\log \pi^{-1}D) \arrow{r} & \UVIC(Y \setminus \pi^{-1}D)
        \end{tikzcd}  
    \]
    commutes, so that $\pi^*$ is fully faithful and essentially surjective because all other arrows are.
\end{proof}

\subsection{De Rham fundamental group of a punctured elliptic curve}\label{par:de-Rham-group}

Recall that, if $X$ is a smooth and geometrically connected $k$-scheme, $\UVIC(X)$ is a neutral Tannakian category over $k$ (cf. \cite[\S 10.26]{deligne89}). Given a fibre functor $b : \UVIC(X) \to \Vect_k$, the \emph{unipotent de Rham fundamental group} of $X$ at $b$ is the Tannakian fundamental group
\[
    \pi_1^{\dR}(X,b) \defeq \underline{\Aut}_{\UVIC(X)}^{\otimes}(b).
\]
It is a pro-unipotent affine group scheme over $k$. 

Let $E$ be an elliptic curve over $k$, $Z\subset E$ be a divisor as in \S\ref{par:kronecker-subbundle}, and $\pi: E^{\natural} \to E$ be the canonical projection from the universal vector extension. The following two results can be attributed to Deligne (cf. \cite{EE18}).

\begin{lemma}\label{lemma:deligne}
    If $\mathcal{V}$ is a unipotent vector bundle over $E^{\natural}$, then the natural map $\Gamma(E^{\natural},\mathcal{V}) \otimes_k \mathcal{O}_{E^{\natural}} \to \mathcal{V}$ is an isomorphism.
\end{lemma}

\begin{proof}
    This follows, as in \cite[Proposition 12.3]{deligne89}, by an inductive argument on the rank of $\mathcal{V}$, using that $H^0(E^{\natural},\mathcal{O}_{E^{\natural}}) = k$ and $H^1(E^{\natural},\mathcal{O}_{E^{\natural}}) = \Ext^1(\mathcal{O}_{E^{\natural}},\mathcal{O}_{E^{\natural}}) = 0$ (Theorem \ref{thm:coleman-laumon}).
\end{proof}

\begin{proposition}\label{prop:fibre-functor}
    The functor
    \[
        b_{\can} : \UVIC(E\setminus Z) \longrightarrow \Vect_k\text{, }\qquad (\mathcal{E},\nabla) \longmapsto \Gamma(E^{\natural},\pi^*\overline{\mathcal{E}})  
    \]
    is a fibre functor over $k$.
\end{proposition}

\begin{proof}
    Let $\operatorname{UV}(E^{\natural})$ be the category of unipotent vector bundles on $E^{\natural}$. The functor $b_{\can}$ is the composition
    \[
        \UVIC(E\setminus Z) \longrightarrow  \UVIC(E,\log Z) \stackrel{\pi^*}{\longrightarrow} \UVIC(E^{\natural},\log \pi^{-1}Z) \longrightarrow \operatorname{UV}(E^{\natural}) \stackrel{\Gamma(E^{\natural}, -)}{\longrightarrow} \Vect_k,
    \]
    where the first arrow is the canonical extension and the third arrow is the forgetful functor $(\mathcal{V},\nabla) \mapsto \mathcal{V}$. By Corollary \ref{coro:canonical-ext} and Theorem \ref{thm:pullback-equiv}, the first two arrows are $k$-linear equivalences of tensor categories. The third is trivially a $k$-linear tensor faithful functor. Finally, the last arrow is $k$-linear, tensor, and faithful by Lemma \ref{lemma:deligne}.
\end{proof}

Our next goal is to relate the fundamental group $\pi_1^{\dR}(E\setminus Z,b_{\can})$ with the Hopf algebra $\mathcal{H}_{E/k,Z}$ constructed in \S\ref{subsec:fundamental-Hopf}. Let $(\mathcal{E},\nabla)$ be an object of $\UVIC(E\setminus Z)$ and write $V \defeq b_{\can}(\mathcal{E},\nabla)$. It follows from Lemma \ref{lemma:deligne} that
\[
    (\pi^*\overline{\mathcal{E}},\pi^*\overline{\nabla}) \cong (\mathcal{O}_{E^{\natural}}\otimes V, d + \omega)
\]
for a unique nilpotent (in the sense of Proposition \ref{prop:comodule-bar}) $\omega \in \Gamma(E^{\natural},\Omega^1_{E^{\natural}/k}(\log \pi^{-1}Z)) \otimes \End(V)$ satisfying $d \omega + \omega\wedge \omega = 0$. Thus, it defines a $\mathcal{H}_{E/k,Z}$-comodule structure $\rho = \sum_{n\ge 0}[\omega]^n$ on $V$. These constructions are natural, so that we obtain a functor
\begin{equation}\label{eq:functor-uvic-comod}
    \UVIC(E\setminus Z) \longrightarrow \operatorname{Comod}(\mathcal{H}_{E/k,Z})\text{, }\qquad (\mathcal{E},\nabla) \longmapsto (V,\rho)
\end{equation}
extending $b_{\can}$.

\begin{theorem}\label{thm:pi1-derham}
    The functor \eqref{eq:functor-uvic-comod} is an equivalence of tensor categories over $k$. In particular, it induces an isomorphism of affine group schemes over $k$:
    \[
        \pi_1^{\dR}(E\setminus Z,b_{\can}) \cong \Spec \mathcal{H}_{E/k,Z}
    \]
\end{theorem}

\begin{proof}
    That \eqref{eq:functor-uvic-comod} is a $k$-linear equivalence of categories is an immediate consequence of Corollary \ref{coro:canonical-ext}, Theorem \ref{thm:pullback-equiv}, and Proposition \ref{prop:comodule-bar}. We are left to show that $\eqref{eq:functor-uvic-comod}$ is a tensor functor.

    We already know that $b_{\can}$ is tensor by Proposition \ref{prop:fibre-functor}. Now, the tensor structure on the category $\operatorname{Comod}(\mathcal{H}_{E/k,Z})$ is induced by the shuffle product: given comodules $(V,\rho), (V,\rho')$, the tensor comodule structure $\rho\shuffle \rho'$ on $V\otimes V'$ is given by
    \[
        V\otimes V'\stackrel{\rho \otimes \rho'}{\longrightarrow} (\mathcal{H}_{E/k,Z}\otimes V)\otimes (\mathcal{H}_{E/k,Z}\otimes V') \cong (\mathcal{H}_{E/k,Z}\otimes \mathcal{H}_{E/k,Z})\otimes (V\otimes V') \stackrel{\shuffle \otimes \id}{\longrightarrow} \mathcal{H}_{E/k,Z}\otimes (V\otimes V'),
    \]
    where all of the above tensor products are over $k$. By Proposition \ref{prop:comodule-bar}, if $\rho = \sum_{i\ge 0}[\omega]^i$ and $\rho'= \sum_{j\ge 0}[\omega']^j$, then
    \begin{align*}
        \rho\shuffle \rho' = \sum_{i,j\ge 0}[\omega]^i\shuffle [\omega']^j  = \sum_{n\ge 0} [\omega \otimes \id + \id \otimes \omega']^n.
    \end{align*}
    To finish, we simply remark that the tensor structure of $\UVIC(E^{\natural},\log \pi^{-1}Z)$ is given by
    \[
        (\mathcal{O}_{E^{\natural}}\otimes V, d + \omega) \otimes (\mathcal{O}_{E^{\natural}}\otimes V', d + \omega') \cong (\mathcal{O}_{E^{\natural}}\otimes V\otimes V', d+ \omega\otimes \id + \id \otimes \omega').\qedhere
    \]
\end{proof}

\begin{corollary}
    There is a canonical isomorphism $\pi_1^{\dR}(E\setminus Z,b_{\can}) \cong \Spec T^cH^1_{\dR}((E\setminus Z)/k)$.
\end{corollary}

\begin{proof}
    This is an immediate consequence of Theorem \ref{thm:pi1-derham} and Theorem \ref{thm:projector}.
\end{proof}

Note that $\mathcal{H}_{E/k,Z}^{\vee}$ is a limit of finite dimensional $k$-vector spaces, and the pullback $f^*\mathcal{H}_{E/k,Z}^{\vee}$ is simply the base change $\mathcal{O}_{E^{\natural}}\hat{\otimes} \mathcal{H}_{E/k,Z}^{\vee}$. Let
\[
    \nabla_{E^{\natural}/k,Z} : \mathcal{O}_{E^{\natural}}\hat{\otimes} \mathcal{H}_{E/k,Z}^{\vee}  \longrightarrow \Omega^1_{E^{\natural}/k}(\log \pi^{-1}Z) \hat{\otimes} \mathcal{H}^{\vee}_{E/k,Z}\text{, }\qquad \nabla_{E^{\natural}/k,Z} = d + \omega_{E^{\natural}/k,Z}
\]
be the elliptic KZB connection of $E/k$ punctured at $Z$ constructed in \S \ref{par:kzb-connection}. It is a pro-object of $\UVIC(E^{\natural},\log \pi^{-1}Z)$. By Theorem \ref{thm:pullback-equiv}, it corresponds to a pro-object $(\mathcal{V}_{\KZB},\nabla_{\KZB})$ of $\UVIC(E,\log Z)$. Note that $\Gamma(E^{\natural},\pi^*\mathcal{V}_{\KZB})$ is the complete Hopf algebra $\mathcal{H}_{E/k,Z}^{\vee}$, and we denote by $1 \in \mathcal{H}_{E/k,Z}^{\vee}$ its unit. 

\begin{proposition}
    The pro-vector bundle with logarithmic connection $(\mathcal{V}_{\KZB},\nabla_{\KZB})$ (resp. its restriction $(\mathcal{V}_{\KZB},\nabla_{\KZB})|_{E\setminus Z}$)  satisfies the following universal property: for every triple $(\mathcal{V},\nabla, v)$, where $(\mathcal{V},\nabla)$ is an object of $\UVIC(E,\log Z)$ (resp. $\UVIC(E\setminus Z)$), and $v \in \Gamma(E^{\natural}, \pi^*\mathcal{V})$ (resp. $v \in \Gamma(E^{\natural}, \pi^*\overline{\mathcal{V}})$), there is a unique horizontal map
    \[
        \varphi : (\mathcal{V}_{\KZB},\nabla_{\KZB}) \longrightarrow    (\mathcal{V},\nabla)\qquad \text{(resp. }\varphi : (\mathcal{V}_{\KZB},\nabla_{\KZB})|_{E\setminus Z} \longrightarrow (\mathcal{V},\nabla)\text{)}
    \]
    satisfying $\varphi(1) = v$.
\end{proposition}

\begin{proof}
    It suffices to combine Proposition \ref{prop:univ-property-kzb} with the equivalence of Theorem \ref{thm:pi1-derham}.
\end{proof}

\end{document}